\documentclass[11pt,letterpaper]{amsart}
\usepackage[margin=1.25in]{geometry}
\usepackage{enumitem, comment, amssymb}

\input{mathcommand.sty}
\usepackage{esint}
\usepackage{cancel}

\theoremstyle{plain}% default
\newtheorem{theorem}{Theorem}[section]
\newtheorem{definition}[theorem]{Definition}

\newtheorem{lemma}[theorem]{Lemma}
\newtheorem{corollary}[theorem]{Corollary}

\newtheorem*{theorem*}{Theorem}
\newtheorem*{corollary*}{Corollary}
\newtheorem{remark}[theorem]{Remark}
\usepackage{xcolor}
\usepackage{hyperref}
\newcommand{\eqplus}[1]{\hyperref[#1]{(\ref{#1}$+$)}}
\newcommand{\eqminus}[1]{\hyperref[#1]{(\ref{#1}$-$)}}
\newcommand{\R}{\mathbb{R}}
\newcommand{\N}{\mathbb{N}}
\newcommand{\E}{\mathcal{E}}

\DeclareMathOperator{\Graph}{Graph}

\newcommand{\e}{\varepsilon}

\allowdisplaybreaks

\title[Asymptotic Free Boundary Solutions]{On the asymptotic properties of solutions to one-phase free boundary problems}

\thanks{M.E. was partially supported by NSF DMS CAREER 2143719 and by a grant from the Simons Foundation (Grant Award ID BD-Targeted-00017375-ME). D.R. was partically supported by the AMS-Simons travel grant. Z.Z. was partially supported by the NSF grant DMS-2350351.}

\author{Max Engelstein}
\address{127 Vincent Hall 206 Church St. SE Minneapolis, MN 55455, United States of America}
\email{mengelst@umn.edu}

\author{Daniel Restrepo}
\address{27 Vincent Hall 206 Church St. SE Minneapolis, MN 55455, United States of America}
\email{drestrep@umn.edu}

\author{Zihui Zhao}
\address{Department of Mathematics, Johns Hopkins University, 3400 N. Charles Street, Baltimore, MD 21218, United States of America}
\email{zhaozh@jhu.edu}

\begin{document}

\begin{abstract}
{\rm In this article we study the structure of solutions to the one-phase Bernoulli problem that are modeled either infinitesimally or at infinity by one-homogeneous solutions with an isolated singularity. In particular, we prove a uniqueness of blow-ups result with a H\"older decay rate under a natural integrability assumption on the one-homogeneous solution (\`a la Allard-Almgren \cite{AllardAlmgren}) and we prove a rigidity result at infinity (\`a la Simon-Solomon \cite{SimonSolomon})  under additional spectral assumptions, which we verify for the axially symmetric cones. We believe these are the first uniqueness of blow-up/blow-down results at singular points for non-minimizing solutions to the one-phase problem. }
\end{abstract}

\maketitle

\setcounter{tocdepth}{1}

\tableofcontents
\section{Introduction}\label{s:intro}
In this paper we prove both a uniqueness of blow-ups result and a rigidity of blow-downs result for solutions to the one-phase Bernoulli problem that are modeled on cones with isolated singularities and additional symmetries. Recall that the one-phase Bernoulli problem is a prototypical free boundary problem, given by the overdetermined equation
\begin{equation}\label{e:bernoulliweak}
	\left\{ \begin{array}{ll}
		\Delta u = 0, & \text{ in } \{u > 0\} \\
    |\nabla u| = 1, & \text{ on } \partial \{u > 0\}. 
	\end{array}
    \right.
\end{equation}
First studied rigorously by Alt and Caffarelli in \cite{AltCaf}, this problem has received a huge amount of attention over the decades due, among other reasons, to its relationship with global solutions of semi-linear PDEs \cite{CafSalsa}, \cite{AuSe}, its non-convex nature \cite{Generic}, and its connections to eigenvalue optimization and the behavior of harmonic measures \cite{KenigToro,KL,MTV}. 

It was first shown by Weiss \cite{Weiss} that solutions to \eqref{e:bernoulliweak} (suitably interpreted, see Definition \ref{d:weaksol}) are modeled both infinitesimally and at infinity by one-homogeneous solutions to the same equation. More precisely, if $x_0 \in \partial \{u > 0\}$ and $r_j \downarrow 0$ (or $r_j \uparrow \infty$) then up to a subsequence $u(r_jx+x_0)/r_j \rightarrow U$, where $U$ is a one-homogeneous solution to \eqref{e:bernoulliweak}. {\it A priori}, one may get different limits $U$ along different sequences $\{r_j\}$. Less pathologically, $U$ can be independent of the sequence, but the convergence can be very slow (see \cite{White}, \cite{Hirsch} for examples of non-unique limits or \cite{CSV, FigalliSerra} for examples of slowly converging limits in related nonlinear PDEs). 

If $U = (x_d)_+$, then the above convergence is well understood; the limit does not depend on the sequence $r_j$ and the convergence happens at a power rate. However, there exist non-flat homogeneous solutions to \eqref{e:bernoulliweak} and when $U$ is one of these solutions, the convergence is much less well understood and essentially nothing is known if $u$ is not assumed to be a minimizer of the Alt-Caffarelli functional (see more discussion after the theorem statements). 

Our two main results completely describe the convergence when $\partial \{U > 0\}$ is smooth away from the origin and has some additional symmetry, which we call integrability through rotations, after the analogous property for minimal surfaces (see Definition \ref{def: intcone}).  The first result says that as long as $U$ satisfies these conditions, the limit is independent of the sequence $\{r_j\}$ and occurs with a H\"older rate. 

\begin{theorem}\label{thm AA}
Let $D\subset \R^d$ be an open set with $B_8 \subset D$ and let $u:D\subset\R^d\to\R_+$ be a non-degenerate weak solution of the one-phase Bernoulli problem, in the sense of Definition \ref{d:weaksol}. Let $U$ be a one-homogeneous solution which is integrable through rotations (see Definition \ref{def: intcone}) and such that $\partial\{U>0\}\setminus \{0\}$ is smooth. There exist $\delta_0>0$, $\alpha\in(0,1)$, and $C>0$, depending only on $U$, such that the following statements hold.

\begin{itemize}
    \item Suppose that $0\in\partial\{u>0\}$ and there exists a sequence $r_j\downarrow0$ such that $u(r_jx)/r_j\rightarrow U$ in $L^2_{\mathrm{loc}}$. If $\Vert u-U\Vert _{L^2(B_8\setminus B_{1/8})}\leq\delta_0,$
then
\begin{equation}\label{eq:rate0}\Vert u(rx)/r - U(x)\Vert _{L^2(B_2\setminus B_\frac{1}{2})} \leq Cr^\alpha {\|u - U\|_{L^2(B_8 \setminus B_{1/8} ) }}, \end{equation}
for all $0<r<1$
\item Suppose that $D=\R^d$ and there exists a sequence $r_j\uparrow\infty$ such that $u(r_jx)/r_j\rightarrow U$ in $L^2_{\mathrm{loc}}$. If $\|u-U\|_{L^2(B_8\setminus B_{1/8})}\leq\delta_0,$ then
\begin{equation}\label{eq:rate} \|u(rx)/r - U(x)\|_{L^2(B_{2}\setminus B_\frac{1}{2})} \leq Cr^{-\alpha} \|u - U \|_{L^2(B_8 \setminus B_{1/8} ) }, \end{equation} for all $ r > 1$.
\end{itemize}
\end{theorem}

Theorem \ref{thm AA} follows from the more general Theorem \ref{thm AArestated}, whose more technical statement and proof are deferred to Section \ref{sec: AAresult}. This ``$L^2$-excess decay" result (as they are sometimes called in the geometric analysis literature) does not assume {\it a priori} that $U$ occurs as a blow-up or blow-down of $u$; instead, closeness to $U$ at one scale and an appropriate one-sided condition on the Weiss energy yield convergence to a rotation of $U$. Its main ingredient is the symmetric growth-decay estimate of Theorem \ref{thm 3annulil}, also presented in Section \ref{sec: AAresult}, which treats small and large scales in a unified framework. Neither result requires $u$ or $U$ to be minimizing or stable. While concluding uniqueness of blow-ups from symmetric growth-decay estimates is common in geometric analysis, such estimates are to our knowledge new for non-minimizing solutions to the Bernoulli problem. See Remark \ref{r:3annuli} for more discussion.

%We first mention that there is a version of Theorem \ref{thm AA} when $r_j \uparrow \infty$ (see Theorem \ref{thm AArestated} for the full statement). NOTE SOMEWHERE THAT THE CONE DOES NOT HAVE TO BE MINIMAL OR STABLE

 Theorem \ref{thm AA} is reminiscent of an analogous result in minimal surfaces due to Allard and Almgren, \cite{AllardAlmgren}. At a high level the general strategy of the proof is inspired by \cite{AllardAlmgren} (and other regularity results for minimal surfaces such as \cite{Simon, Simonbook, MaggiNovack, EdelenMinter}); however, the free boundary setting requires many new ideas and brings up challenges that are not present when analyzing minimal surfaces. We will discuss this in more detail in Section \ref{ss: comparison}.

 As mentioned above, we believe that Theorem \ref{thm AA} is the first known uniqueness of blow-up result for singular points in the free boundary of solutions of \eqref{e:bernoulliweak}. However, there have been uniqueness results for the free boundaries of \emph{minimizers} of the Alt-Caffarelli functional, 
 \begin{equation}\label{eq:functional}
     J_{D}(u) := \int_{D} |\nabla u|^2\, dx + |\{u > 0\}\cap D|.
 \end{equation}

 Minimizers of the above functional satisfy \eqref{e:bernoulliweak} in a distributional sense, but the class of solutions covered by Theorem \ref{thm AA} is broader. In particular, Kamburov and Wang \cite{KW} proved that stable classical solutions are non-degenerate; consequently, Theorem \ref{thm AA} also applies to this class. We refer to Section~\ref{sec: conv} for a precise discussion of our notion of weak solution and its relation to other classes of non-minimizing solutions, including the variational solutions studied by Kriventsov and Weiss \cite{KWeiss}. In \cite{EnSVEpi}, the first named author, along with Spolaor and Velichkov, proved a version of Theorem \ref{thm AA} for minimizing $u$, and also treated the non-integrable case (see also \cite{EdSVLogEpi} for the symmetric growth-decay result for minimizers). The key tool in \cite{EnSVEpi} is an ``epiperimetric inequality" which quantifies the behavior of the energy $J$ near one-homogeneous solutions. However, epiperimetric inequalities can only imply regularity results like Theorem \ref{thm AA} for minimizing solutions.  As such, by necessity the methods of this paper are almost entirely disjoint from those of \cite{EnSVEpi}. Furthermore the main results of these papers do not imply one another: the authors of \cite{EnSVEpi} are able to treat non-integrable cones which we do not discuss here, and for integrable cones our Theorem \ref{thm AA} is more general.

Regarding the assumption of integrability in Theorem \ref{thm AA},  it is natural to wonder which known cones satisfy the hypothesis and to ask if it is a necessary assumption. Until recently, the known regular non-flat cones were largely confined to the axially symmetric solutions and to more general cohomogeneity-one constructions of isoparametric type, for example, \cite{Hong}. We prove here that the axially symmetric cones are integrable through rotations.  Two recent works provided new families of homogeneous solutions with isolated singularities. Hines, Kolesar, and McGrath \cite{HinesKolesarMcGrath} construct infinitely many examples in dimensions $3$ and $4$ whose links on the sphere have only finite symmetry groups. Firester, Tsiamis, and Wang  \cite{FiresterTsiamisWang1, FiresterTsiamisWang2} systematically study the $O(d-k)\times O(k)$ family of one-phase cones, analogous to the Lawson cones in the minimal hypersurfaces setting. They prove that these cones are stable for $d\geq 7$ and establish minimality for a large subfamily. 
 In personal communication, Firester and Tsiamis have told us that they expect these latter families of cones to be integrable through rotations. Determining whether the examples of \cite{HinesKolesarMcGrath} are integrable appears more difficult, since their discrete symmetry groups does not provide a natural coordinate system in which separation of variables arguments can be used to compute the spectrum of the linearized operator at the cone.

For the second question, to show the assumption of integrability is necessary, %To show this is sharp
we would have to produce an example of a singular solution which has a non-unique non-integrable blow-up (or slow convergence to that blow-up). Unfortunately, there are currently no known examples of singular solutions to the Bernoulli problem which are not cones themselves, so we can only discuss the sharpness of the hypothesis in analogy with other nonlinear PDE.  In the setting of minimal surfaces, the assumption that $U$ is integrable is not needed to prove the uniqueness of the blow-up by the fundamental work of \cite{Simon}; however, without this assumption the rate of convergence can be slower than $r^\epsilon$ for any $\epsilon > 0$ (see the example in \cite{AdamsSimon}). Indeed, in \cite{EnSVEpi}, a very closely related hypothesis of integrability was needed in order to prove $r^\alpha$-convergence.  We should point out that the distinction between power convergence and slow convergence is important in several contexts. Indeed, to prove our second main result, Theorem \ref{thm main}, we need the power rate of convergence guaranteed by Theorem \ref{thm AA} to justify, loosely speaking, that the decay rate of solutions to this nonlinear problem is dictated by the possible decay rates of solutions to the linearized equation at $U$. %Fortunately, we are able to show that axially symmetric cones are integrable through rotations (these include the only known minimzing cones in every dimension). 
%{\color{red} Recently, two groups of authors constructed families of homogeneous solutions with isolated singularities: in \cite{HinesKolesarMcGrath} it was proved the existence of infinitely many solutions in dimensions 3 and 4; in \cite{FiresterTsiamisWang1, FiresterTsiamisWang2} the authors constructed new families of cones, partially investigated in \cite{Hong}, corresponding to the family of Lawson cones in the minimal hypersurfaces setting. The authors show the minimality of several of these examples for $7 \leq d \leq 100$.}\\% In future work we hope to show that a large class of symmetric cones discovered by HONG CITE are also integrable through rotations. 

%When the singularity of $\partial \{U > 0\}$ is not isolated we expect this to be a substantially more difficult problem. In the setting of minimal surfaces uniqueness of the blowup was shown for a large class of cones cross $\mathbb R$ by Simon \cite{SimonCylinders} and for the cone over $\mathbb S^3\times \mathbb S^3$ cross $\mathbb R$ by Sz\'ekelyhidi \cite{Gabor}. The main technical issues are the lack of smooth parameterization away from the origin and the fact that even if the cone $\mathbf C$ is integrable, the cone $\mathbf C\times \mathbb R$ may not be. 

Our second result gives a characterization of global solutions of \eqref{e:bernoulliweak} which are modeled at infinity by one-homogeneous solutions with an isolated singularity at the origin and additional symmetries. To state the theorem we recall a recent result of De Silva-Jerison-Shahgholian, which guarantees the existence of a smooth foliation around one-homogeneous minimizers with isolated singularities:

   \begin{theorem}[Theorem 1 of \cite{DSJS}]\label{thm foliation}
Let $U$ be a homogeneous global minimizer of the Alt-Caffarelli functional with $\partial \{U > 0\}\backslash \{0\}$  smooth. There exists a pair $\overline{U},\underline{U}$ of global minimizers with analytic free boundary satisfying
\begin{enumerate}
    \item 
    $\underline{U}\le U\le\overline{U}$ and 
    $\operatorname{dist}\!\bigl(\partial\{\overline{U}>0\},\{0\}\bigr)
      =\operatorname{dist}\!\bigl(\partial \{\underline{U} > 0\},\{0\}\bigr)=1$;

    \item 
     the graphs of $\overline{U}_{t} = \frac{\overline{U}(t\cdot)}{t}$ and $\underline{U}_{t} = \frac{\underline{U}(t\cdot)}{t}$  for $t > 0$, together with $U$, form a foliation of the half-space
    $\mathbb{R}^{d}\times[0,\infty)$.
\end{enumerate}
Moreover, the foliation $\{\overline{U}_t, \underline{U}_t\}_{\{t>0\}}$ is unique.
\end{theorem}

We can now state our rigidity result:

\begin{theorem}\label{thm main}
   Let $u:\R^d\to\R_+$ be a non-degenerate weak solution to the one-phase Bernoulli problem in the sense of Definition \ref{d:weaksol}. Suppose that there exists a sequence $R_j\uparrow\infty$ such that $u(R_jx)/R_j\rightarrow U$ uniformly on compacta. Assume that $U$ is a one-homogeneous global minimizer of the Alt--Caffarelli functional, that $\partial\{U>0\}\backslash{0}$ is smooth, and that $U$ is strongly integrable (see Definition \ref{def: sintegrable}). 
    
  Then $u$ coincides either with a translation of $U$ or with one of the leaves $\overline{U}_t$ or $\underline{U}_t$ for some $t>0$.

   In particular, the above result holds when $U$ is a minimizing axially symmetric cone.
\end{theorem}

We recall that in \cite[Theorem 1.1-(vi)]{DSJS}, the authors showed that a global minimizer \emph{lying on one side of} of a one-homogeneous global minimizer $U$ with an isolated singularity must coincide with $U$ or with a leaf of the associated foliation. They further observed that the same conclusion holds for uniformly non-degenerate global critical points. When the minimizing cone $U$ is strongly integrable, Theorem \ref{thm main} removes the a priori assumption of one-sidedness. We discuss below whether the assumption of strong integrability in Theorem \ref{thm main} is necessary.

For a long time, the axially symmetric cones, commonly referred to as the De Silva--Jerison cones, were the only known examples of non-flat minimizing one-homogeneous solutions to \eqref{e:bernoulliweak}; see \cite{DSJ}. We show in Section \ref{ss:DSJ} that they are indeed strongly integrable.  As mentioned above, after our paper was posted, a new family of one-homogeneous solutions has been studied and proved to be minimizing in \cite{FiresterTsiamisWang1, FiresterTsiamisWang2}. Preliminary computations suggest that these cones are strongly integrable but we expect that there are other isoparametric cones (i.e. cones where $|\nabla_\theta U|$ is constant on the level sets of $U/r$) which are not strongly integrable even if they are minimizing (in high enough dimension).

%(we observe that the examples constructed in \cite{HinesKolesarMcGrath} are likely to have a complicated spectrum, but these examples are not minimizing by the Jerison and Savin classification).}

%It is natural to ask if there are any one-homogeneous solutions which are (or are not) strongly integrable. Unfortunately, very few minimizing cones for the one-phase problem have been identified. In each $d\geq 7$, the De Silva-Jerison cone is the only known example of non-flat minimizing one-homogeneous solutions to \eqref{e:bernoulliweak}, see \cite{DSJ}. We remind the reader that it is an open question as to whether there are any minimizing one-homogeneous solutions in $\R^d$ with $d \in \{5,6\}$ and that all one-homogeneous minimizers are smooth when $d \leq 4$, \cite{CJK, JerisonSavin}.  
%The upshot is that while \cite{DSJS} is {\it a priori} more general, there are no known examples which are treated by \cite{DSJS} and not by Theorem \ref{thm main} (though of course our proof of Theorem \ref{thm main} relies on \cite[Theorem 1.1]{DSJS}). 

Asymptotic rigidity or Liouville theorems play an important role in minimal surface theory and other nonlinear geometric PDEs (see the introduction of  \cite{EdGabor} for a nice discussion). Unfortunately, there are very few such theorems available in the context of the Bernoulli problem. Aside from our Theorem \ref{thm main}, and the aforementioned works \cite{DSJS, EdSVLogEpi} (see also \cite{EdSVMaxPrinc}), we are only aware of relevant results on planar domains, see \cite{KLT, T, JK}. We also mention recent developments on the Bernstein problem in the free boundary setting, see \cite{Bernstein, GraphBernstein}.
%we are aware of no other rigidity results for solutions to the Bernoulli problem that are asymptotic to or lay on one side of one-homogeneous solutions. 

Theorem \ref{thm main} is inspired by a corresponding result due to Simon-Solomon \cite{SimonSolomon} (see also \cite{EdelenSpolaor}), which characterizes complete minimal surfaces asymptotic to minimizing quadratic cones  (quadratic cones are cones over the product of two lower-dimensional spheres). The reason why \cite{SimonSolomon} restrict to quadratic cones is that a separation of variables argument shows that any eigenfunction of the linearized minimal surface equation around a quadratic cone can be written as a sum of products of spherical harmonics. From here it follows that the bottom of the spectrum is generated by geometric rigidities (the foliation, translations and rotations). This motivates the notion of \emph{strong integrability} for homogeneous solutions in free boundary problems (see Definition \ref{def: sintegrable}, adapted after \cite{EdelenSpolaor}). %For free boundary problems, we summarize the sufficient conditions of the linear spectrum in the notion of strong integrability in Definition \ref{def: sintegrable}.
% We mention that the more recent work of Edelen-Spolaor \cite{EdelenSpolaor} %provides a local version of \cite{SimonSolomon}  
% %pointed out that Solomon-Simon did not use that the cones were quadratic so much as they used certain aspects of the spectrum of the minimal surface equation linearized around quadratic cones. This led them to introduce the notion of strongly-integrability for minimal surfaces, and we have introduce the analogous notion for the free boundary problem in Definition \ref{def: sintegrable}. We should mention that \cite{EdelenSpolaor} 
% generalizes \cite{SimonSolomon}, providing a local version of sorts for a convergent sequence of minimal surfaces (not necessarily a blowup sequence), and that we hope to prove an analogous result for free boundary problems in the future. 
At a high level our proof of Theorem \ref{thm main} follows the strategy laid out by \cite{SimonSolomon} but the execution of this strategy must be different in the free boundary setting, which we highlight in Section \ref{ss: comparison}.
%. Again we will discuss these differences in more detail below.  

As before, it is natural to ask whether some assumption on the bottom of the linearized spectrum is necessary in the minimal surface setting in Simon-Solomon (resp. in the free boundary setting in Theorem \ref{thm main}). 
%In the context of minimal surfaces, c
Complete minimal surfaces asymptotic to minimizing non-quadratic cones have been constructed in \cite{WhiteCones, Chan}, see also the discussion in \cite[Section 4]{SimonSolomon}. In particular, \cite{Chan} showed that complete minimal surfaces asymptotic to a given minimizing cone, $C$ should be at least as rich as the Jacobi fields corresponding to the bottom of the spectrum of the linearized operator.  As a consequence, the rigidity result of Simon-Solomon holds only if the bottom of the spectrum is generated by geometric rigidities (and the foliation). No construction analogous to Chan's \cite{Chan} exists in the Bernoulli setting but it is natural to expect that each slowly growing solution to the linearized equation around $U$ should generate a new global solution asymptotic to $U$. This would be an interesting direction for further work that, if it can be shown, would demonstrate the necessity of the strong integrability assumption in Theorem \ref{thm main}.

\subsection{Sketch of the proofs and Comparison with minimal surfaces}\label{ss: comparison}
Let us sketch the proof of Theorem \ref{thm AA}; our main tool is the monotone quantity of Weiss' boundary adjusted energy \cite{Weiss}. In the key Lemma \ref{lemma cone approx}, we show that on an annulus the change of this monotone quantity controls the square of the $L^2$-distance between $u$ and a rotation of the integrable homogeneous solution, $U$. On the other hand, we also show that the difference of Weiss' energies for $u$ and $U$ at the same scale $r$ is controlled by the $C^2$ difference between them on an annulus containing $\partial B_r$,
%Finally, we need to show that the drop the monotone quantity from scale $r$ to $\infty$ (or to $0$) is controlled by the difference between our solution and the blowup in an annulus at scale $r$, 
see Lemma \ref{l:simontrick}. % (this is an analogue of an estimate due to L. Simon \cite{Simon}). 
Putting all this together allows us to show a power rate of decay in the monotone quantity, thus proving Theorem \ref{thm AA}.

Experts in minimal surface theory will find many of the ingredients in this high level outline familiar. However, we would like to emphasize that justifying the linearization of \eqref{e:bernoulliweak}, which is key to Lemma \ref{lemma cone approx}, is non-trivial. That is to say, we want to show that if $u$ converges to $U$ (in $L^2$) then $u-U$ (suitably renormalized) converges to a solution of the homogeneous Robin boundary value problem, \eqref{eq:lineareq}. In the setting of minimal surfaces, justifying the analogous statement (once  $\epsilon$-regularity is established) follows by writing one minimal surface as a graph over the other and directly applying regularity theory for elliptic PDEs.  

This process is more complicated for the Bernoulli problem in two ways. Firstly, the nonlinear problem \eqref{e:bernoulliweak} has two equations, a bulk equation on the positivity set and a boundary condition on the free boundary (this issue has been known to cause serious problems in the stability analysis, see for example \cite{JerisonSavin}); and secondly, $u$ and $U$ solve these equations on different domains. A naive attempt to map $\{u > 0\}$ onto $\{U > 0\}$ will introduce errors that ruin the perturbative analysis (see \cite{FSV} where linearizing the one-phase problem at branch points runs into a philosophically similar issue with a very different solution). %\textcolor{red}{Max: Should we omit the following?} We mention that similar difficulties also arise when trying to prove the analogue of L. Simon's estimate in Lemma \ref{l:simontrick}.

As far as we are aware, %for the Bernoulli problem 
the only other place where this linearization with appropriate error terms is justified is in \cite{DSJS}. In that paper, which has inspired us greatly, the situation is simplified by the fact that the perturbation is signed (so $u$ always lies to one side of $U$). 
See also the very recent, interesting preprint, \cite{Generic2}, which extends this one-sided analysis. We also mention \cite{EnSVEpi} where a related linearization is done. However, the  energy-theoretic approach in \cite{EnSVEpi} allows them to only have to linearize the ``one-homogeneous" part of the energy (i.e. the functional $\mathcal F$ defined in the proof of Lemma \ref{l:simontrick}). This reduction mostly circumvents the two difficulties described above. This difference in approach, i.e.  linearizing the entire equation versus the one-homogeneous part of the energy, is also the reason for the slight discrepancy between the notions of integrability (our Definition \ref{def: intcone} and \cite[Definition 2.4]{EnSVEpi}). See Remark \ref{r:intdiff} for more discussion. %As mentioned above, energy based arguments do not seem easily adapted to the study of general critical points (as opposed to minimizers). 

Thus the main technical contribution of our paper is to rigorously justify the linearization around a general one-homogeneous $U$ with smooth free boundary away from the origin, without assuming one-sidedness. This is done in Section \ref{sec: conv}. Without burdening the reader with too many details, there are two main ideas: (i) we use Cauchy-Kovalevskaya to extend $u, U$ analytically so that they are both defined and harmonic on the same domain (see Lemma \ref{lemma: extension}); (ii) we derive a Moser-type estimate by exploiting that $u-U$ satisfies the homogeneous equation up to an error which is quadratic, i.e. \eqref{eq:linaroundu}, which allows us to control the smooth norms of $u-U$ by the $L^2$-norm of $u-U$ - this is Lemma \ref{lemmaL2Linfty}. We believe this general approach will be useful in other situations where one wants to analyze solutions of \eqref{e:bernoulliweak} near a homogeneous solution with an isolated singularity.

%The study of solutions to the one-phase problem parallels in many senses the study of minimal surfaces. In particular, homogeneous global solutions play a pivotal role in the study of fine regularity properties of solutions, essentially solutions are modeled after conical solutions at a microscopic level. 

%\begin{itemize}
  %  \item Mention that %in our case we have two equations rather than 1. Also that the  domains of definition of solutions differ and for that reason we need to carry out the extension. 
%\end{itemize}

Let us turn to the proof of Theorem \ref{thm main}. From the above discussion, we know that $u-U$ (suitably extended) is a harmonic function satisfying inhomogeneous Robin boundary values - see \eqref{eq:linaroundu}. Using the fact that the inhomogeneity is quadratic in $u-U$, we construct particular solutions, $v_p$, to \eqref{eq:linaroundu} (see Lemma \ref{lemma part sol}) and use linear analysis to conclude that the order of vanishing of $u-U$ at infinity is dictated by the behavior of solutions to the homogeneous equation \eqref{eq:lineareq}. We then use the characterization of the eigenvectors at the bottom of the spectrum (which is the essence of the definition of strong integrability) in addition to a maximum principle argument, to get the desired result. This is done in Section \ref{s:smooth}. 

This tracks very closely the approach of \cite{SimonSolomon}, but, as in the proof of Theorem \ref{thm AA}, the Bernoulli problem poses additional difficulties.\footnote{Of course there are some arguments, such as detecting possible translations, which are actually simpler for the Bernoulli problem.} The major new difficulties in this setting are firstly, characterizing the bottom of the spectrum of the linearization around the De Silva-Jerison cone; and secondly, constructing a particular solution for general strictly stable regular cones. The first issue, in the case of quadratic cones for minimal surfaces, comes from the fact that the links of these cones are products of lower-dimensional spheres, and thus their spectrum is fully classified: all the eigenvalues are explicit and eigenfunctions are products of spherical harmonics.
% (since the links of these cones are products of spheres, the eigenvalues and eigenvectors are explicit). 
 Our spectral analysis is more involved and takes place in Section \ref{ss:DSJ}. In particular, since our cones are axially symmetric (in other words $\mathbb S^{d-2}$-symmetric), we may use the separation of variables to split the spectral problem into two problems, one on $\mathbb S^{d-2}$ with no boundary condition, and one ODE with a Robin boundary condition in the orthogonal direction.  While our analysis in this paper is specifically for the De Silva-Jerison cone, we believe, as mentioned above, that the tools and techniques should allow us to compute the bottom of the spectrum of a large number of symmetric one-homogeneous solutions to the Bernoulli problem. 

With regards to the particular solution, in \cite{SimonSolomon} they construct the desired particular solution by first expanding the non-linearity, slice by slice, in Jacobi fields and then, using the expansion of the Laplacian in spherical coordinates, finding the correct radially varying coefficients via integration. In our setting, \eqref{eq:linaroundu} has the inhomogeneous boundary condition, %which does not have any radial derivatives, 
thus precluding us from using the same trick. Instead, we introduce an additional spectral basis (the boundary Robin basis, see \eqref{eq bdry spectrum}) to transfer the non-linearity from the boundary to the bulk. Once the non-linearity appears in the bulk, we can treat it similarly to \cite{SimonSolomon}.

%In\cite{SimonSolomon}, Simon and Solomon showed a macroscopic counterpart, proving a rigidity result for global minimal boundaries asymptotic to area minimizing quadratic cones. More precisely, Simon and Solomon showed that the surface in question must be either the cone itself or a leaf of the foliation constructed in \cite{HardtSimon}. The analogous foliation for global homogeneous minimizers of the Alt-Caffarelli functional was constructed in \cite{DSJS}. So far the only homogeneous solution known to be minimizing for the Alt-Caffarelli functional is the axially symmetric solution $U$ investigated in \cite{DSJ}.

%Our analysis relies on the spectral properties of the linearized operator. First of all, {\color{red} [ explain how the weak notion of being asymptotic can be improved to uniqueness of blowdown with a rate of convergence. Once the rate of convergence is obtained, using spectral analysis we characterize the possible profiles at infinity of solutions asymptotic to the cone which leads to the classsification provided in Theorem \ref{thm main}]}. 

\subsection{Organization of the paper.} 
In Section \ref{sec:linear}, we discuss the linear operator of %of the operator obtained by linearizing 
the one-phase problem around regular cones (not necessarily stable), presenting a detailed analysis of some of its properties. In particular, we analyze the spectrum associated with the De Silva-Jerison cone and show that it belongs to the class of strongly integrable cones (see Definition \ref{def: sintegrable}) to which Theorems \ref{thm AA} and \ref{thm main} apply. Section \ref{sec: conv} mainly establishes some quantitative estimates for the differences between two nearby solutions, by studying the one-phase problem as a quantitative perturbation of the linear operator. % that will allow us to linearize the one-phase problem. 
Section \ref{sec: AAresult} is devoted to the proof of Theorem \ref{thm AA}, while in Section \ref{s:smooth} we prove Theorem \ref{thm main}. Finally, for ease of reference, the \nameref{sec: appendix} collects some useful inequalities including linear estimates for boundary value problems with Robin boundary conditions (oblique derivatives). \\

\section{Some properties of the linearized operator}\label{sec:linear}

We begin this section by specifying the notions of classical solutions and regular cones.
\begin{definition}\label{def regular sol}
    Given an open set $D \subset \R^d$, we say that $w: D\to \R_+$ is a classical solution to the one-phase free boundary problem in $D$, if $\partial\{w>0\}\cap D$ is a $C^2$ hypersurface, $w \in C^1\big(\overline{\{w>0\}}\cap D\big)$, and $w$ solves \eqref{e:bernoulliweak} in $D$.
\end{definition}

Let us explicitly mention that our definition of classical solution rules out ``two-phase" solutions like $|x_d|$, which satisfy \eqref{e:bernoulliweak} in some sense but is not a classical solution because it is  not $C^1$ in the \emph{closure} of its positivity set. 

\begin{definition}\label{def regular cone}
    If $U:\R^d \to \R^+$ is a $1$-homogeneous function, we say that $U$ is a regular cone if it is a classical solution to the one-phase free boundary problem in $\R^d \setminus \overline{B_r(0)}$ for any $r>0$.
\end{definition}

At the risk of redundancy, we note that under Definition \ref{def regular cone}, $U(x) = |x \cdot e|$ is not a regular cone for any $e\in \mathbb{S}^{d-1}$.\\

Let $U$ be a regular cone in the sense of Definition \ref{def regular cone}. This implies that the positivity set $\Omega = \{U>0\}$ is a non-trivial cone with smooth boundary in $\R^d \setminus \{0\}$. Let $\nu$ denote the unit inner normal vector to $\Omega$, and let $H$ denote the mean curvature of $\partial\Omega$ with the convention $H = \mathrm{div}_{\partial \Omega}(\nu)$. Recall that linearizing the Euler-Lagrange equations \eqref{e:bernoulliweak} around $U$ gives 
\begin{equation}\label{eq:lineareq}
    \left\{\begin{array}{ll}
\Delta u = 0 & \text{ in } \Omega \\
\partial_{\nu} u + H u = 0 & \text { on } \partial\Omega.
\end{array}
\right. \tag{$L$}
\end{equation} 
We also recall that  $H\geq 0$ with $H>0$ iff $\Omega$ is not a half-space, see \cite{CJK, JerisonSavin}.\\

Since $\Omega$ is a cone, we can separate variables and look for solutions to \eqref{eq:lineareq} of the form 
\[ u(x) = |x|^{\gamma} \phi\left( \frac{x}{|x|} \right), \]
where the restriction of $u$ to the sphere $\mathbb{S}^{d-1}$ yields a solution to the (interior) Robin spectral problem in $\Sigma:= \Omega \cap \mathbb{S}^{d-1}$
\begin{equation}\label{eq:evonlink}
    \left\{\begin{array}{ll}
        -\Delta_{\mathbb{S}^{d-1}} \phi = \lambda \phi  & \text{ in } \Sigma \\
        \partial_{\nu} \phi + H\phi = 0 & \text{ on } \partial\Sigma 
    \end{array}\right.\tag{$L_s$}
\end{equation}
with an associated homogeneity given by 
\begin{equation}\label{eq gamma}
    \gamma_\pm := -\frac{d-2}{2} \pm \sqrt{\left(\frac{d-2}{2} \right)^2 + \lambda }.
\end{equation}

Solutions to \eqref{eq:evonlink} are eigenfunctions $\{\phi_j\}_{j=1}^\infty$ of the Laplace-Beltrami operator on $\Sigma \subset \mathbb{S}^{d-1}$ with Robin boundary condition, with discrete spectrum given by the eigenvalues
\[ \lambda_1 < \lambda_2 \leq \cdots \to +\infty. \]
Moreover, the first eigenvalue $\lambda_1$ has the variational characterization
\begin{equation}\label{eq quadratic form}
   \lambda_1 = \inf \frac{ \int_{\Sigma} |\nabla_{\theta}w|^2 - \int_{\partial \Sigma} H w^2}{\int_{\Sigma} w^2}.
\end{equation}
Notice that testing \eqref{eq quadratic form} with a constant function provides the upper bound
\begin{equation}\label{eq quadratic form2}
   \lambda_1 \leq  - \frac{\int_{\partial \Sigma} H }{\mathcal{H}^{d-1}(\Sigma)}<0,
\end{equation}
where the strict inequality holds if and only if $\Omega$ is non-planar.

For $j\geq 2$, $\lambda_j$ is associated with a pair of homogeneous solutions to \eqref{eq:lineareq} which can be written as $f_j(r)\phi_j(\omega)$ and $f_{-j}(r)\phi_j(\omega)$ with $f_j(r)=r^{-\frac{(d-2)}{2}+\delta_j}$, where $\delta_j := \sqrt{\left(\frac{d-2}{2} \right)^2 + \lambda_j}$ and $\delta_{-j}:=-\delta_j$. For the case $j=1$, there is a slight subtlety. Let us first recall that the stability of the solution $U$, as a critical point of \eqref{eq:functional}, is equivalent to 
\begin{equation}\label{eq:stablev}
    \lambda_1 \geq -\left(\frac{d-2}{2}\right)^2,
\end{equation} 
see for example \cite[Proposition 2.1]{JerisonSavin}. If the inequality in \eqref{eq:stablev} is strict, we say that $U$ is strictly stable and define $f_{\pm 1}(r)$ as above; however, in the case of equality (or equivalently $\delta_1=0$), we define $f_1(r)=r^{-\frac{(d-2)}{2}} \ln(r)$ and $f_{-1}(r)=r^{-\frac{(d-2)}{2}}$, see \cite{DSJS}. 

Geometrically, the associated eigenfunction $\phi_1$ is the infinitesimal generator of shrinking/enlarging the aperture of the cone $\Omega$, and induces the foliation constructed in \cite{DSJS}. More precisely, if $U$ is a strictly stable global minimizer, \cite[Theorem 1-(iii)]{DSJS} guarantees the existence of universal constants $R, \bar{a}, \bar{b}, \alpha '>0$ such that
\begin{eqnarray} \label{eq:upperfol}
        \overline{U}(x)&=&U(x)+\bar{a}r^{-\frac{(d-2)}{2}\pm\delta_1}\phi_1(\omega)
        +O\!\bigl(|x|^{-\frac{(d-2)}{2}\pm\delta_1-\alpha'}\bigr),
        \qquad\\  \label{eq:lowerfol}
          \underline{U}(x)&=&U(x)-\bar{b}r^{-\frac{(d-2)}{2}\pm\delta_1}\phi_1(\omega)
        +O\!\bigl(|x|^{-\frac{(d-2)}{2}\pm\delta_1-\alpha'}\bigr),
\end{eqnarray}
for $ x\in\{U>0\}\setminus B_{R}$, where $\overline{U}, \underline{U}$ are global minimizers with analytic free boundaries (see Theorem \ref{thm foliation}).\\\

Additionally, infinitesimal generators of translations and rotations are always  solutions of \eqref{eq:evonlink}. In the case of translations, the generators are given by the directional derivatives $\partial_{e_i} U$ for $i=1,\ldots, d$, which are solutions to \eqref{eq:evonlink} associated with the eigenvalue $\lambda=0$. To see this, notice that the interior equation follows from the harmonicity and $0$-homogeneity of $\partial_{e_i} U$ in $\R^d$, whereas the boundary condition comes from noticing that 
\begin{equation}\label{eq boundaryhessian}
    D^2U[\nabla U]=-H\nabla U
\end{equation}
on $\partial \Omega\setminus \{0\}$, see \cite[Section 2]{{JerisonSavin}}. In the case of rotations, we notice that if $M$ is a skew-symmetric matrix, and $O(t)=\exp(tM)$ is the corresponding one-parameter family of rotations generated by $M$, then $U_t(x)=\frac{U(O(t)x)-U(x)}{t}$ is harmonic for any $t\neq 0$, and thus $ \phi=\frac{d}{d t} U_t|_{t=0}(x) = \nabla U(x)\cdot M[x] $ is a $1$-homogeneous harmonic function and thus a solution to the interior equation in \eqref{eq:evonlink} associated with the eigenvalue $\lambda=d-1$. Regarding the boundary condition, we observe that \eqref{eq boundaryhessian} combined with the skew-symmetry of $M$ yields
$$\nabla\phi\cdot \nabla U(x) = D^2U(x)[\nabla U(x)]\cdot M[x]+M[\nabla U(x),\nabla U(x)]=-H\phi.$$ 

Following the conventions in the minimal surface and free boundary literature we consider the following two definitions.

\begin{definition}[Integrability]\label{def: intcone}
   A regular cone $U$ is said to be integrable through rotations if all $1$-homogeneous solutions of \eqref{eq:lineareq} are generated by rotations of $U$.
\end{definition}

\begin{remark}\label{r:intdiff}
    Let us discuss the difference between our notion of integrability and \cite[Definition 2.4]{EnSVEpi}. Our condition characterizes one-homogeneous solutions to \eqref{eq:lineareq} as rotations, whereas \cite{EnSVEpi} characterizes the kernel elements of the ``one-homogeneous part" of the energy \eqref{eq:functional}, which they call $\mathcal F$, as rotations (and has an additional technical condition). 
    
    These notions are similar but hard to compare because, imprecisely, \cite{EnSVEpi} perturbs directly the free boundary whereas we perturb the function. However, we are able to show that all solutions which are integrable in the sense of \cite{EnSVEpi} are integrable in our sense. 
    
  Indeed note that 1-homogeneous solutions to \eqref{eq:lineareq} are in one-to-one correspondence with solutions to \eqref{eq:evonlink} when $\lambda = d-1$, which in turn are kernel elements of $\delta^2\mathcal F(0)$ in \cite{EnSVEpi}, see \cite[Equation (A.14)]{EnSVEpi} (note that any solution to \eqref{eq:evonlink} with $\lambda = (d-1)$ must integrate to zero on $\partial \Sigma$ by a simple integration by parts argument). This implies that every cone which is integrable in the sense of \cite{EnSVEpi} is also integrable in our sense. 

   The reverse is not so clear because there could be kernel elements of $\delta^2\mathcal F(0)$ which do not integrate to zero on $\Sigma$ (and thus do not solve \eqref{eq:evonlink}). Furthermore, the additional condition in \cite[Definition 2.4]{EnSVEpi}, that the signed eigenvalue of $\delta^2 \mathcal F(0)$ be perpendicular to the other elements of the index of $\delta^2\mathcal F(0)$, does not seem to be connected to the characterization of one-homogeneous solutions to \eqref{eq:lineareq}. 

   However, it should be mentioned that all known examples of integrable cones (i.e. the De Silva-Jerison cones) are integrable in both senses and it seems likely that all isoparametric cones also satisfy both conditions. We do not have an informed guess as to whether these conditions are equivalent a posteriori.
\end{remark}

\begin{definition}[Strong integrability]\label{def: sintegrable}
    A regular cone $U$ is said to be strongly integrable if the linearized operator \eqref{eq:lineareq} satisfies
    \begin{enumerate}
    \item  $\lambda_1 >  -\left(\frac{d-2}{2}\right)^2$, i.e., $U$ is strictly stable.
        \item  The eigenspace associated with $\lambda=0$ is generated by $\partial_{e_i} U$ for $i=1,\ldots, d$.
        \item The eigenspace associated with $\lambda = d-1$ is generated by functions of the form $\nabla U(x)\cdot M[x]$ with $M$ a skew symmetric matrix.
        \item If $\lambda$ is an eigenvalue of \eqref{eq:evonlink}, then $\lambda = \lambda_1, 0, d-1$ or $\lambda> d-1$.
    \end{enumerate} 
\end{definition}
\begin{remark}
    We note that the special threshold $d-1$ above is the first Dirichlet eigenvalue of the extremal domain $\Sigma:= \{U>0\} \cap \mathbb{S}^{d-1}$. In particular, when the first eigenfunction of \eqref{eq:evonlink} takes constant values on each connected component of the boundary, the above definition can also be understood as a classification of the Morse index of the extremal domain $\Sigma$ in $\mathbb{S}^{d-1}$, and $\operatorname{Ind}(\Sigma) = \#\{\lambda_1 < \lambda_j < \lambda_1^D(\Sigma) \}$ counting multiplicities. See \cite{extremalDom} for more details.
\end{remark}

We end this subsection by observing that,  given $0<R_1\leq R_2\leq\infty$, any solution to the linearized problem \eqref{eq:lineareq} of the form $u: \Omega \cap(B_{R_2}\setminus B_{R_1})\to \R$ can be written as linear combinations of solutions to the eigenvalue problem \eqref{eq:evonlink} (which we call Jacobi fields in a slight abuse of nomenclature), namely
    \begin{eqnarray}\nonumber
        u(r \omega) & = &\sum_{|j|>0} a_j f_j(r)\phi_j(\omega), \quad \int_{\Sigma}\phi_j^2=1\\\label{eq fourierexp}
        \text{satisfying } &&\sum_{|j|>0} a_j^2 f_j(r)^2 <\infty.
    \end{eqnarray}

\subsection{Particular solutions of the linearized operator}

In this subsection, we discuss the existence and decay properties of particular solutions of
\begin{equation}\label{eq part sol}
\left\{
\begin{alignedat}{2}
\Delta u &= 0,                   &\quad& \Omega\setminus B_R,\\
\partial_\nu u + H u &= G,      &\quad& \partial\Omega\setminus B_R.
\end{alignedat}
\right.
\end{equation}
where $\Omega$ is a cone, $R>0$ and $G$ is a continuous function with suitable decay properties. Using spherical coordinates, we can rewrite \eqref{eq part sol} for $u(r\cdot)$ on any sphere $\partial B_r$, for $r>R$, as follows
 \begin{equation}\label{eq spherical sol}
\left\{
\begin{alignedat}{2}
    \partial_{rr}u+ \frac{d-1}{r}  \partial_{r}u+ \frac{1}{r^2}  \Delta_{\mathbb{S}^{d-1}} u &= 0,  &\quad& \Sigma,\\
       \partial_\nu u+ Hu &= r G, &\quad& \partial \Sigma,
\end{alignedat}
\right.
\end{equation}
where $\Sigma =\mathbb{S}^{d-1}\cap \Omega$, and $\nu$ is the unit normal vector to $\partial \Sigma$ pointing into $\Sigma$. Note the boundary condition in \eqref{eq spherical sol} follows from the one in \eqref{eq part sol} after observing that $H(r\omega) = \frac{1}{r}H(\omega)$. This transformation allows us to find solutions of \eqref{eq part sol} by solving \eqref{eq spherical sol} for $r$ sufficiently large, which, in turn, we address by ``transferring'' the boundary data to the interior. We do this through the (boundary) Robin spectrum of $\Sigma$, i.e., 
\begin{equation}\label{eq bdry spectrum}
\left\{
\begin{alignedat}{2}
       \Delta_{\mathbb{S}^{d-1}} \psi_k &= 0,  &\quad& \Sigma,\\
       \partial_\nu \psi_k+ H\psi_k &= \ell_k\psi_k, &\quad& \partial \Sigma,
\end{alignedat}
\right.
\end{equation}
Notice that if $H$ is constant, solutions to \eqref{eq bdry spectrum} are Steklov eigenfunctions of $\Sigma$ whose associated eigenvalues are of the form $H-\ell_k$ (recall that $\nu$ is the inner unit normal to $\Sigma$). In general, since $H$ is bounded, standard linear theory guarantees that the eigenvalues of \eqref{eq bdry spectrum} are a discrete set of real numbers satisfying
\[ -\ell_1 < -\ell_2 \leq \cdots \to +\infty \]
and that $\{\psi_k\}$ form a orthonormal basis for $L^2(\partial \Sigma)$.\\

As a last ingredient for the construction of particular solutions, let us recall that interior Robin eigenfunctions, i.e., solutions $(\phi_k,\lambda_k)$ to \eqref{eq:evonlink}, induce homogeneous solutions of \eqref{eq:lineareq} with homogeneity $-\frac{(d-2)}{2}\pm\delta_k$, where $\delta_k := \sqrt{\left(\frac{d-2}{2} \right)^2 + \lambda_k}$.
 
\begin{lemma}\label{lemma part sol}
 Let $U > 0$ be a strictly stable regular cone and let $\Omega = \{U > 0\}$ be its positivity set. Let $R_0>0$, and $\beta>0$ such that $\frac{d}{2}\pm\delta_k -\beta \neq 0$ for $k\in \N$ and assume that $G\in C^2(\Omega \setminus B_{R_0})$ satisfying
 $$ |G(x)| + |x||\nabla G(x)| + |x|^2|\nabla^2 G(x)| \leq C|x|^{-{\beta}}. $$
 Then, there exists a solution $u_p$ to \eqref{eq part sol} with right hand side $G$ such that $|u(x)|\leq C(R_0)|x|^{{1-\beta}}$. 
\end{lemma}

\begin{proof}

 \medskip

 \noindent {\it Step 1:} Firstly, we take care of the inhomogeneous boundary condition and transfer it to the interior. We show the existence of a function $u_1(r\omega)$ solving 
\begin{equation}\label{eq spherical bsol}
  \left\{
\begin{alignedat}{2}
    \partial_{rr}u_1+ \frac{d-1}{r}  \partial_{r}u_1+ \frac{1}{r^2}  \Delta_{\mathbb{S}^{d-1}} u_1 &= f, &\quad& \Sigma\times(R_0,\infty),\\
       \partial_\nu u_1(r\omega)+ H(\omega)u_1(r\omega) &= r G(r\omega), &\quad& \partial \Sigma\times (R_0,\infty),
\end{alignedat}
\right.
\end{equation}
 where $\nu$ is the unit normal to $\partial \Sigma$ pointing into $\Sigma$; and furthermore $u_1$ and $f$ satisfy $\Vert u_1(r \cdot)\Vert_{L^2(\Sigma)}\leq C r^{1-\beta}$ and $\Vert f(r\cdot)\Vert_{L^2(\Sigma)} \leq C r^{-1-\beta}$ for $r\geq R_0$.

 Consider the boundary Robin spectrum $\{(\psi_k, \ell_k)\}_{k\in \N}$ as in \eqref{eq bdry spectrum}.
 For each $r\geq R_0$ we have the expansion  $ G(r \cdot) = \sum_{k\in \N} G_k(r)\psi_k$ in $L^2(\partial \Sigma)$, where $G_k(r)= \int_{\partial \Sigma} G(r\omega)\psi_k(\omega) d\omega$. By the assumption on the growth of $G$ we have that $|G_k(r)|\leq Cr^{-\beta}$ for $r\geq R_0$. 
 In principle, we can define $u_1(r\cdot)$ using the expansion of $G(r\cdot)$ and $f$ will in turn come from the corresponding radial derivative of $G_k(r)$'s. The caveat is the resonance between the boundary Robin spectrum \eqref{eq bdry spectrum} and interior Robin spectrum \eqref{eq:evonlink}, namely they have a common eigenfunction with eigenvalue $\ell_k = \lambda_{k'} = 0$. To resolve this issue, we let $I := \{ k\in \N \, | \, \ell_k =0\}$, i.e., the set of indices corresponding to the $0$ eigenvalue of \eqref{eq bdry spectrum}. Standard spectral theory implies that $I$ has finitely many elements (it has at least $d$, since $\partial_{e_i} U$ is an eigenfunction associated with $0$). For $r\geq R_0$, let 
 $$u_1(r\omega)= \sum_{ k \in \N\setminus I} \frac{rG_k(r)}{\ell_k}\psi_k(\omega)+ \sum_{k\in I} rG_k(r) U(\omega) \psi_k(\omega).$$
 Notice that
 $$\Delta u_{1}= \partial_{rr}u_1+ \frac{d-1}{r}  \partial_{r}u_1+ \frac{1}{r^2} \Delta_{\mathbb{S}^{d-1}} u_1 = f$$ where 
 \begin{align}
 	f(r \omega) & =   \partial_{rr}u_1+ \frac{d-1}{r}  \partial_{r}u_1 \\\label{eq f}
 		& \qquad +\frac{1}{r}\sum_{k\in I} G_k(r) \left(-(d-1) U(\omega)\psi_k(\omega)+2\nabla U(\omega)\cdot \nabla \psi_k(\omega) \right) \notag
 \end{align}
% \begin{eqnarray*}\label{eq f}
%     f(r \omega) =   \partial_{rr}u_1+ \frac{d-1}{r}  \partial_{r}u_1+\frac{1}{r}\sum_{k\in I} G_k(r)(-(n-1) U(\omega)\psi_k(\omega)+2\nabla U(\omega)\cdot \nabla \psi_k(\omega)
% \end{eqnarray*}
 satisfies  $\Vert f(r\cdot)\Vert_{L^2(\Sigma)} \leq C r^{-1-\beta}$ for $r\geq R_0$, thanks to the decay properties of $G$ and its derivatives. Regarding the boundary condition, we notice that since for all $k\in I$
 \[ \partial_{\nu}(U(\omega) \psi_k(\omega))=-H(\omega) U(\omega) \psi_k(\omega)+\partial_\nu U(\omega) \, \psi_k(\omega) = \psi_k(\omega), \]  
 we get
 \begin{equation*}
     \partial_\nu u_{1} (r\cdot) +Hu_{1}(r\cdot) = \sum_{ k \in \N} rG_k(r) \psi_k = rG(r\cdot)
 \end{equation*}
 proving \eqref{eq spherical bsol}. Additionally, from the definition of $u_{1}$ and the orthogonality properties of the spectrum, one readily sees that $\Vert u_1(r\cdot)\Vert_{L^2(\Sigma)}\leq Cr^{1-\beta}$ for $r\geq R_0$.

 \medskip

 \noindent {\it Step 2:}  We construct a function $u_2$ satisfying
\begin{equation}\label{eq spherical rpart}
      \left\{
\begin{alignedat}{2}
    \partial_{rr}u_2+ \frac{d-1}{r}  \partial_{r}u_2+ \frac{1}{r^2}  \Delta_{\mathbb{S}^{d-1}} u_2 &= f, &\quad& \Sigma\times(R_0,\infty),\\
       \partial_\nu (u_2(r\omega))+ H(\omega)u_2(r\omega) &= 0, &\quad& \partial \Sigma\times (R_0,\infty),
\end{alignedat}
\right.
\end{equation}
with $\Vert u_2(r\cdot)\Vert_{L^2(\Sigma)}\leq C r^{1-\beta}$ for $r\geq R_0$.\\

We now return to the interior Robin spectrum
  \begin{equation}\label{eq robin spectrum}
       \left\{
\begin{alignedat}{2}
        -\Delta_{\mathbb{S}^{d-1}} \phi_k &= \lambda_k \phi_k, &\quad& \Sigma,\\
       \partial_\nu \phi_k+ H\phi_k &= 0, &\quad& \partial \Sigma,
\end{alignedat}
\right.
\end{equation}
which forms an orthonormal basis of $L^2(\Sigma)$, as we already discussed. Consider the expansion $f(r\cdot) = \sum f_k(r) \phi_k$ in $L^2(\Sigma)$ for each $r\geq R_0$, where $f_k(r) = \int_{\Sigma} f(r\omega) \phi_k(\omega) d\mathcal{H}^{d-1}(\omega)$ is of the order $O(r^{-1-\beta})$. We consider the function $u_2(r\omega) = \sum u_k(r) \phi_k(\omega)$ where $u_k(r)$ satisfies the ODE
\begin{equation}\label{eq cauchyeuler}
   r^2 u_k''(r) +(d-1)r u_k'(r) -\lambda_k u_k = r^2f_k(r), \quad \text{ for } r\geq R_0.
\end{equation}
%\textcolor{blue}{Max: Are we missing an $1/r$ on $f_k$?}
Note that no matter the choice of $u_k$, the resulting function $u_2$ will satisfy the boundary condition $\partial_\nu u_2 + Hu_2 = 0$ (since $\Omega$ is a cone).

To solve for $u_k$, notice that \eqref{eq cauchyeuler} is an inhomogeneous Cauchy-Euler ODE, see also \cite[Lemma 2.11]{SimonSolomon}. So, if we consider the characteristic equation associated with \eqref{eq cauchyeuler}, namely $\alpha^2  + (d-2)\alpha - \lambda_k = 0$, we obtain the roots   $\gamma_k^{\pm} = \frac{-(d-2)}{2}\pm \delta_k$ with $\delta_k := \sqrt{\left(\frac{d-2}{2} \right)^2 + \lambda_k}$. So, as in \cite[Lemma 2.11]{SimonSolomon} we define 
\begin{equation}\label{e:intrep} u_k(r) = r^{\delta_k - \frac{d-2}{2}} \int_{b_k}^r s^{-2\delta_k - 1}\int_{a_k}^s t^{\delta_k + 1 + \frac{d-2}{2}} f_k(t)\, dt\, ds,
\end{equation} 
where
\[
a_k =
\begin{cases}
R_0       & \text{if } \frac{d}{2}+\delta_k -\beta > 0,\\
\infty  & \text{if } \frac{d}{2}+\delta_k-\beta < 0,
\end{cases}
\qquad
b_k =
\begin{cases}
R_0       & \text{if } \frac{d}{2}-\delta_k -\beta  > 0,\\
\infty  & \text{if } \frac{d}{2}-\delta_k -\beta  < 0,
\end{cases}
\]
which is well defined by our assumptions on $\beta$.\\

It is easily verified by direct computation that $u_k$ solves \eqref{eq cauchyeuler}, whereas since $|f_k(r)| \leq Cr^{-\beta - 1}$ we have that
$$\begin{aligned} \int_{a_k}^s t^{\delta_k + 1 + \frac{d-2}{2}} f_k(t)\, dt =& O(s^{\delta_k + \frac{d}{2} - \beta}) \Rightarrow \\
\int_{b_k}^r s^{-2\delta_k - 1}\int_{a_k}^s t^{\delta_k + 1 + \frac{d-2}{2}} f_k(t)\, dt\, ds =&   O(r^{-\delta_k + \frac{d}{2}-\beta}) \Rightarrow\\
|u_k(r)| =& O(r^{-\beta + 1}).\end{aligned}$$

Thus, we deduce that $u_p=u_1+u_2$ solves \eqref{eq part sol} and satisfies that $\Vert u_p(r\cdot)\Vert_{L^2(\Sigma)} \leq C r^{1-\beta}$ for $r\geq R_0$. Finally, we improve the $L^2$ decay to $L^\infty$ decay using linear regularity, for instance Corollary \ref{cor regdecay}.
\end{proof}

When $G=0$ in \eqref{eq part sol}, the convergence rates of the solutions can be combined with the representation \eqref{eq fourierexp} to refine their Fourier decomposition. 

\begin{lemma}\label{l:fourierimprovement}
Let $R_0>0$, and let $u:\Omega\setminus B_{R_0} \to \mathbb{R}$ be a solution to \eqref{eq:lineareq} with expansion 
\begin{equation}\label{eq expfour}
     u(r\omega)=\sum_{|j| > 0} a_{j}\,f_j(r)\varphi_{j}(\omega).
\end{equation}
Then $|u(x)|\leq C|x|^{-\beta}$ for some $\beta \in \R$ iff $a_{j}=0$ whenever $\delta_{j}-\frac{d-2}{2}>-\beta$.
\end{lemma}

\begin{proof}
The proof is exactly as in \cite[Lemma 2.10]{SimonSolomon} bearing in mind that decay in $L^2$  and decay in $L^\infty$ of solutions to \eqref{eq:lineareq} are equivalent thanks to Corollary \ref{cor regdecay}. 
\end{proof}

\subsection{Spectrum of the De Silva-Jerison cone}\label{ss:DSJ}
We finish this section by exhibiting an example of a strongly integrable solution in the sense of Definition \ref{def: sintegrable}. We consider the so called De Silva-Jerison cones, which are axially symmetric homogeneous solutions in dimensions $d\geq 3$ (see \cite[2.7]{AltCaf}); and, when $d\geq 7$, they are proven to be minimizers in \cite{DSJ}. %In fact, they are the only known examples so far of minimizing homogeneous solutions besides the flat solutions $c(x_d)_+$. % which are minimizers in dimensions $d \geq 7$. 

Up to rotations, we can fully describe these cones in spherical coordinates as follows
\begin{equation}\label{eq:DScone}
    \Omega = \left\{ (r, \theta, \varphi) \in \mathbb{R}_+ \times [0,\pi] \times \mathbb{S}^{d-2}: \frac{\pi}{2} - \theta_0 < \theta < \frac{\pi}{2} + \theta_0 \right\}, 
\end{equation} 
\[  U(r,\theta, \varphi) = \left\{\begin{array}{ll}
   cr\left(1-\cos^2(\theta) \right)^{-\frac{d-3}{4}} Q^{\frac{d-3}{2}}_{\frac{d-1}{2}}(\cos(\theta)),  & \text{ when } d \text { is odd}  \\
    cr\left(1-\cos^2(\theta) \right)^{-\frac{d-3}{4}} P^{\frac{d-3}{2}}_{\frac{d-1}{2}}(\cos(\theta)),  & \text{ when } d \text { is even}  
\end{array} \right. \]
and where the mean curvature of the free boundary on $\partial\Omega \cap \mathbb{S}^{d-1}$ is given by
\[ H = (d-2)\tan(\theta_0). \]
Here $\theta$ denotes the angle with the positive $e_1$-axis, $\theta_0>0$ is the unique angle such that $\Omega$ supports a positive one homogeneous harmonic function, and $P_{\ell}^m(t), Q_{\ell}^m(t)$ are the associated Legendre functions of the first and second kind, respectively.
 The following theorem summarizes the spectral information of the linearized operator \eqref{eq:evonlink} for the link of the De Silva-Jerison cones.

\begin{theorem}\label{thm: spectrum DSJ}
    Let $d\geq 3$ and let $U$ be the De Silva-Jerison cone characterized by the parameterization \eqref{eq:DScone}. Then,
    \begin{itemize}
        \item $\lambda_1 < -(d-2)$. %If $d \geq 7$, it also satisfies $\lambda_1>-\left(\frac{d-2}{2}\right)^2$. 
        \item $\lambda_2 = \ldots = \lambda_{d+1} = 0$ with associated eigenfunctions given by the infinitesimal generators of translations which take the form $\phi_{1+j} = c_j \partial_{e_j} U$, $j=1,\ldots,d$ for suitable normalizing constants $c_j$.
        \item $\lambda_{d+2} = \ldots = \lambda_{2d} = d-1$ with associated eigenfunctions $\phi_{d+j}$ given by (normalized) infinitesimal generators of rotations in the $e_1 e_j$-plane for $j=2,\ldots, d$.
        \item $\lambda_j > d-1$ for $j> 2d$.
    \end{itemize}
\end{theorem}
\begin{remark}
\begin{enumerate}
    \item Since the solution $U$ is axially symmetric there are no solutions generated by rotations in the space $e_1^{\perp} \approx \mathbb{R}^{d-1}$. For this reason the geometric multiplicity of the eigenvalue $d-1$ is
\[ \dim( O(d)) - \dim( O(d-1)) = d-1. \]
    \item Note that when $d\leq 6$, the above theorem implies that
    \[ \lambda_1 < -(d-2) \leq -\left(\frac{d-2}{2} \right)^2. \]
    Thus by the discussion before (see also \cite[Proposition 2.1]{JerisonSavin}) $U$ is unstable. This gives an alternative proof of the Proposition in \cite{CJK}, which says that axially symmetric cones are not minimizers in dimensions $d\leq 6$.\footnote{On the other hand, their argument has the advantage of showing $\lambda_1 \neq -\left(\frac{d-2}{2} \right)^2$ for any dimension. In particular, combined with \cite{DSJ} this implies that $U$ is \emph{strictly} stable when $d\geq 7$. Our argument does not seem to show this.}
\end{enumerate}
\end{remark}
\begin{proof}
\medskip

Using spherical coordinates as in \eqref{eq:DScone}, we can look for eigenfunctions of \eqref{eq:evonlink} associated with an eigenvalue $\lambda$ by separating variables. More precisely, by setting $\phi(\theta, \varphi) = f(\varphi)g(\theta)$ and exploiting the decomposition
\begin{equation}\label{eq decomplap}
    \Delta_{\mathbb{S}^{d-1}} = \frac{1}{\sin^2(\theta)} \Delta_{\mathbb{S}^{d-2}} + \frac{1}{\sin^{d-2}(\theta)} \frac{\partial}{\partial \theta} \left(\sin^{d-2}(\theta) \frac{\partial}{\partial \theta} \right),
\end{equation}
it follows that $f$ solves the eigenvalue problem
\begin{equation}\label{eq sphericalharm}
    -\Delta_{\mathbb{S}^{d-2}} f = \mu f \quad \mbox{on  $\mathbb{S}^{d-2}$,} 
\end{equation}
whereas $g$ satisfies the regular Sturm-Liouville problem
\begin{equation}\label{eq:SL}
\left\{
\begin{alignedat}{2}    
        \left(\sin^{d-2}(\theta) g' \right)'- \mu \sin^{d-4}(\theta)  g &= -\lambda \sin^{d-2}(\theta) g & \quad & \text{ in } \frac{\pi}{2} - \theta_0 < \theta <  \frac{\pi}{2} + \theta_0 \\
       g' + Hg &= 0  & \quad & \text{ when } \theta =  \frac{\pi}{2} - \theta_0 \\
       -g' + Hg &= 0  & \quad & \text{ when } \theta =  \frac{\pi}{2} + \theta_0.
\end{alignedat}
 \right.\tag{SL}
\end{equation}
By the correspondence between spherical harmonics in $\mathbb{S}^{d-2}$ and homogeneous harmonic polynomials in $\mathbb{R}^{d-1}$, we have that the eigenvalues of \eqref{eq sphericalharm} are given by an increasing sequence $\{\mu_j\}$ satisfying
\[ \mu_1 = 0 \quad \text{ with eigenfunction } \quad f \equiv \text{const.} \]
\[  \mu_2 = \cdots = \mu_{d} = d-2 \quad \text{ with eigenfunction } \quad f=c\varphi \cdot e_j \text{ for } j= 2, \cdots, d, \]
\[ \mu_{d+1} = 2(d-1). \]
By the Sturm-Liouville theory, for each $\mu = \mu_j$ fixed (to be the $j$-th eigenvalue of $-\Delta_{\mathbb{S}^{d-2}}$), \eqref{eq:SL} has a discrete set of simple eigenvalues
\begin{equation}\label{eq mon SL}
    \lambda_{j,1} < \lambda_{j, 2} < \cdots < \lambda_{j,k} < \cdots \to +\infty;
\end{equation}
moreover, the corresponding eigenfunction, which we will denote by $g_{j,k}$, has exactly $(k-1)$ zeros. Additionally, by the variational characterization of eigenvalues we have that
\begin{equation}\label{def:lambdajk}
    \lambda_{j,k} = \sup_{V_k} \inf_{g\perp V_k } \mathcal{R}_{\mu_j}(g),
\end{equation}
where the supremum is taken over all $(k-1)$-dimensional linear subspaces in the weighted Sobolev space $W^{1,2}((\frac{\pi}{2}-\theta_0, \frac{\pi}{2}+\theta_0), \sin^{d-2}(\theta) d\theta )$ subject to the weighted Neumann boundary condition\footnote{By the Neumann boundary condition we mean $(\sin^{d-2}(\theta)g)'=0$ at the endpoints $\theta = \frac{\pi}{2} \pm \theta_0$.}, and $\mathcal{R}_{\mu}(g)$ denotes the Rayleigh quotient
\begin{eqnarray*}
   \mathcal{R}_{\mu}(g) &=&
   \dfrac{\int_{\frac{\pi}{2}-\theta_0}^{\frac{\pi}{2}+\theta_0} \left[ \sin^{-(d-2)}(\theta)((\sin^{d-2}(\theta) g)')^2  + \left( \mu - (d-2) \right) \sin^{d-4}(\theta)g^2 \right] d\theta}{\int_{\frac{\pi}{2}-\theta_0}^{\frac{\pi}{2}+\theta_0}  \sin^{d-2}(\theta) g^2 d\theta}.
   % \\
   % & = & \dfrac{\int_{\frac{\pi}{2}-\theta_0}^{\frac{\pi}{2}+\theta_0} \sin^{d-2}(\theta)(g')^2  + \mu \sin^{d-4}(\theta)g^2 d\theta }{\int_{\frac{\pi}{2}-\theta_0}^{\frac{\pi}{2}+\theta_0}  \sin^{d-2}(\theta) g^2}\\
   % &&- \dfrac{ \cos^{d-2}(\theta_0)H\left(g^2(\frac{\pi}{2}+\theta_0) + g^2(\frac{\pi}{2}-\theta_0) \right)}{\int_{\frac{\pi}{2}-\theta_0}^{\frac{\pi}{2}+\theta_0}  \sin^{d-2}(\theta) g^2}.
\end{eqnarray*}
For any pair of eigenvalues $\mu_i$ and $ \mu_j$, we have that 
\begin{equation*}
    \mathcal{R}_{\mu_j}(g) - \mathcal{R}_{\mu_i}(g) = \dfrac{ (\mu_j-\mu_i) \int_{\frac{\pi}{2}-\theta_0}^{\frac{\pi}{2}+\theta_0}  \sin^{d-4}(\theta)g^2 d\theta}{\int_{\frac{\pi}{2}-\theta_0}^{\frac{\pi}{2}+\theta_0}  \sin^{d-2}(\theta) g^2 d\theta}.
\end{equation*}
On the other hand, since  $\sin^{d-2}(\theta)<\sin^{d-4}(\theta)$ for $\theta \in (\frac{\pi}{2}-\theta_0, \frac{\pi}{2}+\theta_0) \setminus \{\frac{\pi}{2}\}$,  it follows that if $\mu_i < \mu_j$, then
\begin{equation}\label{tmp:Rayleighcp}
    \mu_j- \mu_i < \mathcal{R}_{\mu_j}(g) - \mathcal{R}_{\mu_i}(g),
\end{equation} 
for any $g\neq 0$. Hence
\begin{equation}\label{eq:monoev}   
    \mu_j - \mu_i \leq \lambda_{j,k}- \lambda_{i,k}.
\end{equation}
Moreover, we must have 
\begin{equation}\label{eq: strict}
\mu_j - \mu_i <  \lambda_{j,1}-\lambda_{i,1}  
\end{equation}
which follows by plugging in the (nonzero) minimizer $g$ for $\lambda_{j,1}$ and then use the variational characterization  \eqref{def:lambdajk}.\\

With this preliminary discussion in mind, our aim now is to fill in the following table of eigenvalues $\lambda_{j,i}$ for \eqref{eq:SL}.
\medskip
\[
\begin{array}{c|c|c|c}
          & \mu_1 = 0 & \mu_2 = \cdots = \mu_d = d-2  & \mu_{d+1} = 2(d-1) \\
         \hline
        \lambda_{j,1} & <-(d-2)  & 0, \text{ translation in } e_j & > d \\
        \hline
        \lambda_{j,2} & 0, \text{ translation in } e_1 & d-1, \text{ rotation in } e_1e_j\text{-plane} & \geq 2d-1 \\
        \hline
        \lambda_{j,3} & > d-1 & > d-1 & >2d-1 \\
\end{array}
\]

\medskip
\noindent
$\boldsymbol{\lambda_{1,1}:}$ Recall that the first eigenspace of $\Delta_{\mathbb{S}^{d-2}}$ is simple and spanned by the constant function $1$ with corresponding eigenvalue $\mu_1 = 0$. In addition, the eigenfunction $g_{1,1}$ to \eqref{eq:SL} with $\mu = \mu_1 = 0$ does not change sign, and thus the corresponding function $u=cg_{1,1}(\theta)$ does not change sign. Therefore it must be the first eigenfunction of \eqref{eq:evonlink}, with eigenvalue $\lambda=\lambda_{1,1}$. This corresponds to the expansion and shrinking of the cone $\Omega$. 

We also remark that $\lambda_{1,1} > -\left(\frac{d-2}{2} \right)^2$ when $d\geq 7$, namely the De Silva-Jerison cone is \emph{strictly} stable. In \cite[\S 3]{CJK} the authors consider the unique (up to a constant multiple) positive solution $W$ to the differential equation in \eqref{eq:SL} with $\lambda = -\left(\frac{d-2}{2} \right)^2$ and $\mu=0$ and satisfying even symmetry. Equivalently, this means that $G(x) = r^{-\frac{d-2}{2}} W(\theta)$ is a homogeneous harmonic function in $\Omega$. Moreover, when $d\geq 7$
\[ H < \left|W'/W\right| \text{ at the boundary } \theta = \frac{\pi}{2} \pm \theta_0.\footnote{Note that the interval $[\frac{\pi}{2} - \theta_0, \frac{\pi}{2} + \theta_0 ]$ in our notation is equivalent to the interval $[\theta_0, \pi-\theta_0]$ in the notation in \cite{CJK}.} \]
Thus $g_{1,1}$ is not a constant multiple of $W$ and $\lambda_{1,1} \neq -\left(\frac{d-2}{2} \right)^2$ by the Sturm-Liouville theory.\\

To study the next eigenfunctions, we first note that for the $1$-homogeneous solution $U(x) = rU(\theta)$, its partial derivatives (which are homogeneous of degree $0$) can be expressed in polar coordinates $(r,\theta, \varphi)$ as
\begin{equation}\label{eq:pder1U}
    \partial_{e_1} U(x) = \cos(\theta) U(\theta) -\sin(\theta) U'(\theta),
\end{equation}
\begin{equation}\label{eq:pderkU}
    \partial_{e_k} U(x) = \left[\cos(\theta) U'(\theta) + \sin(\theta) U(\theta) \right] \varphi \cdot e_k.
\end{equation}
In particular, we use the following observation: For any function $g(\theta)$ on $(\frac{\pi}{2}-\theta_0, \frac{\pi}{2}+\theta_0)$, if $\hat{g}(t)$ is defined such that $\hat{g}(\cos(\theta)) = g(\theta)$, then its derivative satisfies
\[ \hat{g}'(\cos(\theta)) = g'(\theta) \left( - \frac{1}{\sin(\theta)} \right). \]

\medskip

\noindent
$\boldsymbol{\lambda_{1,2}:}$ We claim that $\lambda_{1,2}=0$ with its eigenspace spanned by $g_{1,2}=\partial_{e_1} U$, which corresponds to translation in the $e_1$-direction. Indeed, notice that $g=\partial_{e_1}U$ only depends on $\theta$ and is an eigenfunction of \eqref{eq:evonlink} associated with $\lambda=0$, then it satisfies the boundary conditions in \eqref{eq:SL} and thanks to \eqref{eq decomplap} we have that
\[ \left(\sin^{d-2}(\theta) g'(\theta) \right)' = 0, \]

In other words, $g_{1,2}$ satisfies the equation in \eqref{eq:SL} with $\mu=0, \lambda = 0$. On the other hand, since $\partial_{e_1} U$ only depends on $\theta$, it changes sign only once at $\theta=\frac{\pi}{2}$ on the interval $(\pi/2-\theta_0, \pi/2+\theta_0)$. Otherwise, if it vanishes for another $\theta' \neq \frac{\pi}{2}$, it must achieve a local maximum or minimum for some $\theta$ between $\frac{\pi}{2}$ and $\theta'$. This violates the maximum principle for the harmonic function $\partial_{e_1}U$ in $\Omega$. 
Therefore, by the Sturm-Liouville theory it must be the second eigenfunction $g_{1,2}$, and thus $\lambda_{1,2}$ is equal to the corresponding eigenvalue $0$.\\

We also remark that since
\[ \lambda_{2,1} - \lambda_{1,1} > \mu_2 - \mu_1 = d-2, \]
we have $\lambda_{1,1} < -(d-2)$.\\

\medskip

\noindent
$\boldsymbol{\lambda_{j,1} \text{ with }   j\geq 2:}$ Let us consider first $j \in \{2,\ldots, d\}$. In this case, $\mu_j =  d-2$ with corresponding eigenfunctions $f_j(\varphi) = \varphi\cdot e_j$. We claim that $\lambda_{j,1} = 0$ with eigenfunctions 
\begin{equation}\label{eq partialder}
    \partial_{e_j} U = \left( \cos(\theta) U'(\theta) + \sin(\theta) U(\theta) \right) f_j(\varphi)
\end{equation}
which corresponds to translation in the $e_j$-direction. This is equivalent to saying that $g(\theta) = \cos(\theta) U'(\theta) + \sin(\theta) U(\theta)$ is the first eigenfunction of \eqref{eq:SL} when $\mu=\mu_j$ is fixed, and the corresponding eigenvalue is $0$. 
By the Sturm-Liouville theory, it suffices to show that $g(\theta)$ satisfies \eqref{eq:SL} with $\lambda = 0, \mu = d-2$ and that it does not change sign. The former simply follows from the fact that $\partial_{e_j} U$ is an eigenfunction of \eqref{eq:evonlink} associated with $\lambda=0$, so we can exploit the particular structure of \eqref{eq partialder} together with the identity \eqref{eq decomplap} as in the $\lambda_{1,2}$ case. To show the latter, we notice that the equation in \eqref{eq:SL} reads in this case
\begin{equation*}
    \left(\sin^{d-2}(\theta) g' \right)'= (d-2) \sin^{d-4}(\theta)  g \qquad \mbox{in $\Big(\frac{\pi}{2} - \theta_0 , \frac{\pi}{2} + \theta_0\Big)$}
\end{equation*}
from which we can readily see that $g$ can attain positive maxima or negative minima only at the boundary. On the other hand, since $U'\Big(\frac{\pi}{2}-\theta_0\Big)=- U'\Big(\frac{\pi}{2}+\theta_0\Big)$, we get by direct computation that $g\Big(\frac{\pi}{2}-\theta_0\Big)=g\Big(\frac{\pi}{2}+\theta_0\Big)$ implying that $g$ must be one-signed.\\

On the other hand, if $j \geq d+1$, since $\mu_j \geq 2(d-1)$ for all $j\geq d+1$, it follows from \eqref{eq: strict} that 
\[ \lambda_{j,1} > \lambda_{d, 1}+ \mu_j-\mu_d\geq 2(d-1)-(d-2) = d. \]

$\boldsymbol{\lambda_{j,2} \text{ with }  j\geq 2:}$ Let $A_j$ denote the skew-symmetric matrix which has only two non-zero entries $a_{1j}=1$ and $a_{j1}=-1$. It is the infinitesimal generator of the rotation in the $e_1e_j$-plane, namely %$\exp(tA_k)$ maps $(r\cos(\theta), r\sin(\theta)\varphi)$ to 
%\[ (\cos(t) r\cos(\theta) + \sin(t)r\sin(\theta) \varphi_k, r\sin(\theta)\varphi_2, \cdots, -\sin(t)r\cos(\theta) + \cos(t)r\sin(\theta) \varphi_k, \cdots, r\sin(\theta) \varphi_d) \]
\[ \exp(tA_j): \begin{pmatrix}
r\cos(\theta) \\
r\sin(\theta)\varphi 
\end{pmatrix} \mapsto \begin{pmatrix}
    \cos(t) \cdot r\cos(\theta) + \sin(t)\cdot r\sin(\theta) \varphi_j \\ r\sin(\theta)\varphi_2 \\ \vdots \\ -\sin(t) \cdot r\cos(\theta) + \cos(t)\cdot r\sin(\theta) \varphi_j \\ \vdots \\ r\sin(\theta) \varphi_d
\end{pmatrix}.  \]
Moreover, by \eqref{eq:pder1U} and \eqref{eq:pderkU} we have
\begin{align*}
    \frac{d}{dt}\Big|_{t=0} U\left(\exp(tA_j) x \right) = \langle \nabla U(x), A_j x\rangle & = \partial_{e_1}U \cdot r\sin(\theta) \varphi_j + \partial_{e_j}U \cdot (-r\cos(\theta)) \\
    & = -rU'(\theta) \varphi_j.
\end{align*}

We claim that $\lambda_{j,2} = d-1$ with eigenfunctions given by (a multiple of) $\phi_j(\theta,\varphi)=U'(\theta) \varphi_j$. This follows from noticing that $U'$ is a solution of \eqref{eq:SL} with $\mu=d-2, \lambda = d-1$. Indeed, we show the latter observing that since $\Delta U = 0$, $U$ is $1$-homogeneous, and $U$ only depends on $\theta$, we can deduce from \eqref{eq decomplap} that
\[ \left(\sin^{d-2}(\theta) U' \right)' = -(d-1) \sin^{d-2}(\theta) U, \]
or equivalently,
\begin{equation}\label{tmp:U}
    U^{''} + \frac{d-2}{\tan(\theta)} U' = -(d-1)U. 
\end{equation} 
Differentiating \eqref{tmp:U} on both sides with respect to $\theta$ yields
\begin{equation*}
    \sin^2(\theta) U^{'''} + (d-2)\sin(\theta)\cos(\theta) U^{''} = \left( (d-2)-(d-1)\sin^2(\theta) \right) U', 
\end{equation*} 
which can be rewritten as
\begin{equation*}
  \left(\sin^{d-2}(\theta) U'' \right)' -(d-2)\sin^{d-4}U'  = -(d-1)\sin^{d-2} U', 
\end{equation*} 
which shows that $U'$ satisfies the interior equation in \eqref{eq:SL}, while the boundary conditions are satisfied since $U'(\theta)\varphi_j$ is the infinitesimal generator of a translation.\\

On the other hand, we have shown that $\partial_{e_1} U$ only vanishes when $\theta = \frac{\pi}{2}$, and thus
\[ \partial_{e_1} U = \cos(\theta) U(\theta) - \sin(\theta) U'(\theta) < 0 \quad \text{ in } \frac{\pi}{2}-\theta_0 \leq \theta < \frac{\pi}{2}.  \]
Since $U \geq 0$, this is only possible if $U'(\theta) > 0$. Combined with the symmetry across $\theta=\frac{\pi}{2}$, it follows that $U'(\theta)$ vanishes only at $\theta = \frac{\pi}{2}$. Again by the Sturm-Liouville theory, for $\mu= d-2$ fixed, $U'(\theta)$ is the second eigenfunction to \eqref{eq:SL}, and thus $\lambda_{j,2}=d-1$ for $j=2,\cdots, d$.\\

Lastly, by \eqref{eq:monoev} $\lambda_{j,2} \geq \lambda_{d,2}+d=2d-1 > d-1$ for $j\geq d+1$.

\medskip

Thanks to the previous considerations together with the monotonicity properties \eqref{eq mon SL} and \eqref{eq:monoev} it suffices to show that $\lambda_{1,3}>d-1$ to finish the proof.\\

\medskip

\noindent $\boldsymbol{\lambda_{1,3}:}$  By Sturm-Liouville theory, the corresponding eigenfunction $g_{1,3}(\theta)$ has exactly two zeroes in $(\frac{\pi}{2}-\theta_0, \frac{\pi}{2}+\theta_0)$, which we denote by $\theta_1<\theta_2$. In particular, within the middle band $(\theta_1, \theta_2)$, $g_{1,3}$ does not change sign, and thus $\lambda_{1,3}$ is the first eigenvalue to the Dirichlet boundary value problem
\[ \left\{\begin{array}{ll}
    \left(\sin^{d-2}(\theta) g' \right)' = -\lambda \sin^{d-2}(\theta) g & \text{ in } \theta_1  < \theta <  \theta_2 \\
    g = 0  & \text{ when } \theta =  \theta_1, \theta_2. \\
\end{array} \right. \]
Recall that $U(\theta)$, the $1$-homogeneous solution restricted to $\mathbb{S}^d$, satisfies the following Dirichlet boundary value problem
\[ \left\{\begin{array}{ll}
    \left(\sin^{d-2}(\theta) U' \right)' = -(d-1) \sin^{d-2}(\theta) U & \text{ in } \frac{\pi}{2}-\theta_0  < \theta <  \frac{\pi}{2}+\theta_0 \\
    U = 0  & \text{ when } \theta =  \frac{\pi}{2}\pm \theta_0. \\
\end{array} \right. \]
Since $U$ is positive in the domain, it must be the first eigenfunction, and thus the first \emph{Dirichlet} eigenvalue to the above weighted Laplacian in $(\frac{\pi}{2}-\theta_0, \frac{\pi}{2}+\theta_0)$ is $(d-1)$. Since $(\theta_1, \theta_2) \subsetneq (\frac{\pi}{2}-\theta_0, \frac{\pi}{2}+\theta_0)$, by comparison of the first Dirichlet eigenvalue we have $\lambda_{1,3} > d-1$.

\end{proof}

\section{Qualitative properties of solutions close to regular cones}\label{sec: conv}

We begin this section by defining our class of well-behaved weak solutions, following \cite[Definition 5.1]{AltCaf} and \cite[Theorem 5.1]{Weiss}.

\begin{definition}\label{d:weaksol}
  Let $D \subset \R^d$ be an open set. We say that $u:D \to \R_+$ is a non-degenerate weak solution of the one-phase Bernoulli problem in $D$ if:
    \begin{itemize}
        \item $u\in H^1_{loc}(D)\cap C(D)$. 
        \item $u\geq 0$, and $\{u > 0\}$ is a set of locally finite perimeter in $D$.
        \item $\Delta u = \mathcal H^{d-1}|_{\partial^* \{u > 0\}}$ in the sense of distributions i.e. $$-\int_D \nabla u \cdot \nabla \eta = \int_{\partial^*\{u > 0\}\cap D} \eta \, d\mathcal H^{d-1},\, \forall \eta \in C_c^\infty(D),$$
        where $\partial^* \{u>0\}$ denotes the reduced boundary of $\{u>0\}$.
        
        \item $u$ is non-degenerate and regular, i.e. there exist constants $c_1, c_2 > 0$ such that for any $x_0\in \partial \{u > 0\}$ and $r > 0$ we have 
        \begin{equation}\label{eq:nondeg/reg}
        c_1 \leq \frac{1}{r}\fint_{\partial B_r(x_0)} u\, d\mathcal H^{d-1} \leq c_2.     
        \end{equation}
       whenever $B_r(x_0)\subset D$.
    
    \end{itemize}
\end{definition}

The proof of \cite[Theorem 5.1]{Weiss} shows that non-degenerate weak solutions in the sense above are variational solutions in the sense of \cite{Weiss}. In particular, the boundary-adjusted energy
\begin{equation}\label{eq weissenergy}
W(u,r)
=\frac{1}{r^d}\int_{B_r}\left(|\nabla u|^2+\chi_{\{u>0\}}\right)dx
-\frac{1}{r^{d+1}}\int_{\partial B_r}u^2d\mathcal H^{d-1}
\end{equation}
satisfies, for a.e.\ $r>0$ (such that $B_r \subset D$), 
\begin{equation}\label{eq weissder}
\frac{d}{dr}W(u,r)
=\frac{2}{r^{d+2}}\int_{\partial B_r}
(x\cdot\nabla u-u)^2d\mathcal H^{d-1}.
\end{equation}

\begin{remark}\label{r: abundance}
In \cite[Theorem 5.1]{Weiss}, Weiss additionally assumes a locally uniform lower density estimate for the zero phase. This assumption is not needed to show that non-degenerate weak solutions are variational solutions or to derive the monotonicity formula; see the discussion after \cite[Definition 5.1]{Weiss}. Its role is to ensure that blow-up limits remain non-degenerate weak solutions by preventing the collapse of the zero set in the limit. 

Our arguments do not require this general closure property, as such we do not need to assume any abundance of the zero set. 
\end{remark}

One important property of our weak solutions, is that they satisfy a comparison type result with classical solutions. We will use this in Section \ref{s:smooth} but state and prove the result here: 

\begin{lemma}\label{lemma comparison}
   Let $D\subset\R^d$ be a connected open set, let $w$ be a classical solution to the one-phase problem in $D$, and let $u$ be a non-degenerate weak solution in $D$. Assume that both $\{w>0\}$ and $\{u>0\}$ are connected and that either \begin{equation}\label{eq regtop}
        w \geq u \quad \mbox{in $D$} 
    \end{equation}
    or    \begin{equation}\label{eq weaktop}
        u \geq w \quad \mbox{in $D$},
    \end{equation}
    holds.
    
If $w(x_0)=u(x_0)$ for some $x_0 \in \overline{\{w>0\}} \cap \overline{\{u>0\}} \cap D$, then $u=w$ in $D$.
\end{lemma}
 \begin{proof}
 Notice that if at $x_0$ either $w$ or $u$ is positive, the result is trivial or follows by the strong maximum principle for the Laplacian combined with the connectedness of the positivity phases. We thus assume that $x_0$ is a free boundary of both solutions. We consider the two cases separately.

 \medskip
 \noindent {\it Case 1:} Assume that \eqref{eq regtop} holds. \\
 
In this scenario, we can derive a contradiction by blowing up at $x_0$. Specifically, consider the rescaled functions
$$u_r(x) = \frac{u(x_0 + r x)}{r} \quad \text{and} \quad w_r(x) = \frac{w(x_0 + r x)}{r}.$$
By the blow-up compactness in \cite{Weiss}, there exists a subsequence $\{r_j\}$ such that $u_j:=u_{r_j} \to u_0$, where $u_0 \neq 0$ is a  $1$-homogeneous function, harmonic in its positivity cone. Since the free boundary of $w$ is regular, we have that  $w_{r_j}\to (x\cdot e)_+$ for some $e \in \mathbb{S}^{d-1}$ as $j\to \infty$, implying by the ordering $w_r \geq u_r$ that the positivity set of $u_0$ is contained in the positivity set of $(x\cdot e)_+$. Denoting still by $u_0$ the restriction of $u_0$ to the sphere, we have that
    \begin{equation*}
        -\Delta_{\mathbb{S}^{d-1}} u_0 = (d-1)u_0 \quad \mbox{in $\{u_0>0\}$},      
    \end{equation*}
with $\{u_0>0\}\subset \{x\cdot e >0\}$. This implies that the first Dirichlet eigenvalue for the Laplace-Beltrami operator on the sphere satisfies $\lambda_1(\{u_0>0\})=d-1 = \lambda_1(\{x\cdot e >0\})$ which forces $\{u_0>0\}= \{x\cdot e >0\}$ and $u_0 = \mu (x\cdot e)_+$, by  the strict monotonicity of the first Dirichlet eigenvalue with respect to inclusions and the simplicity of the first eigenvalue. Moreover, the ordering forces $\mu \leq 1$. On the other hand, since $u_j$ converges weakly in $H^1$ to $u_0$ and $\chi_{\{u_j>0\}}\to \chi_{\{u_0>0\}}$ locally in $L^1$ (see \cite[Section 4.7]{AltCaf}) we deduce from the lower semicontinuity of (weighted) perimeters and the definition of non-degenerate weak solution that
\begin{equation}\label{eq limiteqweak}
    -\int_{\R^d} \nabla  u_0\cdot \nabla \varphi = -\lim_{j\to \infty}  \int_{\R^d} \nabla  u_j\cdot \nabla \varphi = \lim_{j\to \infty}   \int_{\partial^* \{u_j>0\}}  \varphi d\mathcal{H}^{d-1} \geq \int_{\partial^* \{u_0>0\}}  \varphi d\mathcal{H}^{d-1}
\end{equation}
for any $\varphi \in C_0^1(\R^d)$ with $\varphi \geq 0$. On the other hand, computing directly we have that
\begin{equation}
      -\int_{\R^d} \nabla  u_0\cdot \nabla \varphi = \mu \int_{ \{x\cdot e=0\}}  \varphi d\mathcal{H}^{d-1}=\mu\int_{\partial^* \{u_0>0\}}  \varphi d\mathcal{H}^{d-1},
\end{equation}
which combined with \eqref{eq limiteqweak} implies that $\mu=1$ and thus $u_0 = (x\cdot e)_+$. In particular, this implies that $x_0$ is a regular free boundary point of $u$ and by \cite[Theorem 8.1]{AltCaf}, the free boundary of $u$ is regular in a neighborhood of $x_0$. Since $\{u>0\}\subset\{w>0\}$, the two regular free boundaries are tangent at $x_0$. The boundary Hopf lemma applied to $w-u$ then contradicts the fact that both solutions have unit normal derivative at their free boundaries, unless they agree. By unique continuation and the connectedness of the positivity phases, this implies $u=w$ in $D$.

\medskip
 \noindent {\it Case 2:} Assume that \eqref{eq weaktop} holds. \\

In this case, the ordering is unfavorable for the preceding blow-up argument, since the rescalings of $u$ could converge in principle to a two-plane solution, so we argue differently. Notice instead that in this case $\{w>0\} \subset \{u>0\}$ which combined with the regularity of $w$ implies that $B\subset \{w>0\} \subset \{u>0\}$ for  some ball $B$ such that $x_0 \in \partial B$. Let $\nu$ denote the inward unit normal to $B$ at $x_0$. The function $u-w$ is nonnegative and harmonic in $B$. Unless it vanishes identically, the boundary Hopf lemma together with the Bernoulli boundary condition of $w$ gives, for some $\kappa>0$,
\begin{equation}\label{eq Hopf}
   \liminf_{t\to0^+}
\frac{u(x_0+t\nu)-u(x_0)}{t}
\geq 1+\kappa.
\end{equation}
 On the other hand, \cite[Theorem 6.3]{AltCaf} implies that
 \begin{equation*}
    \lim_{r\to 0^+} \sup_{B_r(x_0)} |\nabla u(x)|\leq 1,
 \end{equation*}
 which contradicts \eqref{eq Hopf} by a direct application of the mean value theorem. 
 \end{proof}

\begin{remark}\label{r:mintype}
We note that all minimizers satisfy the conditions in Definition \ref{d:weaksol}, see \cite[Section 1]{CafSalsa}, but so do many other non-minimizing solutions. In particular, Kamburov and Wang \cite{KW} proved the non-degeneracy of stable classical solutions. Together with the standard upper growth estimate and distributional formulation, this shows that stable classical solutions are included in the preceding definition.

More generally, our arguments should extend to any class of solutions for which an analogue of Lemma \ref{lemma comparison} holds and which possesses a Weiss-type monotonicity formula, a flatness-implies-regularity theorem, and a compactness principle under which $L^2_{\mathrm{loc}}$-convergence of the solutions yields local Hausdorff convergence of their free boundaries.

Kriventsov and Weiss \cite{KWeiss} (see also Weiss \cite{Weiss}) study a broader class of variational solutions and establish rectifiability, locally finite Hausdorff measure, and compactness results for their free boundaries. Their framework allows phenomena such as degenerate and two-sided singular points and is therefore not, in general, contained in Definition \ref{d:weaksol}. Nevertheless, their results provide several of the structural ingredients described above, suggesting that our methods could be extended to this class under appropriate uniform Lipschitz and non-degeneracy assumptions.

    % We believe it is possible to drop the assumption on the abundance of the zero set. This condition is only used to ensure the validity of the monotonicity formula; however, in every instance where the formula is applied, the weak solution under consideration is already known to be sufficiently close to a classical solution whose zero set is abundant. We do not pursue this direction -- carefully revisiting each step of the proof to remove the assumption -- to favor clarity over generality.

\end{remark}

In preparation for our quantitative analysis of solutions asymptotic to regular homogeneous solutions, we show that under suitable proximity conditions between a non-degenerate weak solution $u$ and a classical homogeneous solution $U$ of the one-phase problem, we can harmonically extend $u$ to a set (locally) containing $\{U>0\}$. Given $r>s>0$ we denote $A_{s,r} = B_r \setminus \overline{B_s}$ and $A_r := A_{\frac{1}{r},r}$.

\begin{lemma}\label{lemma: extension}
    Let $U$ be a regular cone and let $u$ be a non-degenerate weak solution in $A_4$. There exist $\delta_0, \mu>0$ and $C>1$ depending only on $U$ and on the constants in Definition \ref{d:weaksol}, such that if $\Vert u- U\Vert_{L^2(A_{4})}\leq \delta_0$, the following statements hold.
    \begin{enumerate}
        \item  $u$ is a classical solution of \eqref{e:bernoulliweak} in $A_{3}$.
        \item $\partial \{ u>0\}\cap A_{3}$ is an analytic hypersurface.
             \item Let $\Sigma_U:=\{U>0\}\cap\mathbb S^{d-1}$ and define
    \begin{equation*}
    \Omega_\mu:=
    \left\{r\theta:r>0,\ 
    \operatorname{dist}_{\mathbb S^{d-1}}(\theta,\Sigma_U)<\mu
    \right\}.
    \end{equation*}
There exist unique harmonic continuations $\tilde u$ and $\tilde U$ of $u|_{\{u>0\}}$ and $U|_{\{U>0\}}$, respectively, defined in $\Omega_{2\mu}\cap A_{7/2}$. Furthermore, $\tilde U$ is $1$-homogeneous.

\item Let $(w_1,w_2)\in\{(u,U),(U,u)\}$ and let $\tilde w_i$ denote the corresponding harmonic continuation. We have that
\begin{equation}\label{eq:distest}
   |\tilde{w}_i(x)| \leq C \operatorname{dist}(x, \partial\{w_i>0\}) \quad \mbox{ for $x \in \Omega_\mu \cap A_3$},
\end{equation}
and
\begin{equation}\label{eq:extu}
    \Vert \tilde{w}_i\Vert_{C^k(\Omega_\mu \cap A_3)}\leq C(k)\Vert w_i\Vert_{L^\infty (A_{4})},        
\end{equation}
for any $k\in \N$ and $i=1,2$. We also note that $\tilde w_i \leq 0$ in $(\Omega_\mu \setminus \{w_i > 0\}) \cap A_3.$

  \item Given $\e>0$ and $k\in\N$, there exists
$\delta=\delta(\e,k,U,c_1,c_2)\in(0,\delta_0]$ such that, if
$\Vert u-U\Vert_{L^2(A_4)}\leq\delta$, then
\begin{equation}\label{tmp:compCk}
\|\tilde{w}_1-w_2\|_{C^k(\overline{\{w_2>0\}}\cap A_3)}
\leq\e.
\end{equation}

\item There exists an analytic function
$\eta$ on $\partial\{w_2>0\}\cap A_3$ such that
$$
\partial\{w_1>0\}\cap A_2
=
\left\{
x+\eta(x)\nabla w_2(x):
x\in\partial\{w_2>0\}\cap A_3
\right\}\cap A_2.
$$
Furthermore, if $v=\tilde w_1-w_2$ in
$\overline{\{w_2>0\}}\cap A_3$, then
\begin{equation}\label{eq expansion eta}
\eta(x)
=
-\frac{v(x)}
{\nabla\tilde w_1(x)\cdot\nabla w_2(x)}
+O(v(x)^2).
\end{equation}
Moreover, $v$ satisfies the boundary value problem
\begin{equation}\label{eq:linaroundu}
\left\{
\begin{alignedat}{2}
\Delta v&=0,
&\quad&\mbox{in $\{w_2>0\}\cap A_3$},\\
\partial_\nu v+hv
&=-\frac12|\nabla v|^2+O(v^2),
&\quad&\mbox{on $\partial\{w_2>0\}\cap A_3$},
\end{alignedat}
\right.
\end{equation}
where $\nu=\nabla w_2$ is the unit normal pointing towards
$\{w_2>0\}$,
\begin{equation}\label{eq:h}
h(x)
=
-\dfrac{
D^2\tilde w_1(x)
[\nabla\tilde w_1(x),\nabla w_2(x)]
}{
\nabla\tilde w_1(x)\cdot\nabla w_2(x)
},
\end{equation}
and $|O(v^2)|\leq C|v|^2$.
    \end{enumerate}
\end{lemma}
\begin{remark}
    We remark that by the $C^k$-closeness of $\tilde{w}_1$ and $w_2$, the function $h$ above is sufficiently close to $-\partial_{\nu\nu}w_2(x)$, namely the mean curvature of $\partial \{w_2 >0\}$. If $U$ is non-flat,
then, after decreasing $\delta_0$, the function $h$ is uniformly
positive.
\end{remark}

\begin{proof}
\noindent {\it Step 1:} We show (1)--(3) by compactness and regularity theory for non-degenerate weak solutions.

We start by recalling that any non-degenerate weak solution $u$ is locally Lipschitz, with
$\mathrm{Lip}(u;A_{\frac{15}{4}})\leq Cc_2$, where $c_2$ is as in \eqref{eq:nondeg/reg}; see \cite[Theorem 4.3]{AltCaf}. Thus, by the $L^\infty$-type Gagliardo--Nirenberg interpolation inequality, we have that
\begin{equation}\label{eq:U}
\Vert u-U\Vert_{L^\infty(A_{\frac{15}{4}})}
\text{ is sufficiently small},
\end{equation}
upon shrinking the upper bound $\delta_0$ for its $L^2$ norm.

Fix an open set $D$ such that
$\overline{A_{7/2}}\subset D\subset \subset A_{\frac{15}{4}}$.
We also recall that the uniform non-degeneracy assumption \eqref{eq:nondeg/reg}, together with the $L^\infty$ proximity, implies the local convergence of the free boundaries and of the corresponding phases; see \cite[Section 4.7]{AltCaf}. In particular, upon shrinking $\delta_0$, we have
\begin{equation}\label{eq;H}
d_H\big(\partial\{u>0\}\cap D,\partial\{U>0\}\cap D\big)
\text{ sufficiently small},
\end{equation}
and the positivity and zero phases of $u$ agree with those of $U$ away from a sufficiently small neighborhood of $\partial\{U>0\}$.

On the other hand, since $U$ is a classical solution of \eqref{e:bernoulliweak}, given $\e>0$ we can find $r>0$ such that, for any $x_0\in\partial\{U>0\}\cap\overline D$,
\begin{equation}\label{eq almostflat}
\big((x-x_0)\cdot e_{x_0}-\e r\big)_+
\leq U(x)\leq
\big((x-x_0)\cdot e_{x_0}+\e r\big)_+
\quad\mbox{in $B_r(x_0)$},
\end{equation}
for some $e_{x_0}\in\mathbb S^{d-1}$. We may also decrease $r$, if necessary, so that the scale requirement in \cite[Theorem 8.1]{AltCaf} is satisfied.

Let $y_0\in\partial\{u>0\}\cap A_{7/2}$. By \eqref{eq;H}, there exists $x_0\in\partial\{U>0\}\cap D$ such that $|x_0-y_0|$ is arbitrarily small upon decreasing $\delta_0$. Combining \eqref{eq:U}, the convergence of the phases, and \eqref{eq almostflat}, we deduce that
\begin{equation}\label{tmp:twohp}
\big((x-y_0)\cdot e_{x_0}-C\e r\big)_+
\leq u(x)\leq
\big((x-y_0)\cdot e_{x_0}+C\e r\big)_+
\quad\mbox{in $B_{r/2}(y_0)$},
\end{equation}
where $C$ is a dimensional constant. Here we have chosen $\delta_0$ sufficiently small depending on $\e r$. Thus, after rescaling $B_{r/2}(y_0)$ to the unit ball, $u$ belongs to the flatness class required in \cite[Theorem 8.1]{AltCaf}, with flatness bounded by $C\e$. Choosing $\e$ sufficiently small, we conclude that $\partial\{u>0\}$ is a $C^{1,\alpha}$ hypersurface in a neighborhood of $y_0$. Since the estimates are uniform for $y_0\in\partial\{u>0\}\cap A_{7/2}$, it follows that $\partial\{u>0\}\cap A_3$ is a $C^{1,\alpha}$ hypersurface.

We can then upgrade this fact to uniform $C^\infty$ bounds and analyticity using a standard application of the Schauder estimates, the partial hodograph transform, and a bootstrapping argument; see \cite{Jerison,KN}. The corresponding quantitative analytic estimates for the hodograph equation give a uniform analytic radius. Hence, after decreasing $\delta_0$, the Cauchy--Kovalevskaya theorem provides unique harmonic continuations $\tilde u$ and $\tilde U$, defined in $\Omega_{2\mu}\cap A_{7/2}$ for some $\mu>0$ depending only on $U$ and on the constants in Definition \ref{d:weaksol}.

We finish this step by observing that
\begin{equation*}
\Delta(x\cdot\nabla\tilde U-\tilde U)
=
x\cdot\nabla(\Delta\tilde U)+\Delta\tilde U
=
0.
\end{equation*}
Since $x\cdot\nabla U-U=0$ in $\{U>0\}$ by the homogeneity of $U$, unique continuation gives
\begin{equation*}
x\cdot\nabla\tilde U-\tilde U=0
\end{equation*}
in $\Omega_{2\mu}\cap A_{7/2}$. Thus, $\tilde U$ is $1$-homogeneous.
    
\medskip

\noindent {\it Step 2:} We prove (4).

\medskip

The quantitative analytic estimates obtained in step 1, together with the fact that the extensions are defined in $\Omega_{2\mu}\cap A_{7/2}$, imply that
\begin{equation*}
\|\tilde w_i\|_{C^k(\Omega_\mu\cap A_3)}
\leq C(k,U,c_1,c_2)
\end{equation*}
for every $k\in\N$. Moreover, by \eqref{eq:U} and since $U$ is non-trivial, the norm $\|w_i\|_{L^\infty(A_4)}$ is bounded from below by a positive constant depending only on $U$. After changing the constant, this gives \eqref{eq:extu}.

We next prove the sign and distance estimates. Let $\nu_i=\nabla w_i$ denote the unit normal to $\partial\{w_i>0\}$ pointing towards $\{w_i>0\}$. By the uniform $C^2$ estimates and the $C^1$-closeness of the free boundaries, there exists a uniform tubular neighborhood of $\partial\{w_i>0\}\cap A_3$. After decreasing $\mu$, every point in $
(\Omega_\mu\setminus\{w_i>0\})\cap A_3$ belongs to this tubular neighborhood and can be written uniquely as
\begin{equation*}
x=y-t\nu_i(y),
\qquad
y\in\partial\{w_i>0\},
\quad t\geq0.
\end{equation*}
Since $\tilde w_i(y)=0$ and $\partial_{\nu_i}\tilde w_i(y)=1$, Taylor expansion gives
\begin{equation*}
\tilde w_i(y-t\nu_i(y))
=
-t+O(t^2).
\end{equation*}
Therefore, after decreasing $\mu$ once more $\tilde w_i<0$ in $(\Omega_\mu\setminus\{w_i>0\})\cap A_3$
and
\begin{equation*}
|\tilde w_i(x)|
\leq
C\operatorname{dist}(x,\partial\{w_i>0\})
\end{equation*}
in the tubular neighborhood. Away from this neighborhood, \eqref{eq:distest} follows from the uniform $L^\infty$ bound and the fact that the distance to the free boundary is bounded from below. This proves (4).

\medskip

\noindent {\it Step 3:} We prove (5).

\medskip

We argue by contradiction. Suppose that \eqref{tmp:compCk} fails for some $\e_0>0$ and $k\in\N$. Then there exists a sequence of solutions $u_m$ such that $\|u_m-U\|_{L^2(A_4)}\longrightarrow0$. After passing to a subsequence and fixing one of the two possible choices of $(w_{1,m},w_{2,m})$, such that
\begin{equation*}
\|\tilde w_{1,m}-w_{2,m}\|_
{C^k(\overline{\{w_{2,m}>0\}}\cap A_3)}
\geq\e_0.
\end{equation*}

Applying \eqref{eq:extu} with $k+1$ and using compactness, we may assume that the harmonic continuations converge in $C^k(\Omega_\mu\cap A_3)$. By \eqref{eq:U}, the restrictions of these continuations to any fixed open set compactly contained in $\{U>0\}\cap A_3$ converge to $U$. The unique continuation property then implies that their limits coincide with $\tilde U$ throughout $\Omega_\mu\cap A_3$.

On the other hand, in $\{w_{2,m}>0\}$ we have $w_{2,m}=\tilde w_{2,m}$. Therefore,
\begin{equation*}
\|\tilde w_{1,m}-w_{2,m}\|_
{C^k(\overline{\{w_{2,m}>0\}}\cap A_3)}
\leq
\|\tilde w_{1,m}-\tilde w_{2,m}\|_
{C^k(\Omega_\mu\cap A_3)}
\longrightarrow0,
\end{equation*}
which is a contradiction. This proves (5).

\medskip 

\noindent {\it Step 4:} We prove (6).

\medskip

We first prove graphicality. Let
$x\in\partial\{w_2>0\}\cap A_3$ and consider the function
\begin{equation*}
F(x,t):=
\tilde w_1\big(x+t\nabla w_2(x)\big).
\end{equation*}
By (5), after decreasing $\delta_0$ if necessary,
\begin{equation*}
\partial_tF(x,0)
=
\nabla\tilde w_1(x)\cdot\nabla w_2(x)
\end{equation*}
is uniformly bounded away from zero. Since $F(x,0)=v(x)$ is sufficiently small, the analytic implicit function theorem gives a unique analytic function $\eta$ on $\partial\{w_2>0\}\cap A_3$ such that
\begin{equation*}
\tilde w_1\big(x+\eta(x)\nabla w_2(x)\big)=0.
\end{equation*}
By the sign property proved in step 2, this zero set agrees with $\partial\{w_1>0\}$. The uniqueness in the implicit function theorem, together with the closeness of the two free boundaries, gives
\begin{equation*}
\partial\{w_1>0\}\cap A_2
=
\left\{
x+\eta(x)\nabla w_2(x):
x\in\partial\{w_2>0\}\cap A_3
\right\}\cap A_2.
\end{equation*}

We next prove \eqref{eq expansion eta}. Since $w_2(x)=0$, Taylor expansion gives
\begin{align}
0
&=
\tilde w_1\big(x+\eta(x)\nabla w_2(x)\big)\nonumber\\
&=
v(x)
+
\eta(x)\nabla\tilde w_1(x)\cdot\nabla w_2(x)
+
O(\eta(x)^2).
\label{tmpeq:veta}
\end{align}
Since the scalar product in the second term is uniformly bounded away from zero, \eqref{tmpeq:veta} first implies that $|\eta|\leq C|v|$ and then
\begin{equation*}
\eta(x)
=
-\frac{v(x)}
{\nabla\tilde w_1(x)\cdot\nabla w_2(x)}
+
O(v(x)^2),
\end{equation*}
which proves \eqref{eq expansion eta}.

It remains to derive the boundary condition. Since
\begin{equation*}
|\nabla\tilde w_1|^2
\big(x+\eta(x)\nabla w_2(x)\big)
=
|\nabla w_2|^2(x)
=
1,
\end{equation*}
Taylor expansion gives
\begin{align*}
|\nabla\tilde w_1|^2(x)-|\nabla w_2|^2(x)
={}&
-2\eta(x)
D^2\tilde w_1(x)
[\nabla\tilde w_1(x),\nabla w_2(x)]
+
O(\eta(x)^2).
\end{align*}
On the other hand,
\begin{equation*}
|\nabla\tilde w_1|^2-|\nabla w_2|^2
=
2\partial_\nu v+|\nabla v|^2.
\end{equation*}
Substituting \eqref{eq expansion eta}, we obtain
\begin{align*}
\partial_\nu v
-
\frac{
D^2\tilde w_1
[\nabla\tilde w_1,\nabla w_2]
}{
\nabla\tilde w_1\cdot\nabla w_2
}v
=
-\frac12|\nabla v|^2+O(v^2).
\end{align*}
This is precisely the boundary condition in \eqref{eq:linaroundu}, with $h$ given by \eqref{eq:h}. Since both $\tilde w_1$ and $w_2$ are harmonic in $\{w_2>0\}\cap A_3$, the interior equation follows immediately.
\end{proof}

With this ability to extend solutions, we are able to generalize \cite[Proposition 5.1]{DSJS}, to show that  if $u$ is sufficiently close to $U$ then the $L^\infty$ norm of their difference bounds the $C^{2,\alpha}$-norm of their difference.

\begin{lemma}\label{l:linfinitytoc2}
    Let $U$ be a regular cone and let $u$ be a non-degenerate weak solution in $A_4$. There exist constants $\eta_0 >0, C> 1$ depending on $U$ and on the constants in Definition \ref{d:weaksol}, such that if $\|u-U\|_{L^2(A_4)} \leq \eta_0$, then \begin{equation}\label{e:linfinitytoc2}
    \begin{aligned}
\|u-\tilde{U}\|_{C^2(\overline{\{u>0\}}\cap A_2)} \leq& C\|u-\tilde{U}\|_{L^\infty(\overline{\{u>0\}}\cap A_3)}\\
\|\tilde{u}-U\|_{C^2(\overline{\{U>0\}}\cap A_2)} \leq& C\|
\tilde{u}-U\|_{L^\infty(\overline{\{U>0\}}\cap A_3)}.
\end{aligned}
    \end{equation}
\end{lemma}

\begin{proof}
    In virtue of Lemma \ref{lemma: extension} we have that $\partial \{u > 0\}$ can be written as a smooth graph over $\partial \{U > 0\}$ inside of $A_3$.  Since $\partial \{U > 0\}$ is smooth, there exists some $\delta_0 >0$ such that for any $x_0 \in A_2\cap \partial \{u > 0\}$ there is a direction $e\in \mathbb S^{d-1}$ such that $$((x-x_0)\cdot e -\delta_0 \epsilon_1 )_+ \leq u(x), U(x) \leq ((x-x_0)\cdot e +\delta_0 \epsilon_1 )_+, \qquad \forall x\in B_{2\delta_0}(x_0).$$ Above, $\epsilon_1$ is a universal constant to be chosen later (proportional to the constant in \cite[Proposition 5.1]{DSJS}). 

    It follows that in $B_{\delta_0}(x_0)$ we have $\|u\|_{C^3(\overline{\{u > 0\}})}, \|U\|_{C^3(\overline{\{U > 0\}})} \leq C$ (which depends on $\delta_0$) and that there is a small $c\in (0,1)$ such that $$1-c \leq e\cdot \nabla u, e\cdot \nabla U \leq 1+c,$$ wherever the gradients are supported. Perhaps slightly enlarging $c$ (and shrinking $\delta_0$), we can guarantee that these estimates also hold for $\tilde{u}, \tilde{U}$. Let us focus on the case where we compare $u$ and $\tilde{U}$, as the other case is analogous. 

    Let $U_t(x) \equiv U(x + C\|u-\tilde{U}\|_{L^\infty(B_{2\delta_0}(x_0))}e)$, for some $C > 1$ large enough. Note that $U_t$ is also a classical solution to \eqref{e:bernoulliweak}. By the gradient estimates above we have
  $$((x-x_0)\cdot e -C\delta_0 \epsilon_1 )_+ \leq u(x)\leq U_t(x) \leq ((x-x_0)\cdot e +C\delta_0 \epsilon_1 )_+, \;\; \forall x\in B_{\delta_0}(x_0).$$
We can now invoke \cite[Proposition 5.1]{DSJS} to see that $$\|u-U_t\|_{C^2(B_{\delta_0/2}(x_0)\cap \overline{\{u>0\}})} \leq C\|u-U_t\|_{L^\infty(B_{\delta_0}(x_0))}.$$ Note that the $C > 0$ depends only on $\delta_0$ (which in turn depended on the $C^2$ character of $\partial \{U > 0\}$). By the Lipschitz bound on $U$ and the fact that we translated proportional to $\|u-U\|_{L^\infty}$ we immediately get \begin{equation}\label{e:c2linfinityalmost}
\|u-U_t\|_{C^2(B_{\delta_0/2}(x_0)\cap \overline{\{u>0\}})} \leq C\|u-\tilde{U}\|_{L^\infty(B_{2\delta_0}(x_0))}.
\end{equation}

With a standard covering argument we are finished with our proof once we establish $$\|\tilde{U} -U_t\|_{C^2(B_{\delta_0/2}(x_0)\cap \overline{\{u > 0\}})} \leq C\|u-\tilde{U}\|_{L^\infty(B_{2\delta_0}(x_0))}.$$ However, this estimate immediately follows from the definition of $U_t$ and the previously mentioned bound $\|\tilde{U}\|_{C^3(B_{\delta_0}(x_0))} \leq C$.
\end{proof}

Now we use a Moser-type argument to show that the $L^\infty$-norm of the difference between $u$ and $\tilde{U}$ is controlled by the $L^2$ norm of their difference. In particular, there is no concentration of mass on the boundaries of $\{u>0\}$ and the annulus $A_2$.

\begin{lemma}\label{lemmaL2Linfty} 
   Let $U$ be a regular cone and let $u$ be a non-degenerate weak solution in $A_4$.   There exists $C>0$ and $\delta>0$ depending on $U$ and on the constants in Definition \ref{d:weaksol} such that if $\|u-U\|_{L^2(A_4) } \leq \delta$, then
   \begin{equation}\label{eq:ctrldiff}   
     \Vert u-\tilde{U} \Vert_{C^2\big( A_2\cap \overline{\{u>0\}}\big)}\leq C\Vert u-\tilde{U} \Vert_{L^2\big( A_3\cap \{u>0\}\big)}.
    \end{equation}
\end{lemma}
\begin{remark}
	The estimate proven here is close in spirit to the standard  $L^2-L^\infty$ estimate, which we state in Lemma \ref{lemma lib} in the Appendix. However, we can not directly appeal to Lemma \ref{lemma lib}, because the boundary condition for $v= \tilde{U} - u$ depends on $v$ and $|\nabla v|^2$, see \eqref{eq:linaroundu}. In fact, the main difficulty of Lemma \ref{lemmaL2Linfty} lies in estimating these boundary terms.
\end{remark}

\begin{proof}
By virtue of Lemma \ref{l:linfinitytoc2} it suffices to show
  \begin{equation}\label{eq:ctrldiffinf}   
     \Vert u-\tilde{U} \Vert_{L^\infty\big( A_2\cap \{u>0\}\big)}\leq C\Vert u-\tilde{U} \Vert_{L^2\big( A_3\cap \{u>0\}\big)}.
    \end{equation}

% Set $v=u-\tilde{U}$ and $\Omega = \{u>0\}$. Without loss of generality, we  can assume that Indeed, by the arguments in Lemma \ref{lemma: extension}, if $\Vert u-\tilde{U} \Vert_{L^2\big( A_8\big)}\to 0$, then $\Vert u-\tilde{U} \Vert_{L^\infty\big( A_4\big)} \to 0$ and the free boundary of $u$ converges as graph in $C^k$ to the free boundary of $U$, whereas if $\Vert u-\tilde{U} \Vert_{L^2\big( A_8\big)}$ remains bounded from below the inequality is trivially satisfied by the uniform boundedness of $u$ and $\tilde{U}$ in $A_8$.\\

By our smallness assumption $\|u-U\|_{L^2(A_4) } \leq \delta$, upon taking $\delta$ sufficiently small, we can invoke Lemma \ref{lemma: extension} to guarantee that $v=\tilde{U}-u$ satisfies the boundary value problem \eqref{eq:linaroundu} in the domain $\Omega:=\{u>0\}$ and that $\Vert u-\tilde U\Vert_
{C^2(\overline{\{u>0\}}\cap A_3)}
\leq\delta_0$ with $\delta_0 \to 0^+$ as $\delta\to 0^+$. Let $p\geq 1$ and  $\varphi \in C_c^1(A_3)$ with $\varphi \geq 0$ and $\varphi =1$ in $A_2$.  Testing \eqref{eq:linaroundu} with $\varphi^2 v_+^p$ yields
\begin{eqnarray}\notag
    \int_{ \Omega} p|\nabla v_+|^2 v_+^{p-1}\varphi^2  &\leq&   -2\int_{\Omega}  v_+^p\varphi\nabla v_+\cdot\nabla \varphi \\ \notag
   &&+ \int_{\partial \Omega} v_+^{p+1} \varphi^2(Cv_++ h)\\ \label{eq:dirtest}
   &&+ C\int_{\partial \Omega} v_+^{p} |\nabla v|^2 \varphi^2.
\end{eqnarray}
We remark that in the last term involving $|\nabla v|^2$, we do not use the sign of $\frac12 |\nabla v|^2$ in the boundary condition of \eqref{eq:linaroundu}. Hence we could apply the same reasoning to $-v$ later, to obtain analogous estimates for $v_-$.
%is actually negative, but we keep it nonetheless since we would apply later the same reasoning for $-v$ when the sign is not in our favor.\\

From Young's inequality we have that
$$ 2v_+^p\varphi|\nabla v_+\cdot\nabla \varphi| \leq \frac{p}{2}|\nabla v_+|^2 v_+^{p-1}\varphi^2 + \frac{2}{p} |\nabla \varphi|^2 v_+^{1+p}$$
which combined with \eqref{eq:dirtest} implies
\begin{eqnarray*}\notag
    \int_{\Omega} \frac{p}{2}|\nabla v_+|^2 v_+^{p-1}\varphi^2  &\leq&   \frac{2}{p}\int_{\Omega}  |\nabla \varphi|^2 v_+^{1+p} \\ \notag
   &&+ \int_{\partial \Omega} v_+^{p+1} \varphi^2(Cv_++ h)\\ %\label{eq:dirtest}
   &&+ C\int_{\partial \Omega} v_+^{p} |\nabla v|^2 \varphi^2.
\end{eqnarray*}
Combining the previous inequality with the uniform boundedness of $v_+$ and $h$ and the identity $\left|\nabla 
\left(  v_+^\frac{1+p}{2}\right)\right|^2 = \Big(\frac{p+1}{2}\Big)^2|\nabla v_+|^2 v_+^{p-1}$, we get (recall that $p \geq 1$)
\begin{eqnarray}\notag
    \int_{ \Omega} \left|\nabla 
\left(  v_+^\frac{1+p}{2}\right)\right|^2\varphi^2   &\leq& C \left(\frac{1+p}{p}\right)^2 \int_{\Omega}  |\nabla \varphi|^2 v_+^{1+p}\\ \label{eq:moser2.5}
&&+Cp \int_{\partial \Omega} \left( \varphi^2 v_+^{p+1} + \varphi^2 v_+^p |\nabla v|^2 \right) %(v_+ + |\nabla v|^2) 
\end{eqnarray}
where $C>0$ is a constant depending on $U$. 

We now handle the boundary terms. For this we use the harmonicity of $u$ together with its boundary condition $\partial_\nu u = 1$ and the divergence theorem to deduce
\begin{eqnarray}\notag
    p\int_{\partial \Omega} \varphi^2v_+^{p+1}&=& p\int_{\partial \Omega} \varphi^2v_+^{p+1} \partial_\nu u\\ \notag
    &=& -p\int_{\Omega} \text{div}\left(  \varphi^2v_+^{p+1} \nabla u\right)\\ \label{eq boundaryp}
    &\leq& Cp \int_{\Omega} \left( p |\nabla v_+| v_+^p\varphi^2+ \varphi v_+^{p+1}|\nabla \varphi| \right).
\end{eqnarray}
Similarly,
\begin{align}
	 p\int_{\partial \Omega} \varphi^2v_+^{p}|\nabla v|^2  & = -p\int_{ \Omega} \text{div}\left( \varphi^2v_+^{p}|\nabla v|^2 \nabla u\right) \notag \\ 
    & \leq Cp \int_{\Omega} \left( p\varphi^2 v_+^{p-1} |\nabla v|^3 + \varphi^2 v_+^p |\nabla v| |D^2 v|  + \varphi v_+^p |\nabla v|^2 |\nabla \varphi| \right) \notag \\
    & \leq Cp \int_{\Omega} \left( p\varphi^2 v_+^{p-1} |\nabla v|^3 + \varphi^2 v_+^p |\nabla v| + \varphi v_+^p |\nabla v|^2 |\nabla \varphi| \right) \label{eq boundarygrad} % \\
 %   &\leq& Cp\int_{\Omega} \varphi^2|\nabla v_+|v_+^{p-1}( (p+1)|\nabla v_+|^2 + 2|D^2 v| v_+) \\\label{eq boundarygrad}
 %  && + Cp\int_{  \Omega}  2\varphi v_+^{p}|\nabla \varphi||\nabla v|^2\\\notag
 %   &\leq &  Cp\int_{ \Omega} \varphi|\nabla v_+|v_+^{p-1}( p\varphi|\nabla v_+|^2 + \varphi v_++v_+|\nabla \varphi||\nabla v|),
\end{align}
where we have used the uniform bounds in $|\nabla u|$ and $|D^2 v|$ guaranteed by Lemma \ref{lemma: extension}. 

Our next goal is to reabsorb the gradient terms introduced by the bounds \eqref{eq boundaryp} and \eqref{eq boundarygrad}. Starting with  \eqref{eq boundaryp}, we use Young's inequality to obtain the bound
\begin{eqnarray}\notag
    Cp^2 |\nabla v_+| v_+^p\varphi^2 &=& C\frac{2p^2}{p+1}\left|\nabla  \left(v_+^\frac{p+1}{2}\right)\right|v_+^\frac{p+1}{2}\varphi^2\\\label{eq:y1}
    &\leq& \frac{1}{4}\left|\nabla  \left(v_+^\frac{p+1}{2}\right)\right|^2\varphi^2+Cp^2v_+^{p+1}\varphi^2.
\end{eqnarray} 
In the estimate \eqref{eq boundarygrad} we find three gradient terms. For the second and third terms, we can use Young's inequality as before (as well as the uniform boundedness of $|\nabla v|$) to get 
\begin{align}
	Cp \, \varphi^2 v_+^p |\nabla v| & = C \frac{2p}{p+1} \varphi^2 \left|\nabla \left(v_+^{\frac{p+1}{2}} \right) \right| v_+^{\frac{p+1}{2}} \notag \\
	& \leq \frac14 \left|\nabla \left(v_+^{\frac{p+1}{2}} \right) \right|^2 \varphi^2 + C v_+^{p+1} \varphi^2,\label{eq:y20}
\end{align}
and
\begin{align}
	Cp \, \varphi|\nabla v_+|^2v_+^{p}|\nabla \varphi| & =   C\frac{2p}{p+1}\left|\nabla  \left(v_+^\frac{p+1}{2}\right)\right||\nabla v_+| v_+^\frac{p+1}{2}\varphi |\nabla \varphi| \notag \\ 
   &\leq \frac{1}{4}\left|\nabla \left(v_+^\frac{p+1}{2}\right)\right|^2\varphi^2+Cv_+^{p+1}|\nabla \varphi|^2.\label{eq:y2}
\end{align}
The first term in \eqref{eq boundarygrad}, namely $  p^2\varphi^2 |\nabla v_+|^3v_+^{p-1}$, can be bounded using the smallness assumption $|\nabla v|\leq \delta_0$ as follows
\begin{equation}\label{eq:y3}
    Cp^2|\nabla v_+|^3v_+^{p-1} \varphi^2 
    \leq C\delta_0 \left|\nabla  \left(v_+^\frac{p+1}{2}\right)\right|^2\varphi^2\leq \frac{1}{4}  \left|\nabla  \left(v_+^\frac{p+1}{2}\right)\right|^2\varphi^2,
\end{equation} 
where $\delta_0$ is taken so that $C\delta_0 \leq \frac{1}{4}$. We remark that the estimates \eqref{eq:y20}, \eqref{eq:y3} above use in a crucial way the fact that the boundary condition for $v$ depends \emph{quadratically} on $|\nabla v|$ in \eqref{eq:linaroundu}.
Lastly, observe that
\begin{equation}\label{eq:gradtogether}
    \left|\nabla 
\left( \varphi  v_+^\frac{1+p}{2}\right)\right|^2 \leq 2|\nabla \varphi|^2 v_+^{1+p} +  2\left|\nabla 
\left(  v_+^\frac{1+p}{2}\right)\right|^2\varphi^2.
\end{equation}
Therefore combining \eqref{eq:moser2.5}, \eqref{eq boundaryp}, \eqref{eq boundarygrad}, \eqref{eq:y1}, \eqref{eq:y20}, \eqref{eq:y2}, \eqref{eq:y3}, and \eqref{eq:gradtogether} we obtain the gradient estimate
\[ \int_{\Omega} \left|\nabla 
\left( \varphi v_+^\frac{1+p}{2}\right)\right|^2  \leq C \int_{\Omega} (|\nabla \varphi|^2+p \, \varphi|\nabla \varphi| +p^2\varphi^2) v_+^{1+p}.\]

On the other hand, since $\Omega$ has uniformly regular boundary in $A_3$ we can apply the Sobolev inequality to deduce 
\begin{equation}\label{eq:moser3}
   \left( \int_{\Omega} 
\left( \varphi^2 v_+^{p+1}\right)^{\kappa} \right)^\frac{1}{\kappa} \leq C \int_{\Omega} (|\nabla \varphi|^2+p\varphi|\nabla \varphi| +p^2\varphi^2) v_+^{1+p},
\end{equation}
with the exponent $\kappa \in \left(1, \min\{\frac{d}{d-2},2\} \right)$ fixed. Take $2\leq r<s\leq3$ and take $\varphi$ supported in $A_{s}$ with $\varphi \geq 0$ and $\varphi =1$ in $A_{r}$ and, furthermore, with $|\nabla \varphi|\leq \frac{C}{s-r}$. Under this choice of test function and setting $\xi =p+1\geq 2$ we can deduce from \eqref{eq:moser3} that
\begin{equation*}\label{eq:moser4}
   \left( \int_{A_{r}\cap \Omega} 
  v_+^{\xi\kappa} \right)^\frac{1}{\kappa} \leq \frac{C}{(s-r)^2} \int_{A_{s}\cap  \Omega} (1+\xi(s-r) +\xi^2(s-r)^2) v_+^{\xi}.
\end{equation*}
We define a series of powers $\xi_i = 2\kappa^{i-1}$ and radii  $r_i = 2 + \sum_{j=1}^{i-1} 2^{-j}$ from which we get the recurrence
\begin{equation}\label{eq:recurrence}
   \left( \int_{A_{r_i}\cap \Omega} 
  v_+^{\xi_{i+1}} \right)^\frac{1}{\kappa} \leq 4^i C \int_{A_{r_{i+1}}\cap  \Omega}  v_+^{\xi_i}.
\end{equation}
Here we have used the choice $\kappa<2$, which implies that $\xi_i(r_{i+1}-r_i) = \frac{2}{\kappa} (\frac{\kappa}{2})^i $ is bounded independent of $i$.
\eqref{eq:recurrence} can be rewritten as
\begin{equation}\label{tmp:moser4}
    \Vert v_+\Vert_{L^{\xi_{i+1}}\left( A_{r_i}\cap \Omega\right)} \leq  \left(4^i C \right)^\frac{1}{\xi_i} \Vert v_+\Vert_{L^{\xi_{i}}\left( A_{r_{i+1}}\cap \Omega\right)}.
\end{equation}
Iterating \eqref{tmp:moser4} and taking the limit as $i\to \infty$ finally yields
\[
 \Vert v_+\Vert_{L^{\infty}\left( A_2\cap \Omega\right)} \leq  C \Vert v_+\Vert_{L^{2}\left( A_3\cap \Omega\right)}.
\]

We obtain the corresponding inequality for $v_-$ by noticing that $-v$ satisfies an equation analogous to \eqref{eq:linaroundu}. 

\end{proof}

We end the section with two more estimates, both of which will be useful for us going forward. In the first we show that quantitative $C^2$ control on the difference between $u$ and $U$ implies quantitative $C^2$ control on the distance between the free boundaries. 

\begin{lemma}\label{l:c2oneta}
    Let $U$ be a regular cone and let $u$ be a non-degenerate weak solution in $A_4$.  Assume that $\|u-U\|_{L^2(A_4)} \leq \delta_0$ for $\delta_0 > 0$ sufficiently small depending on $U$ and on the constants in Definition \ref{d:weaksol}. 

   If $\partial \{u > 0\}$ is the graph of $\eta$ over $\partial \{U > 0\}$ (as above), then we can estimate
$$
\sup_{x\in \partial \{U > 0\}\cap A_{\frac{3}{2}}}
\left(|\eta(x)|+|\nabla \eta(x)|+|\nabla^2 \eta(x)|\right)
\leq C\|u-\tilde{U}\|_{C^2(\overline{\{u > 0\}}\cap A_2)},
$$
where $C>1$ depends on $U$ and on the constants in Definition \ref{d:weaksol}, and $\tilde U$ is the analytic extension introduced above.

Similarly, if $\partial \{U > 0\}$ is the graph of $\xi$ over $\partial \{u > 0\}$, then
$$
\sup_{x\in \partial \{u > 0\}\cap A_{\frac{3}{2}}}
\left(|\xi(x)|+|\nabla \xi(x)|+|\nabla^2 \xi(x)|\right)
\leq C\|\tilde{u}-U\|_{C^2(\overline{\{U > 0\}}\cap A_2)},$$ 
where $C > 1$ depends on $U$ and $\tilde{u}$ is the analytic extension as above.
\end{lemma}

\begin{proof}
 We prove the estimate for $\eta$. The estimate for $\xi$ follows mutatis mutandis.
    
Let $v=\tilde U-u$. We first obtain the $L^\infty$ bound on $\eta$. Given $x\in\partial\{U>0\}\cap A_{\frac32}$, set $y=x+\eta(x)\nu(x)\in\partial\{u>0\}$. Since $u(y)=0$, $\tilde U(x)=0$, and $\partial_\nu\tilde U(x)=1$, a Taylor expansion gives
$$v(y)=\tilde U(y) =\eta(x)+O(\eta(x)^2).$$
Therefore, provided $\delta_0$ is sufficiently small,
$$|\eta(x)|\leq C\|v\|_{L^\infty(\overline{\{u>0\}}\cap A_2)}.$$

We then use the fact that 
\begin{equation}\label{eq:vbc}
	u(x+\eta(x)\nu(x)) = 0, \quad \text{ for all } x\in \partial\{U>0\} \cap A_2.
\end{equation}
Let $e_1,\cdots e_{d-1}$ denote tangential unit vectors of the hypersurface $\partial\{U>0\}$ at $x$, such that $e_1, \cdots, e_{d-1}, e_d:=\nu(x) $ form an orthonormal basis.
Denote $y= y(x) = x + \eta(x) \nu(x)$. Differentiating \eqref{eq:vbc} in the $e_i$ direction (for any $i=1,\cdots,d-1$), we deduce that 
\begin{equation}\label{e:towardsetabounds} 
	0 = \partial_{e_i} u + \partial_{e_i} \eta(x) \, \nabla u \cdot \nu(x) + \eta(x) \sum_{j=1}^d \partial_{e_j} u \, \partial_{e_i} \nu^j(x).
	%(1+\eta \partial_{x_i}\nu^i)\partial_{x_i}u + \eta\partial_{x_j}u \partial_{x_i}\nu^j + \partial_{x_i}\eta \nu\cdot \nabla u
\end{equation} 
Note we are using the notation that $\nu = (\nu^1, \nu^2,\ldots, \nu^d)$ and all derivatives of $u$ are evaluated at $y\in \partial \{u > 0\}$.

By the qualitative convergence of $\partial \{u > 0\}$ to $\partial \{U > 0\}$ we have $\partial_{e_d} u(y) = \nabla u(y) \cdot \nu(x) \simeq 1$.
Since $\tilde{U}$ is one-homogeneous we have that $|\nabla_{\partial \{U > 0\}} \nu|(x)\leq C|x|^{-1}$ and that $|\nabla_{\mathbb R^d} \tilde{U}|(x), |x||\nabla_{\mathbb R^d}^2 \tilde{U}|(x) \leq C$. So we can estimate for each $j=1, \cdots, d-1$
$$\begin{aligned}
	|\partial_{e_j}u|(y) \leq& |\partial_{e_j} \tilde{U}|(y) + |\nabla v|(y) \\ \leq& |\eta(x)| \sup_{z\in B_{2|\eta(x)|}(y)} \|\nabla^2 \tilde{U}\|(z) +|\partial_{e_j} U|(x)+|\nabla v|(y) \\\leq& C|x|^{-1}|\eta|(x) + |\nabla v|(y)\\
\leq& C|x|^{-1}\|v\|_{L^\infty} + |\nabla v(y)|,\end{aligned}$$ 
where in the third line, we have used the fact that $U \equiv 0$ on $\partial \{U > 0\}$ and $e_j$ is tangent to $\partial\{U>0\}$.  
Putting all this back into \eqref{e:towardsetabounds} gives the desired bound on $|\partial_{e_i} \eta|$. 

The estimate on the Hessian of $\eta$ follows similarly. 
\end{proof}

Our final lemma is a local estimate of the energy. It says that the difference of the energy $W(U,1)-W(u,1)$, (recall \eqref{eq weissenergy}), is controlled by the $C^2$-norm of $u-U$ near $\partial B_1$. Here we are inspired by Leon Simon, in particular \cite[Equation (7.14)]{Simon}.

\begin{lemma}\label{l:simontrick}
    Let $U$ be a regular cone and let $u$ be a non-degenerate weak solution in $B_4$.  Assume that $\|u-U\|_{L^2(A_4)} \leq \delta_0$ for $\delta_0 >0$ sufficiently small depending on $U$ and on the constants in Definition \ref{d:weaksol}. We have the estimate 
    \begin{equation}\label{e:lstrick}
    	\begin{aligned}
\left|W(U,1)-W(u,1)\right|
&\leq  C\left(  \|\eta\|_{C^2(A_{3/2})}^2
+ \|u-\tilde U\|_
{C^2(A_2\cap\overline{\{u>0\}})}^2
\right)\\
    		& \leq C\|u-\tilde{U}\|_{L^2(A_3\cap \overline{\{u > 0\}})}^2,
    \end{aligned}
    \end{equation}
    where $\tilde{U}$ and $\eta$ are as in Lemma \ref{l:c2oneta} and $C > 0$ depends only on $U$ and the dimension.
\end{lemma}

\begin{proof}
Let $\bar{u}= |x|u(x/|x|)$, namely the one-homogeneous extension of $u|_{\partial B_1}$. By the equipartition of energy in the radial and spherical directions (see, e.g. \cite[Lemma 9.8]{Velichkov}) we have that 
$$W(U, 1) - W(u, 1) = W(U, 1) -W(\bar{u}, 1)  + \frac{1}{d} \int_{\partial B_1} (x\cdot \nabla u - u)^2 d\mathcal{H}^{d-1}.$$ 
To estimate the quantity on the left, we need only to estimate each quantity on the right. Since $\tilde{U}$ is 1-homogeneous we have that 
\begin{align*}
	\int_{\partial B_1} (x\cdot \nabla u - u)^2 d\mathcal{H}^{d-1} & = \int_{\{u > 0\}\cap \partial B_1} \left(x\cdot \nabla (u-\tilde{U}) -(u-\tilde{U}) \right)^2\, d\mathcal{H}^{d-1} \\
	& \leq C\|u-\tilde{U}\|^2_{C^1(\overline{\{u > 0\}}\cap A_2)}.
\end{align*}

We now focus our attention on the term $W(U,1) - W(\bar{u},1)$. Since both functions are one-homogeneous, estimating this term reduces to analyzing the behavior of their restrictions on the sphere. For ease of notation we write $\Sigma_w = \{w > 0\}\cap \mathbb S^{d-1}$ for any continuous function $w$ and we denote by $\lambda^D_i(\Sigma_w)$ the $i$-th Dirichlet eigenvalue of the Laplace-Beltrami operator on $\Sigma_w \subset \mathbb S^{d-1}$. Furthermore, we define the operator 
$$\mathcal F(w):= \frac{\kappa_0^2}{d}(\lambda^D_1(\Sigma_w) - (d-1))+ \frac{1}{d}\mathcal H^{d-1}(\Sigma_w),$$
where $\kappa_0 = \Vert U\Vert_{L^2(\mathbb S^{d-1})}$. 
Recall that by the one-homogeneity and harmonicity of $U$ we have that $$W(U,1) = \mathcal F(U) = \frac{1}{d}\mathcal H^{d-1}(\Sigma_U).$$

On the other hand we can expand $u|_{\partial B_1}$ as 
 $$u|_{\partial B_1}(\theta) = \sum c_i\phi^D_i(\theta),$$ where $\phi^D_i$ is an $i$th Dirichlet eigenfunction of $\Sigma_u$ with unit $L^2$-norm, i.e. $\int_{\Sigma_u} |\phi_i^D|^2 d\mathcal{H}^{d-1} = 1 $. 
 Then we can compute that 
\begin{align}\notag
W(\bar u,1)
={}&
\mathcal F(\bar u)
+\frac{1}{d}(c_1^2-\kappa_0^2)
\big(\lambda_1^D(\Sigma_u)-(d-1)\big)\\
\label{eqWu}
&+
\frac{1}{d}
\sum_{i>1}c_i^2
\big(\lambda_i^D(\Sigma_u)-(d-1)\big).
\end{align}

Since $U$ is a critical point of $\mathcal{F}$, it follows from applying a Taylor expansion (see \cite[Appendix A]{EnSVEpi}\footnote{Note the appendix is written for minimizers but the justification of the Taylor series expansion and that $\delta \mathcal{F}(U) = 0$ uses only the smoothness of the free boundary and that $U$ solves \eqref{e:bernoulliweak}.}) that 
\begin{equation}\label{e:esvest} \left|\mathcal F(\bar{u}) - \mathcal F(U)\right| \leq C\|\eta\|^2_{C^2(A_\frac 32)},\end{equation} 
as long as $\delta_0$ is taken small enough. We are thus left to estimate the last two terms on the right hand side of \eqref{eqWu}. We do this in the following series of claims.\\

\medskip

\noindent{\it Claim 1:} $u$ satisfies
\begin{equation}\label{eqsheeleq}
    -\Delta_{\mathbb{S}^{d-1}} u = (d-1) u +E \quad \mbox{on $\Sigma_u$,}
\end{equation}
with $\Vert E\Vert_{L^\infty(\Sigma_u)} \leq C\|u-\tilde{U}\|_{C^2(A_2\cap \overline{\{u > 0\}})}$.\\

\medskip

\noindent{ \it Proof:} We note that $$\|u-\tilde{U}\|_{C^2(\overline{\{u > 0\}}\cap A_2)} \geq \|\Delta_{\mathbb S^{d-1}} (u-\tilde{U})\|_{L^\infty(\Sigma_u)} = \|(d-1)(u-\tilde{U}) + E\|_{L^\infty(\Sigma_u)}.$$ So by the triangle inequality we have the desired result. \\

\noindent{\it Claim 2:} We prove
\begin{equation}\label{eqspecbd}
    |\lambda_1^D(\Sigma_u)-(d-1)|^2 +\sum_{i > 1} c_i^2[\lambda_i^D(\Sigma_u)-(d-1)] \leq C\|u-\tilde{U}\|^2_{C^2(A_2\cap \overline{\{u > 0\}})}
\end{equation}
and $\lambda_i^D(\Sigma_u)-(d-1) \geq \frac{1}{C(U)}$ for $i\geq 2$.

\medskip

\noindent{ \it Proof:} Since $u = \sum_{i=1}^\infty c_i \phi_i^D$, by plugging this representation, understood as a convergent series in $L^2(\Sigma_u)$, into \eqref{eqsheeleq} we get
\begin{equation*}
     \sum_{i \geq 1} c_i[\lambda_i^D(\Sigma_u)-(d-1)]\phi_i^D =  E,
\end{equation*}
which, after taking $L^2$ norms on both sides and exploiting the orthogonality of the $\phi_i^D$ yields
\begin{equation}\label{eqbeauON}
    c_1^2[\lambda_1^D(\Sigma_u)-(d-1)]^2+ \sum_{i > 1} c_i^2[\lambda_i^D(\Sigma_u)-(d-1)]^2 \leq C\|u-\tilde{U}\|^2_{C^2(A_2\cap \overline{\{u > 0\}})}.
\end{equation}
On the other hand, by taking $\delta_0$ small enough, we have that $c_1 \geq \frac{1}{C(U)}$ with $C(U)$ depending only on $\kappa_0$. By the stability of the Laplacian spectrum under domain perturbations, the same smallness implies that $\lambda_i^D(\Sigma_u)-(d-1) \geq \frac{1}{C(U)}$ for $i \geq 2$ with $C(U)$ depending in this case on the spectral gap $\lambda_2^D(\Sigma_U)-\lambda_1^D(\Sigma_U)> 0$. These observations combined with \eqref{eqbeauON} yield  \eqref{eqspecbd}.\\

\medskip

\noindent{\it Claim 3:}
$|c_1-\kappa_0|
\leq C\|u-\tilde{U}\|_{C^2(A_2\cap\overline{\{u>0\}})}
+C\|\eta\|_{L^\infty(A_{\frac32})}$.

\medskip

\noindent{ \it Proof:} By the reverse triangle inequality we have that $|c_1-\kappa_0| \leq \|c_1\phi_1^D-U\|_{L^2(\mathbb S^{d-1})}$. On the other hand, by Claim 2, we have that
\begin{align}\notag
    \|c_1\phi_1^D-U\|_{L^2(\mathbb S^{d-1})}&\leq \|c_1\phi_1^D-u\|_{L^2(\mathbb S^{d-1})} + \|u-U\|_{L^2(\mathbb S^{d-1})} \\\notag
    & \leq \left(\sum_{i\geq 2} c_i^2 \right)^{\frac12} + \|u-U\|_{L^2(\mathbb S^{d-1})}\\ \label{eqci}
    & \leq C\|u-\tilde{U}\|_{C^2(A_2\cap \overline{\{u > 0\}})} + \|u-U\|_{L^2(\mathbb S^{d-1})}.
\end{align}
On the other hand, we notice that
\begin{align}\notag
   \|u-U\|_{L^2(\mathbb S^{d-1})} &\leq \|u-\tilde{U}\|_{L^2(\Sigma_u)} + \|\tilde{U}\|_{L^2(\{U > 0\}\Delta \Sigma_u)} \\ \label{eqclaim3}
   &\leq \|u-\tilde{U}\|_{L^2(\Sigma_u)} + C\|\eta\|_{L^\infty(A_2)},
\end{align}
where the last inequality follows from \eqref{eq:distest}. Finally, we conclude the proof of the claim by combining \eqref{eqci} and \eqref{eqclaim3}.\\

\medskip

The combination of the three claims proved above with \eqref{eqWu} and \eqref{e:esvest} yields the first inequality in \eqref{e:lstrick}, while the second follows directly from Lemma \ref{lemmaL2Linfty} and Lemma \ref{l:c2oneta}.

\end{proof}

\section{Rates of convergence of solutions asymptotic to integrable cones: Proof of Theorem \ref{thm AA}}\label{sec: AAresult}

This section is devoted to the proof of the following generalization of Theorem \ref{thm AA}. We use the notation
\begin{equation*}
W(u,0):=\lim_{r\downarrow0}W(u,r)
\qquad\mbox{and}\qquad
W(u,\infty):=\lim_{r\to\infty}W(u,r)
\end{equation*}
for the Weiss energies of $u$ at $0$ and at infinity, whenever the corresponding limit is defined. We retain the notation $A_{s,r}:=B_r\setminus\overline{B_s}$ and $A_t:=A_{1/t,t}$ for $t>1$, introduced in the previous section, and write $u_r(x):=u(rx)/r$ for the rescalings of $u$.

Our main result is an ``$L^2$-excess decay" theorem at regular cones which are integrable through rotations. 

\begin{theorem}\label{thm AArestated}
Let $D\subset\R^d$ be an open set with $B_8\subset D$, and let $u$ be a non-degenerate weak solution in $D$. Let $U$ be a regular cone which is integrable through rotations. There exist $\delta_0>0$, $\alpha\in(0,1)$, and $C>0$, depending on $U$ and on the constants in Definition \ref{d:weaksol}, such that the following statements hold.

\begin{itemize}
    \item Suppose that  $0\in\partial\{u>0\}$ and that     \begin{equation}\label{eq blowupdens}
     W(u,0) \geq W(U,1).
    \end{equation}
    If $\Vert u-U\Vert _{L^2(A_8)}\leq\delta_0$,
then there exists $R\in O(d)$ such that
\begin{equation}\label{eq:rate0sec4}\Vert u(rx)/r - U(R\,x)\Vert _{L^2(A_2)} \leq Cr^\alpha {\|u - U\|_{L^2(A_8 ) }}, \end{equation}
for all $0<r<1$.
\item Suppose that $D=\R^d$ and that
\begin{equation}\label{eq blowdowndens}
     W(u,\infty) \leq W(U,1).
    \end{equation}
If $\|u-U\|_{L^2(A_8)}\leq\delta_0$, then there exists $R\in O(d)$ such that  
\begin{equation}\label{eq:ratesec4} \|u(rx)/r - U(R\, x)\|_{L^2(A_2)} \leq Cr^{-\alpha} \|u - U \|_{L^2(A_8 ) }, \end{equation} 
for all $ r > 1$.
\end{itemize}
\end{theorem}

\begin{remark}
Theorem \ref{thm AArestated} replaces the hypotheses in Theorem \ref{thm AA} that require $U$ to arise as a blow-up or blow-down by the respective one-sided density conditions \eqref{eq blowupdens} and \eqref{eq blowdowndens}. Indeed, the monotonicity of the Weiss energy implies that
\begin{equation*}
W(u,r)\geq W(U,1)
\end{equation*}
for every $r>0$ in the blow-up case, while
\begin{equation*}
W(u,r)\leq W(U,1)
\end{equation*}
for every $r>0$ in the blow-down case. 

The one-sided density conditions are analogous to the bounds arising in the no-holes framework of Simon \cite{SimonCylinders} and, more directly, to the hypothesis in \cite[Theorem 1.2]{CESpoly}. Under these weaker hypotheses, we still obtain convergence to a unique profile, although this profile may be a rotation of $U$. If $U$ is already known to occur as a blow-up or blow-down, uniqueness gives $U(R\,\cdot)=U$.
\end{remark}
  
Theorem \ref{thm AArestated}, in turn, will follow from a symmetric growth-decay estimate. To state it, we introduce an additional quantity. Let $u$ and $U$ be as in Theorem \ref{thm AArestated}. For $r>0$, we define the monotonicity drop
\begin{equation*}
I(r):=W(u,8r)-W(u,r/8).
\end{equation*}

\begin{theorem}\label{thm 3annulil}
Let $D\subset\R^d$ be an open set, let $u$ be a non-degenerate weak solution in $D$, and let $U$ be a regular cone which is integrable through rotations. There exist $\delta_2>0$, $\alpha\in(0,1)$, and $C>0$, depending on $U$ and on the constants in Definition \ref{d:weaksol}, with the following property.

Let $r_1\in(0,1/8]$, $r_2\geq8$, and assume that $B_{8r_2}\subset D$. Suppose that
\begin{equation}\label{eq initsmall}
\|u-U\|_{L^2(A_8)}\leq\delta_2
\end{equation}
and
\begin{equation}\label{eq monosmallness}
W(u,r_2)-W(u,r_1)\leq\delta_2^2.
\end{equation}
Then, for every $r\in[8r_1,r_2/8]$,
\begin{equation}\label{eq 3annuli}
I(r)\leq
C\left(\frac{r_1}{r}\right)^{2\alpha}I(8r_1)
+
C\left(\frac{r}{r_2}\right)^{2\alpha}I(r_2/8).
\end{equation}
% Moreover, for every $8r_1\leq s\leq r\leq t\leq r_2/8$,
% \begin{equation}\label{eq 3annulil2}
%     \E_2(r)\leq
%     C\left(\frac{s}{r}\right)^\alpha\E_4(s)
%     +C\left(\frac{r}{t}\right)^\alpha\E_4(t)
% \end{equation}
% for every $8r_1\leq s\leq r\leq t\leq r_2/8$. Finally, for every
% $8r_1\leq s\leq\rho_1\leq\rho_2\leq t\leq r_2/8$,
% \begin{equation}\label{eq Dini}
%     \int_{\rho_1}^{\rho_2}
%     \|\partial_\rho u_\rho\|_{L^2(\partial B_1)}d\rho
%     \leq
%     C\left(\frac{s}{\rho_1}\right)^\alpha\E_4(s)
%     +C\left(\frac{\rho_2}{t}\right)^\alpha\E_4(t).
% \end{equation}
\end{theorem}

\begin{remark}\label{r:3annuli}
Theorem \ref{thm 3annulil} is a nonlinear three-annuli estimate: closeness to the rotation orbit of $U$ at one scale, together with a small Weiss drop, propagates through the full range of intermediate scales and yields a symmetric growth-decay alternative. Analogous estimates appear in the minimal-hypersurface setting; see, for example, \cite{EdelenMinter}. The conclusion is also closely related to the growth-decay estimate obtained from the symmetric epiperimetric inequality in \cite{EdSVLogEpi}, which can be viewed as a form of logarithmic convexity of the energy. We remark again that neither $u$ nor $U$ is assumed to be minimizing or stable, so even if our conclusions are formally similar to those in \cite{EdSVLogEpi}, the underlying mechanism and proofs are different. 
\end{remark}

Our first lemma shows that if $u$ is sufficiently close to a regular cone $U$ which is integrable through rotations, then the drop in monotonicity controls the distance of $u$ to the rotation orbit of $U$. A qualitative version of this follows by compactness (see, for example, \cite[Lemma 3.1]{EdEng}); here we require a quantitative version of it.  We remark that this is the key place where we use the assumption that $U$ is integrable through rotations, see Definition \ref{def: intcone}. Indeed, modulo rotations, integrability rules out nontrivial $1$-homogeneous solutions of the linearized problem and yields the \L{}ojasiewicz-type inequality \eqref{eq quadestimate} with the optimal exponent: the squared distance to the rotation orbit of $U$ is controlled linearly by the Weiss monotonicity drop. 

For $t>1$, we define the excess of $u$ from the rotation orbit of $U$ on $A_t$ by
\begin{equation*}
\E_t(r):=
\inf_{R\in O(d)}
\|u_r-U(R\,\cdot)\|_{L^2(A_t)}
=
\frac{1}{r^{d/2+1}}
\inf_{R\in O(d)}
\|u-U(R\,\cdot)\|_{L^2(A_{\frac rt,tr})}.
\end{equation*}

\begin{lemma}\label{lemma cone approx}
Let $U$ be a regular cone which is integrable through rotations. There exist $\delta_1>0$ and $C_1>1$, depending on $U$ and on the constants in Definition \ref{d:weaksol}, such that the following holds. Let $r>0$ and let $u$ be a non-degenerate weak solution in $B_{8r}$ satisfying
\begin{equation}\label{e:closeatinfinity}
\|u_r-U\|_{L^2(A_8)}\leq\delta_1.
\end{equation}
Then
\begin{equation}\label{eq quadestimate}
\E_2(r)^2
\leq
C_1\bigl(W(u,4r)-W(u,r/4)\bigr).
\end{equation}
\end{lemma}

\begin{remark}\label{rmk:tildeU}
The proof gives a slightly stronger estimate, which is the form that we will use in the remainder of this section. Let $R_r\in O(d)$ attain the infimum in the definition of $\E_2(r)$. In view of \eqref{e:closeatinfinity}, we may assume that $R_r$ is sufficiently close to the identity. Define $v_r(x):=u_r(R_r^{-1}x)$. Then
\begin{equation*}
\|v_r-U\|_{L^2(A_2)}
=
\|u_r-U(R_r\,\cdot)\|_{L^2(A_2)}
=
\E_2(r).
\end{equation*}

Let $\eta_r$ be the normal graph function parametrizing $\partial\{v_r>0\}$ over $\partial\{U>0\}$ in $A_3$, and recall the notation $\Omega_\mu$ from Lemma \ref{lemma: extension}. Set
\begin{equation*}
\mu_r:=
C\|\eta_r\|_{L^\infty(\partial\{U>0\}\cap A_3)},
\end{equation*}
where $C=C(U)$ is chosen sufficiently large. After decreasing $\delta_1$, if necessary, we have $\{v_r>0\}\cap A_2\subset\Omega_{\mu_r}\cap A_2$, and $\mu_r$ is smaller than the extension radius supplied by Lemma \ref{lemma: extension}.

Extending $v_r$ by zero outside $\{v_r>0\}$, we define
\begin{equation*}
\widetilde{\E}_2(r):=
\|v_r-\tilde U\|_{L^2(\Omega_{\mu_r}\cap A_2)}.
\end{equation*}
The stronger conclusion proved below is
\begin{equation}\label{eq enhancedquadestimate}
\widetilde{\E}_2(r)^2
\leq
C_1\bigl(W(u,4r)-W(u,r/4)\bigr).
\end{equation}

To see that \eqref{eq enhancedquadestimate} implies \eqref{eq quadestimate}, notice that Lemma \ref{lemma: extension} gives $\tilde U\leq0$ in $(\Omega_{\mu_r}\setminus\{U>0\})\cap A_2$. Consequently,
\begin{equation*}
|v_r-\tilde U|
\geq
|v_r-U|
\quad\mbox{in $\Omega_{\mu_r}\cap A_2$}.
\end{equation*}
Moreover, both $v_r$ and $U$ vanish outside $\Omega_{\mu_r}\cap A_2$. Therefore,
\begin{equation*}
\E_2(r)
=
\|v_r-U\|_{L^2(A_2)}
\leq
\widetilde{\E}_2(r),
\end{equation*}
and hence \eqref{eq enhancedquadestimate} implies \eqref{eq quadestimate}.

We have kept the simpler estimate \eqref{eq quadestimate} as the statement of the lemma because it does not involve the auxiliary harmonic continuation and may be useful in other contexts. It is precisely an optimal-exponent \L{}ojasiewicz-type inequality for the distance to the rotation orbit of $U$.
\end{remark}

\begin{proof}
Since the statement of the lemma is scale invariant, upon replacing $u$ by $u_r$ we may assume that $r=1$. We prove the stronger estimate \eqref{eq enhancedquadestimate}.

Suppose, for the sake of contradiction, that \eqref{eq enhancedquadestimate} fails. After rotating each element of the contradicting sequence by a rotation attaining the infimum in the definition of $\E_2(1)$, and relabeling the resulting sequence, we obtain non-degenerate weak solutions $u_k$ such that
\begin{equation}\label{eq proximityU}
\|u_k-U\|_{L^2(A_8)}\longrightarrow0,
\end{equation}
and
\begin{equation}\label{eq:optA}
\|u_k-U\|_{L^2(A_2)}
=
\inf_{R\in O(d)}
\|u_k-U(R\,\cdot)\|_{L^2(A_2)}.
\end{equation}
Let $\eta_k$ be the normal graph function parametrizing $\partial\{u_k>0\}$ over $\partial\{U>0\}$ in $A_3$, set
\begin{equation*}
\mu_k:=
C\|\eta_k\|_{L^\infty(\partial\{U>0\}\cap A_3)},
\qquad
\Omega_k:=\Omega_{\mu_k},
\end{equation*}
and define $d_k:=
\|u_k-\tilde U\|_{L^2(\Omega_k\cap A_2)}.$

The contradiction hypothesis reads
\begin{equation}\label{eq contrradialnew}
d_k^2
\geq
k\bigl(W(u_k,4)-W(u_k,1/4)\bigr).
\end{equation}
In particular, $d_k>0$.

By the optimality condition \eqref{eq:optA}, for any skew-symmetric matrix $M$, the function $$t\mapsto\|u_k-U(\exp(tM)\,\cdot)\|_{L^2(A_2)}^2$$
attains a minimum at $t=0$. It follows that
\begin{equation}\label{eq ocond}
\int_{A_2}(u_k-U)\nabla U\cdot M[x]\,dx=0.
\end{equation}

We extend $u_k$ by zero outside $\{u_k>0\}$. Since $u_k$ has zero trace on its free boundary, this extension belongs to $W^{1,2}(\Omega_k\cap A_4)$. Moreover, $\tilde U$ is $1$-homogeneous, and hence $x\cdot\nabla\tilde U-\tilde U=0$. Applying the Poincar\'e-type inequality \eqref{e:MNpoinc} on the cone $\Omega_k$, we obtain
\begin{align}
\int_{\Omega_k\cap A_4}(u_k-\tilde U)^2\,dx
&\leq
C\int_{\Omega_k\cap A_4}
(x\cdot\nabla u_k-u_k)^2\,dx
+C\int_{\Omega_k\cap A_2}(u_k-\tilde U)^2\,dx \notag\\
&\leq
C\bigl(W(u_k,4)-W(u_k,1/4)\bigr)+Cd_k^2
\leq Cd_k^2.
\label{eq:largecontrol}
\end{align}
Here we used \eqref{eq contrradialnew} in the last inequality.

Applying the argument of Lemma \ref{lemmaL2Linfty}, with $A_3$ in the conclusion and $A_4$ on the right hand side, together with the corresponding local graph estimates from Lemma \ref{l:c2oneta}, we obtain
\begin{eqnarray}\notag
\|u_k-\tilde U\|_{C^2(A_3\cap\overline{\{u_k>0\}})}
+
\|\eta_k\|_{C^2(\partial\{U>0\}\cap A_2)}
+
\|\eta_k\|_{L^\infty(\partial\{U>0\}\cap A_3)}
&\leq&
C\|u_k-\tilde U\|_{L^2(\Omega_k\cap A_4)}\\ \label{eqLinftycontrol}
&\leq& Cd_k.
\end{eqnarray}
In particular, $\mu_k\leq Cd_k$.

Set $w_k:=(u_k-\tilde U)/d_k$ on $\Omega_k\cap A_3$. Our goal is to show that $w_k$ converges to a nontrivial $1$-homogeneous function \(w\) solving \eqref{eq:lineareq}. By \eqref{eqLinftycontrol}, the functions $w_k$ are uniformly bounded in $C^2(\overline{\{u_k>0\}}\cap A_3)$. Thus, we obtain a harmonic function $w$ such that $w_k\to w$ smoothly on compact subsets of $\overline{\{U>0\}}\cap A_3$.

We claim that no $L^2$ mass accumulates near $\partial\{U>0\}$. For $\epsilon>0$, let
\begin{equation*}
\Gamma_\epsilon:=
\left\{
x\in A_2:
\operatorname{dist}(x,\partial\{U>0\})<\epsilon
\right\}.
\end{equation*}
In $\{u_k>0\}\cap A_2$, \eqref{eqLinftycontrol} gives $|w_k|\leq C$. On the other hand, graphical convergence and the definition of $\Omega_k$ imply
\begin{equation*}
\bigl(\Omega_k\setminus\{u_k>0\}\bigr)\cap A_2
\subset
\left\{
\operatorname{dist}(x,\partial\{U>0\})\leq C\mu_k
\right\}.
\end{equation*}
Consequently, by \eqref{eq:distest} and the estimate $\mu_k\leq Cd_k$, we also have $|w_k|\leq C$ in $(\Omega_k\setminus\{u_k>0\})\cap A_2$. Therefore,
\begin{equation}\label{eq:thinbd}
\int_{\Omega_k\cap\Gamma_\epsilon}w_k^2\,dx
\leq C|\Gamma_\epsilon|.
\end{equation}
Since $|\Gamma_\epsilon|\to0$ as $\epsilon\to0$, it follows that
\begin{equation}\label{eq nonacum}
\lim_{\epsilon\to0^+}
\limsup_{k\to\infty}
\|w_k\|_{L^2(\Omega_k\cap\Gamma_\epsilon)}
=0.
\end{equation}
Thus, $w_k\to w$ strongly in $L^2(\{U>0\}\cap A_2)$. Since $\|w_k\|_{L^2(\Omega_k\cap A_2)}=1$, we conclude that
\begin{equation*}
\|w\|_{L^2(\{U>0\}\cap A_2)}=1.
\end{equation*}
In particular, $w\not\equiv0$.

By the Weiss monotonicity formula \eqref{eq weissder}, the $1$-homogeneity of $\tilde U$, and \eqref{eq contrradialnew}, we have
\begin{equation}\label{eq conradial2}
\int_{\Omega_k\cap A_3}
(x\cdot\nabla w_k-w_k)^2\,dx
\leq
\frac{C}{d_k^2}
\bigl(W(u_k,4)-W(u_k,1/4)\bigr)
\leq\frac{C}{k}.
\end{equation}
Passing to the limit on compact subsets of $\{U>0\}\cap A_2$, we obtain $x\cdot\nabla w-w=0$. Thus, $w$ is $1$-homogeneous.

It remains to show that $w$ satisfies the boundary condition in \eqref{eq:lineareq}. Let $z_k:=u_k-\tilde U=d_kw_k$ in $\{u_k>0\}\cap A_2$. By \eqref{eq:linaroundu}, applied with $\tilde w_1=\tilde U$ and $w_2=u_k$, the function $z_k$ satisfies
\begin{equation*}
\partial_{\nu_k}z_k+h_kz_k
=
\frac12|\nabla z_k|^2+O(z_k^2)
\quad\mbox{on $\partial\{u_k>0\}\cap A_2$},
\end{equation*}
where $\nu_k=\nabla u_k$ is the inner unit normal. Consequently, for every $\psi\in C_c^\infty(A_{2})$,
\begin{align}
\int_{\{u_k>0\}}\nabla w_k\cdot\nabla\psi\,dx
-
\int_{\partial\{u_k>0\}}h_kw_k\psi\,d\mathcal H^{d-1}
&=
-\frac{d_k}{2}
\int_{\partial\{u_k>0\}}|\nabla w_k|^2\psi\,d\mathcal H^{d-1}
\notag\\
&\quad+
d_k\int_{\partial\{u_k>0\}}O(w_k^2)\psi\,d\mathcal H^{d-1}.
\label{eq:BCwk}
\end{align}
By \eqref{eqLinftycontrol} and the uniform area bounds for the free boundaries, the right hand side is bounded in absolute value by $Cd_k\|\psi\|_{L^\infty}$.

Using the normal graph parametrizations, we construct diffeomorphisms from $\Omega=\{U>0\}$ to $\{u_k>0\}$ in a neighborhood of $A_{2}$ such that $\Phi_k\to I$ in $C^2$ and
\begin{equation*}
\Phi_k\bigl(\partial\Omega\cap A_{2}\bigr)
=
\partial\{u_k>0\}\cap A_{2}.
\end{equation*}
Set $\widehat w_k:=w_k\circ\Phi_k$. Pulling \eqref{eq:BCwk} back by $\Phi_k$, we obtain
\begin{equation*}
\int_{\Omega}G_k\nabla\widehat w_k\cdot\nabla\varphi\,dx
-
\int_{\partial\Omega}
(h_k\circ\Phi_k)\widehat w_k\varphi J_k\,
d\mathcal H^{d-1}
=
d_k\,O(\|\varphi\|_{C^1})
\end{equation*}
for every $\varphi\in C_c^\infty(A_{2})$, where the matrices $G_k$ and Jacobians $J_k$ introduced by the diffeomorphisms $\Phi_k$ satisfy $G_k\to I$ and $J_k\to1$ uniformly. Moreover, by \eqref{eq:h} and the smooth convergence of the free boundaries, $h_k\circ\Phi_k\to-\partial_{\nu\nu}U=H$ uniformly on $\partial\Omega\cap A_{\frac12,2}$. Passing to the limit gives
\begin{equation*}
\int_{\Omega}\nabla w\cdot\nabla\varphi\,dx
-
\int_{\partial\Omega}Hw\varphi\,d\mathcal H^{d-1}
=0.
\end{equation*}
Thus, $w$ solves \eqref{eq:lineareq}. By its $1$-homogeneity, $w$ extends from $\{U>0\}\cap A_2$ to a $1$-homogeneous solution of \eqref{eq:lineareq} in $\{U>0\}\setminus\{0\}$.

Since $U$ is integrable through rotations and $w$ is $1$-homogeneous, there exists a skew-symmetric matrix $M_0$ such that $w(x)=\nabla U(x)\cdot M_0[x]$ in $\{U>0\}$. On the other hand, dividing \eqref{eq ocond} by $d_k$ and using the strong $L^2$ convergence together with \eqref{eq nonacum}, we obtain
\begin{equation*}
\int_{A_2}
w\,\nabla U\cdot M[x]\,dx=0
\end{equation*}
for every skew-symmetric matrix $M$. Taking $M=M_0$ gives $\int_{A_2}w^2\,dx=0$, contradicting $\|w\|_{L^2(\{U>0\}\cap A_2)}=1$. This proves \eqref{eq enhancedquadestimate}, and therefore also \eqref{eq quadestimate}.
\end{proof}

A simple corollary is that integrable one-homogeneous solutions are, in a sense, isolated in the space of one-homogeneous solutions. We imagine this is well known but we were unable to locate a reference in the setting of the one-phase problem.

\begin{corollary}\label{c:integrability}
Let $U$ be a regular cone integrable through rotations. There exists
$\delta_1>0$, depending on $U$ and on the constants in Definition
\ref{d:weaksol}, such that if $V$ is a one-homogeneous non-degenerate
weak solution in $\R^d$ and
\[
\|U-V\|_{L^2(A_8)}\leq\delta_1,
\]
then $V(x)=U(Rx)$ for some $R\in O(d)$.
\end{corollary}
\begin{proof}
Taking $\delta_1$ sufficiently small, we may apply Lemma \ref{lemma cone approx} to $V$ at scale $1$. Since $V$ is $1$-homogeneous, $r\mapsto W(V,r)$ is constant, and hence
\begin{equation*}
\E_2(1)^2
\leq
C_1\bigl(W(V,4)-W(V,1/4)\bigr)=0.
\end{equation*}
Thus, $V=U(R\,\cdot)$ in $A_2$ for some $R\in O(d)$. Since both functions are $1$-homogeneous, the equality holds globally.
\end{proof}

We now prove the symmetric growth-decay theorem.

\begin{proof}[Proof of Theorem \ref{thm 3annulil}]
We divide the proof into steps.

\medskip

\noindent{\it Step 1:} We show that there exists $\delta_2>0$ such that, if \eqref{eq initsmall} and \eqref{eq monosmallness} hold, then
\begin{equation}\label{eq A8everyscale}
    \inf_{R\in O(d)}\|u_r-U(R\,\cdot)\|_{L^2(A_8)}^2\leq\delta_1^2
\end{equation}
for every $r\in[8r_1,r_2/8]$. Additionally, at every such scale we may choose a minimizing rotation $R_r$ for which the enhanced estimate from Remark \ref{rmk:tildeU} satisfies
\begin{equation}\label{eq everyscale}
    \widetilde{\E}_2(r)^2
    \leq C_1\bigl(W(u,4r)-W(u,r/4)\bigr).
\end{equation}

We prove \eqref{eq A8everyscale} for $r\geq1$; the case $r\leq1$ is analogous. Since \eqref{eq initsmall} holds at $r=1$, after decreasing $\delta_2$ the estimate holds at the initial scale. Set
\begin{equation*}
    r_M:=\sup\bigl\{s\in[1,r_2/8]:
    \text{\eqref{eq A8everyscale} holds for every $r\in[1,s]$}\bigr\}.
\end{equation*}
Assume that $r_M<r_2/8$. By continuity, \eqref{eq A8everyscale} holds at $r_M$. Let $Q_M$ attain the infimum in \eqref{eq A8everyscale} at $r_M$. We apply Lemma \ref{lemma cone approx} and Remark \ref{rmk:tildeU} to $u_{r_M}(Q_M^{-1}\,\cdot)$. Returning to the original coordinates, we obtain a minimizing rotation $R_M$ for $\E_2(r_M)$ such that
\begin{equation*}
\E_2(r_M)^2\leq\widetilde{\E}_2(r_M)^2
\leq C_1\bigl(W(u,4r_M)-W(u,r_M/4)\bigr).
\end{equation*}
Applying \eqref{e:MNpoinc} to $u_{r_M}-U(R_M\,\cdot)$, with $A_8$ on the left hand side and $A_2$ in the last term, we obtain
\begin{align}\label{eq 8ring}
    \|u_{r_M}-U(R_M\,\cdot)\|_{L^2(A_8)}^2
    &\leq C\bigl(W(u,8r_M)-W(u,r_M/8)\bigr)+C\E_2(r_M)^2 \notag\\
    &\leq C(1+C_1)\bigl(W(u,r_2)-W(u,r_1)\bigr)
    \leq C(1+C_1)\delta_2^2.
\end{align}
Choosing $\delta_2$ so that the last quantity is strictly smaller than $\delta_1^2$, continuity shows that \eqref{eq A8everyscale} holds on an interval extending beyond $r_M$, a contradiction. The analogous argument toward smaller scales proves \eqref{eq A8everyscale} on $[8r_1,r_2/8]$. Once this is known, \eqref{eq everyscale} follows at every such scale from the enhanced conclusion of Lemma \ref{lemma cone approx}.

\medskip

\noindent{\it Step 2:} We prove \eqref{eq 3annuli}.

Set $F(r):=W(u,r)-W(U,1)$. We first claim that
\begin{equation}\label{eq refineq}
    |F(q)|\leq C_0I(q)
\end{equation}
for $q\in[8r_1,r_2/8]$. Let $R_q$ be the minimizing rotation chosen in Step 1 and set $v_q(x):=u_q(R_q^{-1}x)$. The Poincar\'e inequality \eqref{e:MNpoinc} and \eqref{eq everyscale} give
\begin{equation*}
    \|v_q-\tilde U\|_{L^2(\Omega_{\mu_q}\cap A_4)}^2
    \leq C\bigl(W(u,4q)-W(u,q/4)\bigr).
\end{equation*}
Applying the argument of Lemma \ref{lemmaL2Linfty}, with $A_3$ in the conclusion and $A_4$ on the right hand side, we obtain the required $C^2$ bound on $A_3$. Lemma \ref{l:c2oneta} then gives the required graph control on $A_2$. We can therefore apply Lemma \ref{l:simontrick} to $v_q$ and $U$ to obtain
\begin{align*}
    |F(q)|
    &=|W(v_q,1)-W(U,1)|\\
    &\leq C\|v_q-\tilde U\|_{L^2(\Omega_{\mu_q}\cap A_3)}^2
    \leq C\bigl(W(u,4q)-W(u,q/4)\bigr)\leq C_0I(q),
\end{align*}
which proves \eqref{eq refineq}. Consequently, if $8r_1<r<s<r_2/8$, then
\begin{equation*}
    W(u,s)-W(u,r)=F(s)-F(r)
    \leq |F(s)|+|F(r)|\leq C_0\bigl(I(s)+I(r)\bigr).
\end{equation*}

Let $\vartheta:=C_0/(1+C_0)\in(0,1)$. We claim that
\begin{equation}\label{eq alg}
    F(8q)_-\leq\vartheta F(q/8)_-,
    \qquad
    F(q/8)_+\leq\vartheta F(8q)_+,
\end{equation}
for $q\in[8r_1,r_2/8]$, where $F_+:=\max\{F,0\}$ and $F_-:=\max\{-F,0\}$. Indeed, if $F(8q)<0$, then the monotonicity of $F$ gives $F(q)_-\geq F(8q)_-$, and \eqref{eq refineq} yields
\begin{equation*}
    F(8q)_-\leq F(q)_-
    \leq C_0I(q) = C_0\bigl(F(q/8)_--F(8q)_-\bigr).
\end{equation*}
This proves the first inequality in \eqref{eq alg}; if $F(8q)\geq0$, that inequality is immediate. Similarly, if $F(q/8)>0$, then
\begin{equation*}
    F(q/8)_+\leq F(q)_+
    \leq C_0I(q) = C_0\bigl(F(8q)_+-F(q/8)_+\bigr),
\end{equation*}
which proves the second inequality; the remaining case is again immediate.

Moreover, monotonicity and \eqref{eq refineq} at the endpoint scales give
\begin{equation*}
    F(r_1)_-\leq |F(8r_1)|+F(8r_1)-F(r_1)\leq CI(8r_1).
\end{equation*}
and, analogously, $F(r_2)_+\leq CI(r_2/8)$. Choose
\begin{equation*}
    \alpha:=\min\left\{\frac12,-\frac{\log\vartheta}{2\log64}\right\}.
\end{equation*}
Iterating \eqref{eq alg} over powers of $64$ and using monotonicity between consecutive powers, we deduce, for $\rho\in[r_1,r_2]$,
\begin{equation*}
    F(\rho)_-\leq C\left(\frac{r_1}{\rho}\right)^{2\alpha}I(8r_1),
    \qquad
    F(\rho)_+\leq C\left(\frac{\rho}{r_2}\right)^{2\alpha}I(r_2/8).
\end{equation*}
Since $I(r)=F(8r)-F(r/8)\leq F(8r)_++F(r/8)_-$, it follows that
\begin{equation}\label{eq Iproof}
    I(r)\leq
    C\left(\frac{r_1}{r}\right)^{2\alpha}I(8r_1)
    +C\left(\frac{r}{r_2}\right)^{2\alpha}I(r_2/8),
\end{equation}
which is \eqref{eq 3annuli}.
\end{proof}

As a consequence, we obtain the proof of the main theorem of this section.

\begin{proof}[Proof of the blow-up statement in Theorem \ref{thm AArestated}]
Set $\epsilon:=\|u-U\|_{L^2(A_8)}$. By Lemmas \ref{lemma: extension}, \ref{lemmaL2Linfty}, \ref{l:c2oneta}, and \ref{l:simontrick}, after decreasing $\delta_0$ if necessary, we have\footnote{
An attentive reader might notice that this would require the bound $\|u-\tilde{U}\|_{L^2(A_7\cap \{u > 0\})} \leq C\|u-U\|_{L^2(A_8)}$ (as opposed to the qualitative bound given by Lemma \ref{lemma: extension}). We give this technical argument in the proof of Lemma \ref{lm:FBreg} (with the roles of $u$ and $U$ harmlessly reversed), but omit it here so as not to obscure the main ideas.}
\begin{equation}\label{eq initialWeisscontrol}
|W(u,1)-W(U,1)|\leq C\epsilon^2.
\end{equation}
We choose $\delta_0$ sufficiently small throughout the proof so that Theorem \ref{thm 3annulil} applies below.

By \eqref{eq blowupdens} and the monotonicity of the Weiss energy,
\begin{equation*}
0\leq W(u,1)-W(u,0)
\leq W(u,1)-W(U,1)
\leq C\epsilon^2.
\end{equation*}
Set $v:=u_{1/8}$. Applying \eqref{e:MNpoinc} to $v-U$, with $A_8$ on the left hand side and $A_{1,2}$ in the last term, and using the $1$-homogeneity of $U$, we obtain
\begin{align*}
\|v-U\|_{L^2(A_8)}^2
&\leq
C\bigl(W(v,8)-W(v,1/8)\bigr)
+C\|v-U\|_{L^2(A_{1,2})}^2\\
&\leq
C\bigl(W(u,1)-W(u,1/64)\bigr)+C\epsilon^2
\leq C\epsilon^2.
\end{align*}
Notice that $v$ is defined in $B_{64}$. Therefore, for every $a\in(0,1/8]$, we may apply Theorem \ref{thm 3annulil} to $v$ with $r_1=a$ and $r_2=8$. If $I_v$ denotes the monotonicity drop associated with $v$, we obtain
\begin{equation*}
I_v(q)\leq
C\left(\frac{a}{q}\right)^{2\alpha}I_v(8a)
+
C\left(\frac{q}{8}\right)^{2\alpha}I_v(1)
\end{equation*}
for $q\in[8a,1]$. Since $I_v(8a)\leq W(v,8)-W(v,0)\leq C\epsilon^2$, we may let $a\to0^+$. Recalling that $I_v(q)=I(q/8)$, we conclude that
\begin{equation}\label{eq Iblowup}
I(r)\leq Cr^{2\alpha}I(1/8)
\leq Cr^{2\alpha}\epsilon^2
\end{equation}
for every $0<r\leq1/8$.

The Weiss monotonicity formula \eqref{eq weissder}, the Minkowski inequality, and the Cauchy--Schwarz inequality give, for $0<s<t$,
\begin{equation}\label{eq rings2}
\|u_t-u_s\|_{L^2(\partial B_1)}
\leq
C\bigl(W(u,t)-W(u,s)\bigr)^{1/2}
\left(\log\frac ts\right)^{1/2}.
\end{equation}
Fix $0<t\leq1$ and set $a_i:=t/64^i$. Applying \eqref{eq rings2} on each interval $[a_{i+1},a_i]$ and using \eqref{eq Iblowup}, we obtain
\begin{equation*}
\|u_{a_i}-u_{a_{i+1}}\|_{L^2(\partial B_1)}
\leq
CI(a_i/8)^{1/2}
\leq
C\epsilon t^\alpha64^{-i\alpha}.
\end{equation*}
Summing this geometric series shows that $u_t$ converges in $L^2(\partial B_1)$ as $t\to0^+$ to some function $V$ and that
\begin{equation}\label{eq boundaryblowup}
\|u_t-V\|_{L^2(\partial B_1)}
\leq C\epsilon t^\alpha
\end{equation}
for $0<t\leq1$.

The scale propagation established in Step 1 of the proof of Theorem \ref{thm 3annulil}, together with Lemma \ref{lemma cone approx}, gives
\begin{equation*}
\E_2(r)^2
\leq
C\bigl(W(u,4r)-W(u,r/4)\bigr)
\longrightarrow0
\end{equation*}
as $r\to0^+$. Standard compactness for non-degenerate weak solutions therefore shows that $V=U(R\,\cdot)$ on $\partial B_1$ for some $R\in O(d)$.

For $0<t\leq1/2$, polar coordinates, the homogeneity of $U$, and \eqref{eq boundaryblowup} give
\begin{align*}
    \|u_t-U(R\,\cdot)\|_{L^2(A_2)}^2
    &=
    \int_{1/2}^{2}\rho^{d+1}
    \|u_{\rho t}-U(R\,\cdot)\|_{L^2(\partial B_1)}^2\,d\rho\\
    &\leq C\epsilon^2t^{2\alpha}.
\end{align*}
To cover the remaining range $1/2\leq t<1$, notice that
\[
\|u_t-U\|_{L^2(A_2)}
\leq C\epsilon,
\]
because $A_{t/2,2t}\subset A_8$ and $t$ is bounded away from zero. Taking $t=1/2$ in the previous estimate and using the triangle inequality gives
$\|U-U(R\,\cdot)\|_{L^2(A_2)}\leq C\epsilon$. Hence, for $1/2\leq t<1$,
\[
\|u_t-U(R\,\cdot)\|_{L^2(A_2)}
\leq C\epsilon
\leq Ct^\alpha\epsilon.
\]
This proves \eqref{eq:rate0sec4}.
\end{proof}

We can now complete the proof of Theorem \ref{thm AA}.

\begin{proof}[Proof of Theorem \ref{thm AA}]
The result is a direct application of Theorem \ref{thm AArestated}. In this case, necessarily $U(R \, \cdot) =U$ since we know that $u_r$ converges to $U$ at least through one subsequence.
\end{proof}

\section{Classification of global solutions asymptotic to strongly integrable cones: Proof of Theorem \ref{thm main}}\label{s:smooth}

Combining the general theory of Section \ref{sec: conv} with the rate of convergence given by Theorem \ref{thm AA}, we have the following smooth parameterization result, which will help with the proof of Theorem \ref{thm main}.

\begin{lemma}\label{lm:FBreg}
Let $u$ be a non-degenerate weak solution in the sense of Definition \ref{d:weaksol} and assume that there exist $R_j\uparrow \infty$ such that $u(R_jx)/R_j \rightarrow U(x)$ uniformly on compacta. Further assume that $U$ is a regular integrable cone (see Definition \ref{def: intcone}). 

Then there exist $R_1 > 0$ and an open set $\mathcal O \supset \{U > 0\}\setminus B_{R_1}$ such that $u$ can be extended analytically to a function $\tilde{u}$ in $\mathcal O$. 

Furthermore, there exists an $\eta \in C^\omega(\partial \{U > 0\}\backslash B_{R_1/2}(0))$ such that \begin{equation}\label{e:analyticparam} \partial \{u > 0\}\backslash B_{R_1}(0) = \{x + \nu(x)\eta(x)\mid x\in \partial \{U > 0\} \backslash B_{R_1}(0)\},\end{equation} where $\nu$ is the inward pointing unit normal to $\partial \{U > 0\}$. 

Finally there exist $C > 0$ and $\alpha_0 \in (0,1)$ (depending on $U$ but not $u$) such that both $\eta$ and $v:= \tilde{u} -U$ satisfy the following estimates:

%Let $u,b$ be non-degenerate weak solutions in the sense of Definition \ref{d:weaksol}. Further assume that $b$ is one-homogeneous and $\partial \{b > 0\}$ is smooth away from the origin.  Finally, assume there exists a $\alpha_0 \in (0,1), r_0 > 0, C > 0$ such that $$\|u - b\|_{L^\infty(B_R^c)} < CR^{1-\alpha_0}, \forall R > r_0.$$

  % Then there exists an $r_1 > 0$ such that
  %$$\partial \{u > 0\}\backslash B_{r_1}(0) = \{x + \nu(x)\eta(x)\mid x\in \partial \{b > 0\}\},$$
 %  where $\nu$ is the inward pointing unit normal to $\partial \{b > 0\}$ and $\eta \in C^{\omega}(\partial \{b > 0\}\backslash B_{r_1/2}(0))$.\footnote{We use $C^\omega$ here to denote the class of analytic functions.}

    %Furthermore there exists a constant $C > 0$ (depending on the constants above and the dimension) such that for any $R > r_1$ we have the estimates that 
    \begin{equation}\label{e:etaandvdecay}\begin{aligned}\left|\frac{v(x)}{|x|}\right| + &|\nabla v(x)|+ ||x|\nabla^2 v(x)| \leq C|x|^{-\alpha_0}, \quad \forall x\in \overline{\{U > 0\}}\setminus B_{2R_1}(0), \\
    \left|\frac{\eta(x)}{|x|}\right| + &|\nabla \eta(x)| + ||x|\nabla^2 \eta(x)| \leq C|x|^{-\alpha_0}, \quad \forall x\in \partial \{U  > 0\}\setminus B_{2R_1}(0).
    \end{aligned}
    \end{equation}
\end{lemma}

Note that we are considering $\eta$ to be a function of $x\in \partial \{U > 0\}$ so the gradient and Hessian are with respect to the tangential directions.

\begin{proof}
By Theorem \ref{thm AA}, we have that there exists an $R_1 > 0$ (depending on $u$) such that if $r > R_1$ then 
$$\|u(rx)/r - U\|_{L^2(B_8\backslash B_{1/8})} \leq Cr^{-\alpha_0}.$$

Perhaps increasing $R_1$, we can guarantee that $r^{-\alpha_0} \leq R_1^{-\alpha_0} \leq \delta_0$ (the parameter from Lemma \ref{lemma: extension}). Invoking Lemma \ref{lemma: extension}, we have the existence of $\tilde{u}$ and $\eta$. 

Let us assume for a second that we can establish the $L^2$-bound \begin{equation}\label{e:tildedecay}\|\tilde{u}(rx)/r - U\|_{L^2(\{U > 0\}\cap (B_8\backslash B_{1/8}))} \leq Cr^{-\alpha_0}.\end{equation} Then invoking Lemmas \ref{l:linfinitytoc2} and \ref{lemmaL2Linfty} (but reversing the roles of $u, U$ which causes no problems) we will have the desired bounds for $v$ in \eqref{e:etaandvdecay}. The desired bounds for $\eta$ come from Lemma \ref{l:c2oneta}.

Thus we are left with proving the enhanced decay on the $L^2$ norm of $\tilde{u} - U$ in $\{U > 0\}$, i.e. \eqref{e:tildedecay}, which boils down to proving the  $L^2$ decay of $\tilde{u} - U$ in $\{U > 0\}\setminus \{u>0\}$.  Note that $\tilde{u}$ in $\mathcal O\cap \{u =0\}$ satisfies the same Lipschitz condition as $u$ (this follows because $u$ has a sign in this regime, perhaps shrinking $\mathcal O$ slightly; see step 4 in the proof of Lemma \ref{lemma: extension}). Using the qualitative smooth convergence of $\partial \{u > 0\}$ to $\partial \{U > 0\}$ we know then that if $y\in \{U > 0\}\backslash \{u > 0\}$ we can write $y= x + t\eta(x)\nu(x)$ for some $x\in \partial \{U > 0\}$ and $t\in (0,1)$, so that $U(y) \simeq t\eta(x)$(we use the notation $f(x)\simeq g(x)$ to mean that there is a constant $C$, independent of $x$, such that $Cf(x) \geq g(x) \geq C^{-1}f(x)$); whereas since $\tilde{u}(x+\eta(x)\nu(x))=0$, we have that $-\tilde{u}(y) \simeq (1-t)\eta(x)$. In particular, we have that $U(x+t\eta(x)\nu(x)) \simeq -\tilde{u}(x+(1-t)\eta(x)\nu(x))$ for $t\in (0,1)$. This implies
$$\|\tilde{u}\|_{L^2((\{U > 0\}\backslash\{u > 0\})\cap (B_{8r}\backslash B_{r/8}))} \simeq \|U\|_{L^2((\{U > 0\}\backslash\{u > 0\})\cap (B_{8r}\backslash B_{r/8}))},$$
by direct integration on the normal coordinates $(x,t)$. 
Thus $\|\tilde{u} - U\|_{L^2(\{U>0\}\cap B_{8r}\backslash B_{r/8})} \simeq \|u - U\|_{L^2(\{U>0\}\cap B_{8r}\backslash B_{r/8})}$ (by the triangle inequality) and so \eqref{e:tildedecay} follows from the decay on the $L^2$-norm of $u-U$. 

\end{proof}

%From now on, we denote $\Omega := \{U>0\}$, $\Omega' :=\{u>0\}$ and $C=\partial\Omega$. Recall that the domains $\Omega, \Omega'$ are not necessarily nested. Recall we have shown in Lemma \ref{lm:FBreg} that the free boundary $\partial\Omega'$ is graphical over the cone. Namely
%\begin{equation}\label{eq:graph}
%\partial \Omega' \setminus B_{R_0}(0) = \Graph_C(\eta) \setminus B_{R_0}(0), 
%\end{equation}
%where $\Graph_C(\eta) = \{x + \eta(x) \nu(x) : x\in C \}$, $\nu(x)$ denotes the unit normal vector pointing inside $\Omega$, and the map $\eta: C \to \mathbb{R}$ is analytic and satisfies
%\begin{equation}\label{eq:decayFB}
 %   R^{-1} \cdot \|\eta\|_{L^\infty(\{R<|X|<2R\} )} + \|\nabla \eta\|_{L^\infty(\{R<|X|<2R\} )} + R \cdot \|\nabla^2 \eta\|_{L^\infty(\{R<|X|<2R\} )} = O(R^{-\alpha_0} ) 
%\end{equation} 
%for some $\alpha_0 \in (0,1)$, for any $R$ sufficiently large.
%In particular, as long as $R_0$ is sufficiently large, we may extend $u$ across the boundary $\partial\Omega'$ using Cauchy–Kovalevskaya theorem so it remains a harmonic function in (an open set containing) $\overline{\Omega} \setminus B_{R_0}(0)$. Consider the vertical displacement of $u$ with respect to $b$, namely the function $v:= u-b$ on $\Omega \setminus B_{R_0}(0)$. Next, we show that it satisfies the linearized equation \eqref{eq:lineareq} with a quadratic perturbation.

Combining the above Lemma \ref{lm:FBreg} with Lemma \ref{lemma: extension} Part (6) immediately yields the following:

\begin{lemma}\label{lm:quadraticp}
    Let $u, U$ be as in Lemma \ref{lm:FBreg} and define as before $v:= \tilde{u} - U$. Then there exists an $R_1$ (as in Lemma \ref{lm:FBreg}) such that whenever $R > R_1$, $v$ satisfies the differential equation
    
    %Let $A_R = B_{2R}\setminus B_R$. The following is true as long as $R$ is sufficiently large. For any $\alpha>0$, assume that $v$ and $\eta$ satisfies 
    %\begin{equation}\label{as:decayv}
    %    R^{-1} \cdot \|v\|_{L^\infty(\Omega \cap A_R )} + \|\nabla v\|_{L^\infty(\Omega \cap A_R )} = O(R^{-\alpha}),
   % \end{equation}
    %\begin{equation}\label{as:decayeta}
    %    R^{-1} \cdot \|\eta\|%_{L^\infty(C \cap A_R )} + \|\nabla \eta\|_{L^\infty(C \cap A_R )} = O(R^{-\alpha}).
   % \end{equation}
    
    \begin{equation}\label{eq:PL}
        \left\{\begin{array}{ll}
            \Delta v = 0, & \text{ in }\{U > 0\} \cap B_{2R}\backslash B_R \\
            \partial_{\nu} v + Hv = g_v, %O\left( \left(\frac{\|v\|_{L^\infty}}{R} \right)^2 \right), 
            & \text{ on } \partial \{U > 0\} \cap B_{2R}\backslash B_R,
        \end{array} \right.\tag{PL}
    \end{equation}
    where $g_v$ is a $C^2$ function with decay $O(R^{-2\alpha_0})$.
    % which is a perturbation of the linearized equation \eqref{eq:lineareq}. Here the $L^\infty$ norm are restricted to the region $\Omega \cap B_{2R} \setminus B_R$.
\end{lemma}

Before starting the proof of Theorem \ref{thm main}, we need one more technical result on the connectedness of the positivity set $\{u > 0\}$. 

\begin{lemma}\label{lemma connectedness}
    Let $u$ be a non-degenerate weak solution in the sense of Definition \ref{d:weaksol}, defined on all of $\mathbb R^d$. Assume that $u(R_jx)/R_j \rightarrow U$ locally uniformly on compacta, where $U$ is a regular cone in the sense of Definition \ref{def regular cone}. Then $\{u > 0\}$ is connected. 
\end{lemma}
\begin{proof}
We start pointing out that $\{U > 0\}$ is connected, which by the homogeneity of $U$ is equivalent to the connectedness of $\{U > 0\} \cap \mathbb S^{d-1}$. This is well known but we sketch a proof: if this is not the case, consider two connected components $\Omega_1$ and $\Omega_2$, which are also cones, such that $\Omega_1 \cap \Omega_2 =\emptyset$ and let $U_i = U\chi_{\Omega_i}$. Notice that $U_1$ and $U_2$ are non-degenerate solutions to the one-phase problem. Since $\Omega_1$ and $\Omega_2$ are disjoint open cones, and $U_i$ is harmonic 1-homogeneous and positive in $\Omega_i$, the links $\Sigma_i = \mathbb{S}^{d-1}\cap \Omega_i$ must satisfy $\lambda_1(\Sigma_i) = d-1$ for $i=1, 2$ (where $\lambda_1$ is the first Dirichlet eigenvalue). However, the spherical Faber-Krahn inequality implies that $|\Sigma_i| \geq \frac{1}{2}|\mathbb S^{d-1}|$ with equality if and only if each $\Sigma_i$ is a half-sphere. If both $\Sigma_i$ are half-spheres then $U = |x\cdot e|$ for some $e\in \mathbb S^{d-1}$, which is not a regular cone. This yields the contradiction. Since $U$ is a regular cone, this argument also shows that $\{U=0\}  = \bigcup_{i=1}^k D_i$ where $D_i$ are cones with smooth, connected, and mutually disjoint links $C_i$ on $\mathbb{S}^{d-1}$.\\

On the other hand, thanks to Lemma \ref{lm:FBreg} there exists $R_1>0$ such that $\partial \{u > 0\}\backslash B_{R_1}(0)$ can be written as a smooth graph over $\partial \{U > 0\} \backslash B_{R_1}(0)$ as in \eqref{e:analyticparam}, with decay near $\infty$ as in \eqref{e:etaandvdecay}. Together, the structure of the zero set of $U$ and the previous observation imply that $\partial \{u > 0\}\backslash B_{R_1}(0) = \bigcup_{i=1}^k F_i$ where the $F_i$ are smooth, connected, and mutually disjoint hypersurfaces that can be seen individually as graphs over the corresponding $D_i \setminus B_{R_1}(0)$. Additionally, since $\chi_{\{u(R\cdot)>0\}} \to \chi_{\{U>0\}}$ on compacta as $R\to \infty$, we must have that ${\{u = 0\}}\backslash B_{R_1}(0)  = \bigcup_{i=1}^k E_i$ where $E_i$'s are closed connected sets lying at positive distance $\delta_0$ (we may need to take $R_1$ larger if necessary) and satisfying $\partial E_i \setminus B_{R_1}(0) =F_i$. Moreover, for any $R>R_1$ we have that $\frac{1}{R} E_i \cap \mathbb{S}^{d-1}$ are disjoint connected closed sets each with connected smooth boundary on $\mathbb{S}^{d-1}$ and, thus, by virtue of  Jordan's separation theorem  $\mathbb{S}^{d-1} \setminus \bigcup_{i=1}^k \frac{1}{R} E_i $ is connected in $\mathbb{S}^{d-1}$. The same argument allows us to conclude that the complement of tubular neighborhoods, i.e., $\mathbb{S}^{d-1} \setminus \bigcup_{i=1}^k \frac{1}{R} N_\delta(E_i) $ is also connected in $\mathbb{S}^{d-1}$ provided that $\delta$ is small enough.\\

By piecing together the previous considerations, we show now that $\{u > 0\}\backslash B_{R_1}(0)$ is connected. Indeed, let $x_1, x_2 \in \{u > 0\}\backslash B_{R_1}(0)$ with $|x_1| = r_1$ and $|x_2| = r_2$. If $r_1 = r_2$, the result follows by the connectedness of $\mathbb{S}^{d-1} \setminus \bigcup_{i=1}^k \frac{1}{r_1} E_i$. Otherwise, take $\delta \in (0, \delta_0)$ small enough such that the boundaries of the tubular neighborhoods $\partial N_\delta( E_i)$ are graphs over the corresponding $D_i$ and such that $x_1, x_2 \notin \bigcup_{i=1}^k N_\delta( E_i) \setminus B_{R_1}(0)$. By connectedness of $\mathbb{S}^{d-1} \setminus \bigcup_{i=1}^k {\frac{1}{r_1}} N_\delta(E_i)$ we can take a path in this set connecting $x_1/r_1$ to a point $y$ on the boundary of $N_\delta( E_i)$ for some $i$. In particular, since $D_i$ is a connected cone in $\mathbb{R}^d$, we can find a path along $\partial N_\delta( E_i)$ joining $y$ to a point $z \in \{u > 0\}\cap \partial B_{r_2}$ which can be joined in turn to $x_2$ by connectedness of $\mathbb{S}^{d-1} \setminus \bigcup_{i=1}^k \frac{1}{r_2} E_i$.\\

Finally, we notice that since $u$ is harmonic in its positivity set, the maximum principle forces each connected component of $\{u>0\}$ to be unbounded, but since $\{u>0\}\setminus B_{R_1}(0)$ is connected it follows that $\{u>0\}$ must be connected.

\end{proof}

As a last preparatory step, we require the following observation about the asymptotic behavior of the foliating solutions from \cite{DSJS}.

\begin{remark}\label{remark asympfoliation}
    As shown in Lemma \ref{lemma: extension}, in particular in equation \eqref{eq expansion eta}, the normal distance between the free boundary of a regular cone and the free boundary of a nearby solution is comparable to their vertical displacement. We can exploit this fact to obtain further information about the behavior of the foliating solutions from \cite{DSJS} --recalled in Theorem \ref{thm foliation}. More precisely, if $U$ is an entire minimizing cone with associated foliating solutions $\overline{U}$ and $\underline{U}$, the asymptotic expansions \eqref{eq:upperfol} and \eqref{eq:lowerfol} allow us to write, for some $r_1>0$,
     \begin{equation*}
          \partial \{\overline{U}_t > 0\}\backslash B_{r_1} = \{x - \nu(x)\overline{\eta}_t(x)\mid x\in \partial \{U > 0\}\},
     \end{equation*}
and
     \begin{equation*}
          \partial \{\underline{U}_t > 0\}\backslash B_{r_1} = \{x + \nu(x)\underline{\eta}_t(x)\mid x\in \partial \{U > 0\}\},
     \end{equation*}
 where $\overline{\eta}_t = \frac{\overline{\eta}(t\cdot)}{t}$, $\underline{\eta}_t =\frac{\underline{\eta}(t\cdot)}{t}$,
\refstepcounter{equation}\label{eq pm uperfoliation-num}
\begin{align}
    \overline{\eta}(r\omega) 
    &\simeq \overline{U}(r\omega) - U(r\omega) \notag \\
    &= \bar{a} r^{-\frac{(d-2)}{2}\pm\delta_1} \phi_1(\omega)
       + o\bigl(r^{-\frac{(d-2)}{2}\pm\delta_1}\bigr),
    \tag{\theequation~$\pm$}\label{eq pm uperfoliation}
\end{align}
and

\refstepcounter{equation}\label{eq pm lowerfoliation-num}
\begin{align}
    \underline{\eta}(r\omega)
    &\simeq  U(r\omega)-\underline{U}(r\omega) \notag \\
    &= \bar{b} r^{-\frac{(d-2)}{2}\pm\delta_1} \phi_1(\omega)
       + o\bigl(r^{-\frac{(d-2)}{2}\pm\delta_1}\bigr).
    \tag{\theequation~$\pm$}\label{eq pm lowerfoliation}
\end{align}
\end{remark}

\begin{proof}[Proof of Theorem \ref{thm main}]
We start by noting that thanks to Lemma \ref{lm:FBreg} $v = \tilde u-U$ satisfies $v=O(R^{1-\alpha_0})$ for some $\alpha_0\in (0,1)$ along with appropriate decay estimates for its first and second-order derivatives. 

    \medskip

\noindent {\it Step 1:} We show that $v$ satisfies one of the following conditions

    \begin{equation}\label{eq:vtrsl}
        v=\phi_2(\omega) + o(1),\tag{a}
    \end{equation}
    \begin{equation}\label{eq:vfl}
        v=\phi_1(\omega) r^{-\frac{(d-2)}{2}+\delta_1}+o(r^{-\frac{(d-2)}{2}+\delta_1}),\tag{b}
    \end{equation}
    \begin{equation}\label{eq:vfl-}
        v=\phi_1(\omega) r^{-\frac{(d-2)}{2}-\delta_1}+o(r^{-\frac{(d-2)}{2}-\delta_1}),\tag{c}
    \end{equation}
    \begin{equation}\label{eq:vtrsl-}
        v=O(r^{-(d-2)}),\tag{d}
    \end{equation}
    where $\phi_1$ denotes the first eigenfunction of \eqref{eq:evonlink}, associated with $\lambda_1$, $\delta_1 = \sqrt{\left(\frac{d-2}{2} \right)^2 + \lambda_1 }$, and $\phi_2$ denotes any of the second eigenfunctions of \eqref{eq:evonlink} with eigenvalue $0$.\\

  Exploiting the decay properties of $v$, by Lemma \ref{lm:quadraticp} (with $\alpha=\alpha_0$), $v$ satisfies the equation \eqref{eq:PL} with the right hand side of the boundary condition $g_v =O(R^{-2\alpha_0})$, along with appropriate decays for its derivatives. By Lemma \ref{lemma part sol} (if necessary we may need to take the exponent $\beta$ to be slightly smaller than $2\alpha_0$ to avoid a countable discrete set and satisfy the assumption therein), we may construct a particular solution to
    \[ \left\{\begin{array}{ll}
       -\Delta v_p = 0  & \text{ in } \Omega \setminus B_{R_0} \\
       \partial_{\nu}v_p + Hv_p = g_v  & \text{ on } \partial\Omega \setminus B_{R_0},
    \end{array} \right. \]
    such that $v_p$ (and its derivatives) decays like $O(R^{1-2\alpha_0})$. It follows that $v_h := v-v_p$ satisfies
    \[ \left\{\begin{array}{ll}
       -\Delta v_h = 0  & \Omega \setminus B_{R_0} \\
       \partial_{\nu}v_h + Hv_h = 0  & \partial\Omega \setminus B_{R_0}.
    \end{array} \right. \]
    Since $v_h=v-v_p$ at least has the decay of $O(R^{1-\alpha_0})$, by Lemma \ref{l:fourierimprovement}, $v_h$ satisfies one of the above four conditions. Recall that $\Omega = \{U > 0\}$ and that $\Omega' = \{u > 0\}$.

    Now suppose $k\in \mathbb{N}$ is the largest integer so that $v=O(R^{1-k\alpha_0})$. If $1-k\alpha_0 \leq -(d-2)$, then $v$ satisfies \eqref{eq:vtrsl-} and we are done. Assume $1-k\alpha_0 > -(d-2)$. By the argument in the previous paragraph, we can construct a particular solution $v_p$ with decay $O(R^{1-2k\alpha_0})$, and $v_h = v-v_p$ satisfies one of the above four conditions. Either $v_h$ dominates $v_p$, in which case $v$ satisfies one of the above four conditions as well; or else $v_p$ dominates, or $v_p$ has the same rate of decay as $v_h$. In the latter case, we get that $v$ decays like $O(R^{1-2k\alpha_0})$. This is a contradiction since $k$ is assumed to be the largest integer so that $v=O(R^{1-k\alpha_0})$.\\

\medskip

In the remaining steps, our arguments will be based on comparing the global solution $u$ with the ``foliating'' rescalings $\overline{U}_{t} = \frac{\overline{U}(t\cdot)}{t}$ and $\underline{U}_{t} = \frac{\underline{U}(t\cdot)}{t}$ introduced in Theorem \ref{thm foliation}. We will exploit Remark \ref{remark asympfoliation} to see their free boundaries as graphs over $\partial\Omega\setminus B_r$ for some $r>0$ with decay at infinity given by the expansions \eqref{eq pm uperfoliation} and \eqref{eq pm lowerfoliation}.

\medskip
\noindent {\it Step 2:} If \eqref{eq:vtrsl-} holds, then $u=U$.\\

Lemma \ref{lm:FBreg} implies the existence of $r_1>0$ such that  
\begin{equation}\label{eq decayfb}
    \partial \Omega'\backslash B_{r_1} = \{x + \nu(x)\eta(x)\mid x\in \partial \Omega \setminus B_{r_1}\}, \quad \mbox{$|\eta(x)|\leq C|x|^{\gamma}$,}
\end{equation} 
with $\gamma = -(d-2)$. On the other hand, from \eqref{eq quadratic form} we have that $\delta_1<\frac{d-2}{2}$, and thus $-\frac{(d-2)}{2}\pm\delta_1 > -(d-2)$. In particular, we have that $\eta$ decays faster than $\overline{\eta}_t$ and $\underline{\eta}_t$ for all $t>0$ regardless of the signs in \eqref{eq pm uperfoliation} and \eqref{eq pm lowerfoliation}. So, by combining these considerations with Theorem \ref{thm foliation}, we can find $t_0$ such that for $t \leq t_0$, $\{\overline{U}_t>0\} \supset \Omega'\supset \{\underline{U}_t>0\}$ and, furthermore, we can guarantee thanks to this nestedness and the asymptotic behavior of $u$ that for $t_0$ sufficiently small    $\overline{U}_t\geq u \geq \underline{U}_t$ if $t\in (0,t_0)$. Let us consider
\begin{equation}\label{eq uppert}
    t_1= \sup\{t>0: \overline{U}_t\geq u \quad \mbox{in $\R^d$}\}
\end{equation}
and
\begin{equation}\label{eq lowert}
    t_2= \sup\{t>0: u\geq \underline{U}_t \quad \mbox{in $\R^d$}\}.
\end{equation}
We claim that $t_1, t_2= \infty $. To see this, assume by contradiction that $t_1 <\infty$. Since at infinity $\partial \Omega'$ approaches  $\partial \Omega$  faster than $\partial \{ \overline{U}_{t_1}>0\}$ (seen as graphs), we have that $\overline{U}_{t_1}$ must touch $u$ from above either in $\Omega$ or at $\partial \Omega$. Since by construction $\overline{U}_{t_1}\geq u$, we can combine Lemma \ref{lemma connectedness} with Lemma \ref{lemma comparison} to conclude that $u= \overline{U}_{t_1}$, which cannot happen because $u$ decays faster than $\overline{U}_{t_1}$.

Analogous considerations, exploiting the comparison principle from Lemma \ref{lemma comparison}, show that $t_2=\infty$. Since $\lim_{t\to \infty}  \overline{U}_t = \lim_{t\to \infty}  \underline{U}_t =U$, we deduce that $u=U$.
\medskip

\noindent {\it Step 3:} If \eqref{eq:vfl} or \eqref{eq:vfl-} holds, then $u= \overline{U}_t$, $u= \underline{U}_t$ for some $t>0$, respectively.\\

We start by noticing that since $\phi_1$ is an eigenfunction corresponding to the first eigenvalue in \eqref{eq:evonlink}, it does not change sign. Thus, by our growth assumptions, there exists $R>0$ such that either $u(x) \leq U(x)$ or $u(x) \geq U(x)$ in $B_R^c$, with strict inequalities in the intersection of their positivity sets. If the former holds, there exist some $t>0$ sufficiently small such that $\overline{U}_t \geq u$ in $\R^d$ and thus
\begin{equation*}
    t_1=\sup\{t>0: \overline{U}_t\geq u \quad \mbox{in $\R^d$}\} >0
\end{equation*} 
Arguing as in the previous step, we show that $t_1 =\infty$. If instead $u(x) \geq U(x)$ in $B_R^c$, we argue similarly comparing instead with the family $\underline{U}_t$ to show that $u \geq U$ in $\R^d$.

Let us specialize now in the case $u(x)\geq U(x)$; the other case follows by an analogous argument. Recall that $\overline{U}$ exhibits the two possible asymptotic behaviors \eqref{eq pm uperfoliation}. Arguing as in the previous step, consider 
\begin{equation}\label{eq uppert2}
    t_1= \sup\{t>0: \overline{U}_t\geq u \quad \mbox{in $\R^d$}\}
\end{equation}
if \eqplus{eq pm uperfoliation-num} holds, or
\begin{equation}\label{eq lowert2}
    t_1= \inf\{t>0: u\geq \overline{U}_t \quad \mbox{in $\R^d$}\}.
\end{equation}
if \eqminus{eq pm uperfoliation-num} holds instead. In either situation, the growth assumptions on $u$ imply --exactly as before-- that $t_1 \in (0,\infty)$. In consequence, $\overline{U}_{t_1}$ must touch $u$ either from below or from above, which implies $u = \overline{U}_{t_1}$ in virtue of Lemma \ref{lemma comparison}.

\medskip

\noindent {\it Step 4:} We are left to consider if \eqref{eq:vtrsl} holds, i.e. if $v(r\omega) = \phi_2(\omega) + o(1)$. We note that $\phi_2(\omega)$ is of the form $c\partial_{e_i}U$ for some $i= 1,\ldots, d$ and $c\neq 0$. After a harmless rotation let $i = 1$ and we define a family of solutions $u_t(x) := u(x+te_1)$. We now want to show that there is a choice of $t\in \mathbb R$ such that $v_t(x):= u_t(x) - U(x)$ satisfies one of the conditions \eqref{eq:vfl}, \eqref{eq:vfl-}, \eqref{eq:vtrsl-}. 

We first note that for any $t\in \mathbb R$ we have $|u(x+te_1) - u(x)| \leq C|t|$. In particular, for any $t$ we have $u_t$ blows down to $U$ and, perhaps enlarging $R_0$ depending on an upper bound for $|t|$, we have that $v_t = O(R^{1-\alpha_0})$ in $B_{R_0}^c$. Thus arguing as above we can ensure that $v_t$ satisfies one of the four alternatives spelled out in Step 1. 

Now assume that $|x| \equiv R \gg t$ and using the definition of $v$ and letting $\omega_t = \frac{x+te_1}{\|x+te_1\|}$ we have that $$\begin{aligned} v_t(x) =& v(x+te_1) + U(x+te_1) - U(x)\\ =& v(x+te_1) - t\partial_{e_1}U(x+te_1) + \frac{1}{2}t^2\partial_{e_1e_1}U(z)\\
=& \phi_2(\omega) -\frac{t}{c}\phi_2(\omega) + o(1) + O(\frac{t^2}{R}),\end{aligned}$$ where we have used the mean value theorem for $\partial_{e_1}U$ and the expansion for $v$ in \eqref{eq:vtrsl}. So, picking $t= c$ we have that $v_t(rw) = o(1) + O(R^{-1})$ and thus $v_t$ satisfies one of the conditions \eqref{eq:vfl}, \eqref{eq:vfl-}, \eqref{eq:vtrsl-}. 

By the analysis above either $u_t(x) = U(x)$ (see Step 2), or $u_t$ is a leaf of the foliation, i.e. $u_t(x) = \overline{U}_s$ or $u_t = \underline{U}_s(x)$ (see Step 3). This implies of course that $u(x) = U(x-ce_1)$ or $u(x) = \overline{U}_s(x-ce_1)$ or $u(x) = \underline{U}_s(x-ce_1)$. This completes our proof. 
\end{proof}

\phantomsection

\section{Appendix}\label{sec: appendix}

In this appendix, we recall some linear estimates used in the body of our work. We start with the following Poincar\'e-type inequality similar in spirit to \cite[(A.5)]{MaggiNovack}.\\

We recall the notation $A_{s,r} := B_r \setminus \overline{B_{s}}$ for $0<s<r$ and $A_r := A_{1/r, r}$. In addition, we recall that $O$ is an open cone if $O$ is open and $\lambda O =O$ for any $\lambda >0$.
\begin{lemma}
Let $O$ be an open cone and let $0<r_0<r_1<r_2<r_3$. There exists a constant $C>0$, depending only on the dimension and on $r_0,r_1,r_2,r_3$, such that
\begin{equation}\label{e:MNpoinc}
\int_{A_{r_0,r_3}\cap O}f^2\,dx
\leq
C\left(
\int_{A_{r_0,r_3}\cap O}(x\cdot\nabla f-f)^2\,dx
+
\int_{A_{r_1,r_2}\cap O}f^2\,dx
\right)
\end{equation}
for every $f\in W^{1,2}(A_{r_0,r_3}\cap O)$.
\end{lemma}

\begin{proof}
By an approximation argument, we may assume that $f$ is smooth. Set $\Sigma:=O\cap\partial B_1$. For $\theta\in\Sigma$, apply the fundamental theorem of calculus to $f_t(\theta):=\frac{f(t\theta)}{t}$.

If $\rho\in[r_2,r_3]$ and $s\in[r_1,r_2]$, we obtain
\begin{equation*}
\frac{f(\rho\theta)}{\rho}
=
\frac{f(s\theta)}{s}
+
\int_s^\rho
\frac{t\theta\cdot\nabla f(t\theta)-f(t\theta)}{t^2}\,dt.
\end{equation*}
Taking squares, averaging with respect to $s\in[r_1,r_2]$, and then integrating with respect to $\rho\in[r_2,r_3]$ and $\theta\in\Sigma$, we find
\begin{eqnarray*}
\int_{A_{r_2,r_3}\cap O}f^2\,dx
&\leq&
C\int_{A_{r_1,r_3}\cap O}(x\cdot\nabla f-f)^2\,dx
+
C\int_{A_{r_1,r_2}\cap O}f^2\,dx.
\end{eqnarray*}
Here and below, the factors arising from polar coordinates are absorbed into $C$, since all the radii lie between the fixed positive numbers $r_0$ and $r_3$.

Similarly, if $\rho\in[r_0,r_1]$ and $s\in[r_1,r_2]$, then
\begin{equation*}
\frac{f(\rho\theta)}{\rho}
=
\frac{f(s\theta)}{s}
-
\int_\rho^s
\frac{t\theta\cdot\nabla f(t\theta)-f(t\theta)}{t^2}\,dt.
\end{equation*}
Taking squares and integrating as above gives
\begin{eqnarray*}
\int_{A_{r_0,r_1}\cap O}f^2\,dx
&\leq&
C\int_{A_{r_0,r_2}\cap O}(x\cdot\nabla f-f)^2\,dx
+
C\int_{A_{r_1,r_2}\cap O}f^2\,dx.
\end{eqnarray*}
Adding these two estimates and the integral of $f^2$ over $A_{r_1,r_2}\cap O$ proves \eqref{e:MNpoinc}.
\end{proof}

We recall now some linear estimates for boundary value problems with Robin boundary condition.
\begin{lemma}\label{lemma lib}
Let $\Omega$ be an open set with $\partial \Omega\setminus\{0\}$ smooth, and let $v$ be a distributional solution of
    \begin{equation}\label{eq: oblique}
    \left\{\begin{array}{ll}
\Delta v = f & \text{ in } \Omega \cap A_{4,1} \\
\partial_{\nu} v + h v = g & \text { on } \partial\Omega \cap A_{4,1},
\end{array}
\right.
\end{equation} 
where $f$, $h$, and $g$ are continuous bounded functions and $\nu$ is the inner unit normal to $\Omega$. Then,
\begin{equation}\label{eq oblique estimate}
    \Vert v\Vert_{L^\infty(A_{3,2}\cap \Omega)}\leq C( \Vert v\Vert_{L^2(A_{4,1}\cap \Omega)}+\Vert f\Vert_{L^\infty(A_{4,1}\cap \Omega)}+\Vert g\Vert_{L^\infty(A_{4,1}\cap \partial \Omega)})
\end{equation}
where $C$ only depends on $\Vert h\Vert_{L^\infty(\partial \Omega \cap A_{4,1})}$.
\end{lemma}
\begin{proof}
    The proof essentially corresponds to Theorem 5.31 (for subsolutions) and Corollary 5.32 (for supersolutions) in \cite{Lieberman}. In these statements $\Omega$ is assumed to be a bounded domain and the boundary condition is imposed on the entire boundary $\partial \Omega$, and the proof is done through a standard Moser iteration. We just remark that the same iteration yields \eqref{eq oblique estimate} provided that one localizes the argument by adding a (radial) cut-off $\eta^2$ with $\eta\in C_c^\infty(A)$ with $\eta =1$ in $A'$.
\end{proof}

Rescaling, we obtain the following decay estimate for solutions of the inhomogeneous linearized operator \eqref{eq part sol}.

\begin{corollary}\label{cor regdecay}
Let $R_0>0$, $\Omega$ be a cone with $\partial \Omega \setminus \{0\}$ smooth, let $h, G: \partial \Omega\setminus B_{R_0} \to \R$ be continuous functions with $h$ homogeneous of degree $-1$, and let $u$ be a distributional solution to
\begin{equation}\label{eq part solap}
    \begin{cases}
        \Delta v = 0, \quad \Omega \setminus B_{R_0},\\
       \partial_\nu v+ hv = G, \quad \partial \Omega\setminus B_{R_0}.
    \end{cases}
\end{equation}
Set $\Sigma = \Omega \cap \mathbb{S}^{d-1}$. If, in addition, $|G(x)|\leq C|x|^{-\beta}$ and $\Vert v(r \cdot)\Vert_{L^2(\Sigma)}\leq C r^{1-\beta}$ for $r\geq R_0$ and some $\beta>0$, then $|v(x)|\leq C|x|^{1-\beta}$ for $|x| \geq 2R_0$, with $C$ depending on $R_0$ and $\Vert h\Vert_{L^\infty(\mathbb{S}^{d-1})}$.
\end{corollary}
\begin{proof}
  Given $r\geq R_0$ consider $v_r(x) = v\Big( \frac{r}{R_0} x\Big)$ and $G_r(x) = G\Big( \frac{r}{R_0} x\Big)$. Since $\Omega$ is a cone and $h$ is homogeneous of degree  $-1$, we can rewrite \eqref{eq part solap} as
\[
    \begin{cases}
        \Delta v_r = 0, \quad \Omega \setminus B_{1},\\
       \partial_\nu v_r+ hv = \frac{r}{R_0}G_r , \quad \partial \Omega\setminus B_{1}.
    \end{cases}
\]
So, thanks to \eqref{eq oblique estimate}, we deduce that 
\begin{eqnarray*}
    \Vert v_r\Vert_{L^\infty(\Omega \cap A')}&\leq& C( \Vert v_r\Vert_{L^2(\Omega \setminus B_1 )}+\frac{r}{R_0}\Vert G_r\Vert_{L^\infty(\partial \Omega \setminus B_1 )})\\
    &\leq& C(R_0)( r^{1-\beta}+r^{1-\beta})\\
    &\leq& C(R_0)r^{1-\beta}.
\end{eqnarray*}
\end{proof}

\end{document}